\documentclass[
jmp
 amsmath,amssymb,
reprint, onecolumn 
]{revtex4-1}

\usepackage{graphicx}
\usepackage{dcolumn}
\usepackage{bm}

\usepackage[utf8]{inputenc}
\usepackage[T1]{fontenc}
\usepackage{mathptmx}
\usepackage{etoolbox}

\usepackage{amsmath,amssymb,comment,amsthm,enumerate}

\makeatletter
\def\@email#1#2{%
 \endgroup
 \patchcmd{\titleblock@produce}
  {\frontmatter@RRAPformat}
  {\frontmatter@RRAPformat{\produce@RRAP{*#1\href{mailto:#2}{#2}}}\frontmatter@RRAPformat}
  {}{}
}%
\makeatother

\newtheorem*{Thm*}{Theorem}
\newtheorem{Thm}{Theorem}[section]
\newtheorem{Cor}[Thm]{Corollary}
\newtheorem{Prop}[Thm]{Proposition}
\newtheorem{Lemma}[Thm]{Lemma}

\theoremstyle{definition}
\newtheorem{Defn}[Thm]{Definition}
\newtheorem{Notation}[Thm]{Notation}

\newtheorem{Remark}[Thm]{Remark}
\newtheorem{Example}[Thm]{Example}
\newtheorem{Construction}[Thm]{Construction}

\newcommand{\mf}[1]{\mathbb{#1}}
\newcommand{\mc}[1]{\mathcal{#1}}
\newcommand{\mb}[1]{\mathbf{#1}}

\DeclareMathOperator{\Part}{\mathcal{P}}
\DeclareMathOperator{\Int}{\mathit{Int}}
\DeclareMathOperator{\tr}{\mathrm{tr}}
\DeclareMathOperator{\Tr}{\mathrm{Tr}}
\DeclareMathOperator{\cyc}{\mathrm{cyc}}

\DeclareMathOperator{\Par}{\mathrm{Par}}
\DeclareMathOperator{\End}{\mathrm{End}}

\renewcommand{\phi}{\varphi}

\newcommand{\Exp}[1]{\mathbb{E} \left[ #1 \right]}
\newcommand{\norm}[1]{\left\Vert#1\right\Vert}
\newcommand{\abs}[1]{\left\vert#1\right\vert}
\newcommand{\chf}[1]{\mathbf{1}_{#1}}
\newcommand{\set}[1]{\left\{#1\right\}}
\newcommand{\ip}[2]{\left \langle #1, #2 \right \rangle}
\newcommand{\Span}[1]{\mathrm{Span} \left( #1 \right)}
\newcommand{\Alg}[1]{\mathrm{Alg} \left( #1 \right)}
\newcommand{\state}[1]{\varphi \left[ #1 \right]}
\newcommand{\Sing}[1]{\mathrm{Sing} \left( #1 \right)}
\newcommand{\Pair}[1]{\mathrm{Pair} \left( #1 \right)}
\newcommand{\W}[1]{\mathnormal{W} \left( #1 \right)}
\newcommand{\T}[1]{T \left( #1 \right)}
\newcommand{\supp}[1]{\mathrm{supp} \left( #1 \right)}
\newcommand{\Wick}[1]{: \!\! #1 \!\! :}
\newcommand{\statez}[1]{\tilde{\varphi}_0 \left[ #1 \right]}

\begin{document}

\preprint{AIP/123-QED}

\title[Hermite trace polynomials]{Hermite trace polynomials and chaos decompositions \\ for the Hermitian Brownian motion}
\author{Michael Anshelevich}
\author{David Buzinski}%
 \email{manshel@tamu.edu, davidbuzinski@gmail.com}
\affiliation{Department of Mathematics, Texas A\&M University, College Station, TX 77843-3368}%

\date{\today}

\begin{abstract}
For a non-zero parameter $q$, we define Hermite trace polynomials, which are multivariate polynomials indexed by permutations. We prove several combinatorial properties for them, such as expansions and product formulas. The linear functional determined by these trace polynomials is a state for $q = \frac{1}{N}$ for $N$ a non-zero integer. For such $q$, Hermite trace polynomials of different degrees are orthogonal. The product formulas extend to the closure with respect to the state. The state can be identified with the expectation induced by the $N \times N$ Hermitian Brownian motion. Hermite trace polynomials are martingales for this Brownian motion, while the elements in the closure can be interpreted as stochastic integrals with respect to it. Using the grading on the algebra, we prove several chaos decompositions for such integrals, as well as analyze the corresponding creation and annihilation operators. In the univariate, pure trace polynomial case, trace Hermite polynomials can be identified with the Hermite polynomials of matrix argument.
\end{abstract}

\maketitle

\section{Introduction}

Let $\set{B(t) : t \geq 0}$ be the Brownian motion. It is well known that there exist polynomials $H_n(x,t)$ (the Hermite polynomials) such that for each $n$, $H_n(B(t), t)$ is a martingale with respect to the filtration induced by $\set{B(t)}$. These polynomials have numerous other familiar properties. The list of properties relevant to this article includes:
\begin{itemize}
\item
\textbf{Orthogonality}: $\mf{E}[H_n(B(t), t) H_k(B(t), t)] = 0$ for $n \neq k$;
\item
\textbf{Three-term recursion}: $x H_n(x,t) = H_{n+1}(x,t) + n t H_{n-1}(x,t)$;
\item
\textbf{Expansion}: $H_n(x,t) = \sum_{j=0}^{[n/2]} (-1)^j \frac{n!}{(n - 2j)! 2^j j!} t^j x^{n - 2 j}$, where the coefficients are related to the number of matchings of $n$ objects;
\item
\textbf{Product formulas}: $H_n(x,t) H_k(x,t) = \sum_{j=0}^{n \wedge k} \binom{n}{j} \binom{k}{j} j! t^j H_{n + k - 2 j}(x,t)$, where the coefficients are related to the number of inhomogeneous matchings of $n$ objects;
\item
\textbf{Solution of the heat equation}: $\partial_t H_n(x,t) + \frac{1}{2} \partial_x^2 H_n(x,t) = 0$ with the initial condition $H_n(x,0) = x^n$;
\item
\textbf{Stochastic integral representation}: $H_n(B(t), t) = \int_{[0,t)^n} \,dB(t_1) \ldots \,dB(t_n)$.
\end{itemize}
Moreover, some of these results, notably the product formulas, extend to more general stochastic integrals
\[
\int_{[0,t)^n} f(t_1, \ldots, t_n) \,dB(t_1) \ldots \,dB(t_n)
\]
for $f \in L^2$.

Now let $\set{X(t) : t \geq 0}$ be the $N \times N$ Hermitian Brownian motion: Hermitian random matrices with jointly complex Gaussian entries and the covariance function $\Exp{X(s)_{ij} X(t)_{k \ell}} = \frac{1}{N} \delta_{i=\ell} \delta_{j=k} (s \wedge t)$.  Fix $N \geq 2$. Then as is also well-known by now, there is no polynomial $P(x,t)$ of degree greater than $2$ such that $P(X(t), t)$ is a martingale. For example, a simple calculation shows that for $s < t$,
\[
\Exp{ X(t)^3 | \leq s} = X(s)^3 + (t-s) \left(2 X(s) + \frac{1}{N} \Tr[X(s)] I_N \right),
\]
where the conditional expectation is taken onto the subalgebra generated by the matrix entries up to time $s$. Thus to have a family closed under conditional expectations, the algebra of polynomials needs to be replaced by a larger algebra of trace polynomials. Here a trace polynomial in some collection of variables $\set{x_i : i \in S}$ is, roughly speaking, a polynomial in these variables as well as in the values of a tracial functional $\tr$ applied to these variables. See \cite{Cebron-Free-convolution} for a formal definition using a universal property. For example, $x^3 + 2 x + \tr[x]$ (as above) and 
\begin{equation}
    \label{Eq:Trace-ply-example}
    x_2 x_4 \tr[x_1 x_3 x_7] \tr[x_5 x_6]
\end{equation}
are trace polynomials. See Section~\ref{Subsec:Trace-background} for details. An incomplete but representative list of related work involving trace polynomials includes the study of random matrices \cite{Cebron-Free-convolution, Driver-Hall-Kemp,Kemp-large-N-GL,Kemp-Heat-kernel}, noncommutative functions \cite{Klep-Spenko-Free-function-theory,Klep-Pascoe-Volcic}, and operator algebras \cite{Dabrowski-Guionnet-Shl-Convex,Jekel-Elementary,Jekel-Li-Shl-Wasserstein}. Other related work includes \cite{Rains,Graczyk-Vostrikova,Levy-Schur-Weyl,Huber-trace}.

In this article, we study the star-algebra of trace polynomials, the state on it induced by the Hermitian Brownian motion, the corresponding basis of Hermite trace polynomials, and the larger algebra obtained as its completion. However, we keep the context somewhat more general.  First, a homogeneous trace monomial is naturally encoded by a permutation. Let $\alpha \in S_0(n)$, that is, a permutation of the set $\set{0, 1, \ldots, n}$. Denote
\[
\Tr_\alpha[x_1, \ldots, x_n]
= \prod_{i \text{ in the cycle starting with $0$}} x_i \prod_{\text{other cycles}} \Tr\left[ \prod_{i \text{ in the cycle}} x_i \right].
\]
Compare with Notation~22.29 in \cite{Nica-Speicher-book}. For example, the trace monomial \eqref{Eq:Trace-ply-example} corresponds to $\alpha = (024)(137)(56)$.

Second, the reciprocal matrix dimension $1/N$ is replaced by a general real parameter $q$. Finally, rather than considering variables indexed by time (or, extending linearly, by $L^2(\mf{R}_+, dx)$), we index them by a general real Hilbert space $\mc{H}_{\mf{R}}$. Thus, for each $n$, $\alpha \in S_0(n)$ and $h_1, \ldots, h_n \in \mc{H}_{\mf{R}}$, we will consider symbols $\T{\alpha \otimes (h_1 \otimes \ldots \otimes h_n)}$ subject to the symmetry relation
\[
\T{\alpha \otimes_s (h_1 \otimes \ldots \otimes h_n)} = \T{(\sigma \alpha \sigma^{-1}) \otimes_s U_\sigma (h_1 \otimes \ldots \otimes h_n)}
\]
for any $\sigma \in S(n)$, where
\[
U_\sigma (h_1 \otimes \ldots \otimes h_n) = h_{\sigma^{-1}(1)} \otimes \ldots \otimes h_{\sigma^{-1}(n)}.
\]
Extending linearly in both arguments, we obtain
\[
\mc{TP}(\mc{H}_{\mf{R}}) = \Span{\set{\T{\eta \otimes_s F} : n \geq 0, \eta \in \mf{C}[S_0(n)], F \in \mc{H}_{\mf{R}}^{\odot n}}},
\]
where for now we are considering the algebraic tensor product of Hilbert spaces. We now define a star-algebra structure on $\mc{TP}(\mc{H}_{\mf{R}})$ by
\[
\T{\alpha \otimes_s F} \T{\beta \otimes_s G} = \T{(\alpha \cup \beta) \otimes_s (F \odot G)}
\]
and
\[
\T{\eta \otimes_s F}^\ast
= \T{\eta^\ast \otimes_s F}.
\]
Here for $\alpha \in S_0(n)$, $\beta \in S_0(k)$, the permutation $\alpha \cup \beta \in S_0(n+k)$ is obtained by shifting the cycles of $\beta$ by $n$, and merging the cycles of $\alpha$ and $\beta$ containing $0$. See Notation~\ref{Notation:Union}.

Next, let $q$ be a non-zero parameter. For a transposition $\tau$, we define the contraction $C_\tau$ as follows. For each $n$ and $\alpha \in S_0(n)$, $C_\tau(\alpha)$ is obtained by: multiplying $\tau \alpha$; erasing the support of $\tau$ and shifting to obtain a permutation in $S_0(n-2)$; and multiplying by a weight depending on $q$ and the number of cycles in $\tau \alpha$. See Definition~\ref{Defn:Contraction-gr}. Then as usual,  for each $n$, we define the Laplacian on $\mf{C}[S_0(n)]$ as the sum over transpositions $\mc{L} = \sum_{\tau \in S(n)}  C_\tau$, and the Hermite trace polynomial as $\W{\eta \otimes_s F} = \T{e^{-\mc{L}}(\eta \otimes_s F)}$. Here, the contraction on the tensor part of the argument is the usual tensor contraction. Hermite trace polynomials satisfy several properties which parallel those of ordinary Hermite polynomials.

We may now define a linear functional $\phi$ on $\mc{TP}(\mc{H}_{\mf{R}})$ by requiring it to be unital and zero on each $\W{\eta \otimes_s F}$. Note that with respect to the spanning set $\set{\T{\eta \otimes_s F}}$, the coefficients for multiplication do not depend on $q$, while those of the linear functional $\phi$ do. On the other hand, if we use $\set{\W{\eta \otimes_s F}}$ as a spanning set, the coefficients for multiplication depend on $q$, while those of the linear functional $\phi$ do not. $\phi$ is positive if and only if $q$ is zero or of the form $\pm \frac{1}{N}$ for $N \in \mf{N}$  \cite{Gnedin-Gorin-Kerov-Block-characters,Kostler-Nica-CLT-S-infty}. For such $q$, we can define the GNS Hilbert space $\mc{F}_q(\mc{H})$ as the completion of the quotient $\mc{TP}_q(\mc{H}_{\mf{R}})$ of $\mc{TP}(\mc{H}_{\mf{R}})$ with respect to the seminorm induced by $\phi$. The inner product takes the form
\[
\ip{(\eta \otimes_s F)}{(\zeta \otimes_s G)}_q = \delta_{n=k} \sum_{\sigma \in S(n)} \chi_q[\eta \sigma \zeta^\ast \sigma^{-1}] \ip{F}{U_\sigma G}_{\mc{H}^{\otimes n}}.
\]
Here $\chi_q$ is the normalized character $\chi_q(\alpha) = q^{\abs{\alpha}}$, where $\abs{\alpha}$ is the standard length function on the symmetric group.

Using the seminorm induced by $\phi$ allows us to extend the star-algebra structure to
\[
\overline{\mc{TP}}(\mc{H}_{\mf{R}}) = \Span{\set{\W{\eta \otimes_s F} : n \geq 0, \eta \in \mf{C}[S_0(n)], F \in \mc{H}_{\mf{R}}^{\otimes n}}},
\]
where $F$ may now lie in the Hilbert space tensor product. In the case $\mc{H} = L^2(\mf{R}_+, dx)$, $\W{\eta \otimes_s F}$ may then be interpreted as a stochastic integral of $F$. We thus obtain contraction, product, and conditional expectation formulas for such integrals. Note that $\T{\eta \otimes_s F}$ is in general not defined for $F \in \mc{H}_{\mf{R}}^{\otimes n}$.

Since $\mc{TP}(\mc{H}_{\mf{R}})$ is naturally graded, we may interpret $\mc{F}_q(\mc{H})$ as a Fock space. There are several natural choices of creation (and annihilation) operators which we describe. The corresponding field operators live in smaller subalgebras of $\mc{TP}(\mc{H}_{\mf{R}})$. There is an interesting connection with a different Fock space construction from \cite{Bozejko-Guta} which we describe in some detail. Incidentally, another article exploring the connection between the characters of the symmetric groups and GUE matrices is \cite{Kostler-Nica-CLT-S-infty}. We have not elucidated the connection between their work and ours.

Finally, we show that, as expected, for $q = \pm \frac{1}{N}$ the structures above are isomorphic to the algebra of trace polynomials, and equivariant square-integrable matrices, in a Hermitian Brownian motion. Various corollaries follow. In particular, we obtain several versions of the chaos decomposition for this process. For $q=0$, with a different scaling, one obtains objects related to free probability.

We thus may interpret Hermite trace polynomials as elements $\W{\alpha \otimes_s (h_1 \otimes \ldots \otimes h_n)}$ in the algebra $\overline{\mc{TP}}(\mc{H}_{\mf{R}})$ with the linear functional $\phi$; or the corresponding trace polynomials in formal variables; or (for $q = 1/N$) the corresponding trace polynomials in Gaussian random matrices $\set{X(h_i)}$. In this last interpretation, we prove that the matrix entries of Hermite trace polynomials are themselves (multivariate) Hermite polynomials in the matrix entries of $\set{X(h_i)}$.

To the best of our knowledge, the Hermite trace polynomials $\W{\eta \otimes_s F}$ are new even in the univariate case $\mc{H} = \mf{C}$. However, if we further restrict to $\eta \in \mf{C}[S(n)]$ (rather than $S_0(n)$), the corresponding objects have appeared in the literature. Indeed, denoting $\chi^\lambda$ the character of the irreducible representation indexed by the partition $\lambda$, $\W{\chi^\lambda}$ is closely related to the corresponding Hermite polynomial of matrix argument. In particular, all of these elements are orthogonal with respect to $\phi$. Moreover, one can form a more general set of characters $\chi^{\lambda, \lambda'}$ indexed by pairs of partitions which differ by one box, such that $\set{\W{\chi^{\lambda, \lambda'}}}$ form an orthogonal basis for $\mc{F}_{1/N}(\mf{C})$. This collection of trace polynomials is clearly deserving of additional study.

The paper is organized as follows. In the background Section~\ref{Sec:Prelim}, we review in particular key properties of the symmetric group, and Gaussian Hilbert spaces. In Section~\ref{Sec:Group} we discuss the kernel of the character $\chi_q$, define the multiplication and contractions on the tensor algebra of symmetric groups, and study their properties. In Section~\ref{Sec:Fock}, we define the Fock space $\mc{F}_q(\mc{H})$, describe the kernel of the inner product, and list three chaos decompositions for this space for different choices of $\mc{H}$. In Section~\ref{Sec:Algebra}, we upgrade the algebra structure from $\mc{TP}(\mc{H}_{\mf{R}})$ to $\overline{\mc{TP}}(\mc{H}_{\mf{R}})$, define the Hermite trace polynomials $\W{\eta \otimes_s F}$ and the functional $\phi$, prove conditional expectation and product formulas, and note that the case $q = - \frac{1}{N}$ is isomorphic to that for $q = \frac{1}{N}$. We also study creation and annihilation decompositions on the Fock space, and three subalgebras which arise: the Gaussian subalgebra, the commutative subalgebra corresponding to pure trace polynomials, and the subalgebra corresponding to pure polynomials, which is not closed under conditional expectations. We finish the section by describing the relation to a construction by Bo\.{z}ejko and Gu\c{t}\u{a}. Finally, in Section~\ref{Sec:GUE}, we give some background on trace polynomials and the Hermitian Brownian motion, prove the isomorphism with the random matrix picture for $q = \frac{1}{N}$, including the description of the matrix entries of Hermite trace polynomials in terms of ordinary Hermite polynomials, and list several corollaries. For $q=0$, we note an isomorphism involving the free Fock space. We also describe the relation with Hermite polynomials of matrix argument.

\section{Preliminaries}
\label{Sec:Prelim}

\subsection{Permutations and partitions}
\label{Subsec:Permutations}
Denote $[n] = \set{1, \ldots, n}$ and $[0, n] = \set{0, 1, \ldots, n}$.

A permutation in $S(n)$ induces a (cycle) set partition in $\Part(n)$. Conversely, a partition $\pi \in \Part(n)$ will be identified with a permutation the elements of whose cycles are ordered in increasing order. In particular, a partition whose blocks are pairs and singletons will be identified with the corresponding involutive permutation. For such a partition, we denote by $\Sing{\pi}$ its single-element blocks, and $\Pair{\pi}$ its two-element blocks.  

Similarly, a set partition $\pi$ induces a (number) partition whose parts are the sizes of blocks of $\pi$. Conversely, a partition $\lambda \in \Par(n; k)$ with parts $\lambda_1 \geq \lambda_2 \geq \ldots \geq \lambda_k$ induces a set partition whose blocks are intervals of size $\lambda_1, \ldots, \lambda_k$. $\Par(n+1;\leq N)$ are the partitions of $n+1$ with at most $N$ parts.

Denote by $S_0(n)$ the permutations of $[0, n]$. $S(n)$ is a subgroup of $S_0(n)$, and so acts on $S_0(n)$ by conjugation. The equivalence classes under this action are subsets of the standard conjugacy classes where the number of elements of the cycle containing $0$ is preserved. So they are in a natural bijection with number partitions of $n+1$ with a marked element. Equivalently, they are indexed by pairs of partitions $(\lambda, \lambda')$, where $\lambda \in \Par(n)$, $\lambda' \in \Par(n + 1)$, and the corresponding diagrams differ by one box.

For $\alpha \in S_0(n)$, denote $\cyc_0(\alpha) = \cyc(\alpha) - 1$, where $\cyc(\alpha)$ is the number of cycles of $\alpha$. In other words, $\cyc_0(\alpha)$ is the number of cycles of $\alpha$ not containing $0$. Denote
\[
\abs{\alpha} = (n+1) - \cyc(\alpha) = n - \cyc_0(\alpha),
\]
which is the usual length function on the symmetric group on $[0, n]$. For a partition $\pi \in \Part_{1,2}(n)$ with only blocks of size $1$ and $2$, we will also denote by $\abs{\pi} = \abs{\Pair{\pi}}$ the distance from the identity for the corresponding permutation (rather than the number of blocks of $\pi$).

For $n < m$, we may sometimes identify the element $\alpha \in S_0(n)$ with the corresponding element of $S_0(m)$ under the natural inclusion $[0,n] \subset [0,m]$ (so that the image of $\alpha$ fixes the elements of $[n+1,m]$). Note that $\abs{\alpha}$ is preserved by this identification.

A \emph{partial permutation} $\alpha \in \mc{PS}_0(n)$ is a bijection from a subset of $[0, n]$ onto a subset of $[0, n]$; proper subsets and the empty subset are allowed. Orbits of a partial permutation fall into two types. Cyclic orbits are cycles in the usual permutation sense. Linear orbits have an initial and a final element. Note that a linear orbit has at least two elements. It is convenient to abuse the terminology and consider elements of $[0, n]$ which do not belong to any orbit of $\alpha$ as single-element linear orbits of $\alpha$. Denote $\mc{PS}_0(n, N)$ the set of partial permutations of $[0, n]$ with $N$ linear orbits.

\subsection{Structure theory of the symmetric group}
\label{Sec:Centralizer}
See \cite{Gill-Rep} for basic representation theory background. For a partition $\lambda \in \Par(n+1)$, denote by $\chi^\lambda$ the character of the irreducible representation of $S_0(n)$ corresponding to $\lambda$. We will identify $\chi^\lambda$ with the element
\[
\sum_{\sigma \in S_0(n)} \chi^\lambda(\sigma) \sigma \in \mf{C}[S_0(n)]
\]
(recall that the characters of the symmetric group are real-valued). These elements span the center $Z(\mf{C}[S_0(n)])$, and are orthogonal, in the sense that for $\lambda \neq \mu$,
\[
\left( \sum_{\sigma \in S_0(n)} \chi^\lambda(\sigma) \sigma \right) \left( \sum_{\tau \in S_0(n)} \chi^\mu(\tau) \tau \right)
= \sum_{\rho \in S_0(n)} (\chi^\lambda \ast \chi^\mu)(\rho) \rho = 0.
\]
In particular, for any character $\chi$, $\chi^\lambda$ and $\chi^\mu$ are orthogonal with respect to the inner product induced by $\chi$.

The centralizer of $\mf{C}[S(n)]$ in $\mf{C}[S_0(n)]$ is
\[
Z(\mf{C}[S_0(n)] : \mf{C}[S(n)]) = \set{\eta \in \mf{C}[S_0(n)] : {\sigma^{-1}} \eta \sigma = \eta \text{ for all } \sigma \in S(n)}.
\]
The following are well-known \cite{Okounkov-Vershik}, \cite{Gill-Rep}.
\begin{itemize}
\item
$Z(\mf{C}[S_0(n)] : \mf{C}[S(n)])$ is generated (as an algebra) by $Z(\mf{C}[S(n)])$ and the Jucys–Murphy element ${(01)} + \ldots + {(0n)}$.
\item
For $\lambda \in \Par(n)$, write $\lambda' = \lambda + \square$ if the Young diagram for $\lambda$ is obtained by removing one of the boxes from the Young diagram for $\lambda'$. Denote $\chi^{\lambda' : \lambda}$ the character of the compression of the $\lambda'$-irreducible representation of $S_0(n)$ to the (unique) component giving a $\lambda$-irreducible representation of $S(n)$. Then $\set{\chi^{\lambda' : \lambda} : \lambda \in \Par(n), \lambda' = \lambda + \square}$ are orthogonal and span $Z(\mf{C}[S_0(n)] : \mf{C}[S(n)])$.
\end{itemize}

Let $\mc{W}$ be the Wedderburn isomorphism
\[
\mc{W} : \sum_{\lambda \in \Par(n+1)} M_{d_\lambda}(\mf{C}) \rightarrow \mf{C}[S_0(n)],
\]
where $d_\lambda$ is the dimension of the irreducible representation of $S_0(n)$ corresponding to $\lambda$. Then $\mc{W}(M_{d_\lambda}(\mf{C}))$ is the ideal generated by (any one of) the Young symmetrizer(s) $c_\lambda$, and is spanned by these symmetrizers (for different choices of the tableau corresponding to $\lambda$). In particular,  $\mc{W}(I_{M_{d_\lambda}(\mf{C})}) = \frac{d_\lambda}{(n+1)!} \chi^\lambda$, which can be characterized as minimal central projections.

Denote
\[
\mf{C}[S_0(n)]_{\leq N}
= \mc{W} \left(\sum_{\lambda \in \Par(n+1; \leq N)} M_{d_\lambda}(\mf{C}) \right)
\]
and
\[
\mf{C}[S_0(n)]_{> N}
= \mc{W} \left(\sum_{\lambda \in \Par(n+1; \geq N+1)} M_{d_\lambda}(\mf{C}) \right).
\]

Let $\chi$ be a character of $S_0(n)$. Then $\chi \circ \mc{W}$ is a trace on $\sum_{\lambda \in \Par(n+1)} M_{d_\lambda}(\mf{C})$, and so has the form
\[
\chi \circ \mc{W} = \sum_{\lambda \in \Par(n+1)} n_{\lambda} \Tr_{M_{d_\lambda}(\mf{C})}.
\]
Here
\[
n_\lambda = \frac{\sum_\alpha \chi[\alpha] \chi^\lambda[\alpha]}{\sum_\alpha \chi^\lambda[\alpha] \chi^\lambda[\alpha]}
= \frac{1}{(n+1)!} \sum_{\alpha \in S_0(n)} \chi[\alpha] \chi^\lambda[\alpha].
\]
Denote $E_{ij}^\lambda$ the matrix units in $M_{d_\lambda}(\mf{C})$. Then for any $\chi$,
\[
\set{\mc{W}(E_{ij}^\lambda): \lambda \in \Par(n+1), 1 \leq i, j \leq d_\lambda}
\]
span $\mf{C}[S_0(n)]$, and are orthogonal with respect to the (typically degenerate) inner product induced by $\chi$,
\[
\chi[\mc{W}(E_{ij}^\lambda) \mc{W}(E_{ij}^\mu)^\ast]
= \delta_{\lambda=\mu} \chi[\mc{W}(E_{ii}^\lambda)]
= \delta_{\lambda=\mu} n_\lambda.
\]
See for example \cite{Structure-symmetric} for a detailed exposition and explicit expressions for $\mc{W}(E_{ij}^\lambda)$.

\subsection{Algebraic conditional expectation}

\begin{Prop}
\label{Prop:Algebraic-CE}
Let $\mc{A}$ be a unital star-algebra, $\mc{B}$ a unital star-subalgebra, and $\phi$ a faithful, tracial state on $\mc{A}$, where positivity means that $\phi[a^\ast a] \geq 0$ for any $a \in \mc{A}$. Denote by $L^2(\mc{A}, \phi)$ and $L^2(\mc{B}, \phi)$ the corresponding GNS Hilbert spaces, with the common state vector $\Omega$. Suppose $F : \mc{A} \rightarrow \mc{B}$ is a function such that the map $a \Omega \mapsto F(a) \Omega$ extends to the orthogonal projection $P : L^2(\mc{A}, \phi) \rightarrow L^2(\mc{B}, \phi)$. Then
\begin{itemize}
\item
$\phi[F(a)] = \phi[a]$.
\item
$F$ is a $\mc{B}$-bimodule map.
\item
For any $a \in \mc{A}$, the operator on $L^2(\mc{B}, \phi)$ induced by $F(a)$ is $P a P$. In particular, the operator induced by $F(a^\ast a)$ is positive.
\end{itemize}
We call such a map $F$ an algebraic conditional expectation.
\end{Prop}

If $\mc{A}$ is a $C^\ast$ algebra, it follows that $F$ is a genuine $\phi$-preserving conditional expectation.

\begin{proof}
By assumption, for any $a \in \mc{A}$ and $b \in \mc{B}$,
\[
\phi[b^\ast a] = \phi[b^\ast F(a)],
\]
and $F(a)$ is uniquely determined by this condition. By taking $b = 1$, we get the first property. Next, for $b' \in \mc{B}$, using the fact that $\phi$ is tracial,
\[
\phi[b^\ast F(a b')] = \phi[b^\ast a b'] = \phi[b' b^\ast a] = \phi[b' b^\ast F(a)] = \phi[b^\ast F(a) b'],
\]
so $F$ is a right $\mc{B}$-module map. The proof for the left action is similar, and does not require the tracial property. Finally,
\[
\ip{b \Omega}{F(a) b' \Omega}_\phi = \ip{b \Omega}{a b' \Omega}_\phi = \ip{b \Omega}{P a P b' \Omega}_\phi
\]
and
\[
\ip{b \Omega}{F(a^\ast a) b \Omega}_\phi = \ip{a b \Omega}{a b \Omega}_\phi \geq 0. \qedhere
\]
\end{proof}

\subsection{Gaussian Hilbert spaces}
\label{Section:Gaussian-Hilbert}

See \cite{Janson-GHS} for proofs and more details.

Let $\mc{H}_{\mf{R}}$ be a real Hilbert space. A Gaussian Hilbert space indexed by $\mc{H}_{\mf{R}}$ is a jointly Gaussian family of random variables $X(\mc{H}_{\mf{R}}) = \set{X(h) : h \in \mc{H}_{\mf{R}}}$ on the same probability space, such that the map $h \mapsto X(h)$ is real linear, each $X(h)$ is centered, and their covariance is $\mf{E}[X(f) X(g)] = \ip{f}{g}$.

For each $n \geq 0$, denote
\[
\mc{P}_n(\mc{H}_{\mf{R}}) = \set{\text{Polynomials in the variables $X(\mc{H}_{\mf{R}})$ of degree $\leq n$}}
\]
and $\mc{P}(\mc{H}_{\mf{R}}) = \bigcup_{n=0}^\infty \mc{P}_n(\mc{H}_{\mf{R}}) = $ all such polynomials. Denote $\overline{\mc{P}}_n(\mc{H}_{\mf{R}})$ the closure of $\mc{P}_n(\mc{H}_{\mf{R}})$ in $L^2(\mc{P}(\mc{H}_{\mf{R}}), \mf{E})$ and $\mc{H}_n = \overline{\mc{P}}_n(\mc{H}_{\mf{R}}) \ominus \overline{\mc{P}}_{n-1}(\mc{H}_{\mf{R}})$, the orthogonal complement taken in $L^2(\mc{P}(\mc{H}_{\mf{R}}), \mf{E})$. Denote $p \mapsto \Wick{p}$ the orthogonal projection from each $\overline{\mc{P}}_n(\mc{H}_{\mf{R}})$ onto $\mc{H}_n$; note that it maps each $\mc{P}_n(\mc{H}_{\mf{R}})$ to itself. Define the Wick product map $W$ from the algebraic symmetric Fock space $\bigoplus_{n=0}^\infty \mc{H}_{\mf{R}}^{\odot_s n}$ into $L^2(\mc{P}(\mc{H}_{\mf{R}}), \mf{E})$ by the linear extension of
\[
\W{f_1 \otimes \ldots \otimes f_n} = \Wick{X(f_1) \ldots X(f_n)}.
\]
Explicitly $W$ is determined by the recursion $\W{\emptyset} = 1$, $\W{f_1} = X(f_1)$,
\[
\W{f_0 \otimes \ldots \otimes f_n} = X(f_0) \W{f_1, \ldots, f_n} - \sum_{i=1}^n \ip{f_0}{f_i} \W{f_1 \ldots \hat{f}_i \ldots f_n}.
\]
We call $\W{f_1 \otimes \ldots \otimes f_n}$ a multivariate Hermite polynomial. It also has an explicit expansion
\begin{equation}
\label{Eq:Wick-Gaussian}
\W{f_1 \otimes \ldots \otimes f_n} = \sum_{\pi \in \mc{P}_{1,2}(n)} (-1)^{\abs{\pi}} \prod_{\set{i,j} \in \Pair{\pi}} \ip{f_i}{f_j} \prod_{k \in \Sing{\pi}} X(f_k).
\end{equation}
$W$ extends to the isomorphism between the symmetric Fock space $\mc{F}_s(\mc{H})$ and $L^2(\mc{P}(\mc{H}_{\mf{R}}), \mf{E})$.

In the complex case, there are two natural types of Gaussian Hilbert spaces.
\begin{itemize}
\item
We may take $\mc{H} = \set{f + i g : f, g \in \mc{H}_{\mf{R}}}$ to be the complexification of $\mc{H}_{\mf{R}}$, with $X(f + i g) = X(f) + i X(g)$. Then each $X(h)$ is centered complex Gaussian, and their covariance is $\mf{E}[X(h) X(k)] = \ip{h}{\bar{k}}_{\mc{H}}$.
\item
Alternatively, we may take a (real) Gaussian Hilbert space $X(\mc{H}_{\mf{R}} \oplus \mc{H}_{\mf{R}})$,  and form $Z(h) = X(h,0) + i X(0, h)$. Then each $Z(h)$ is centered complex Gaussian, and their covariance is $\mf{E}[Z(h) Z(k)] = 0$, $\mf{E}[Z(h) \overline{Z(k)}] = 2 \ip{h}{k}_{\mc{H}_{\mf{R}}}$.
\end{itemize}

If $\set{\xi_i : i \in \Xi}$ is an orthonormal basis for $\mc{H}_{\mf{R}}$, then 
\[
\set{\W{\xi_{u(1)}, \ldots, \xi_{u(n)}} : n \geq 0, \mb{u} \in \Delta(\Xi^n)}
\]
is an orthogonal basis for $L^2(\mc{P}(\mc{H}_{\mf{R}}), \mf{E})$. Here $\Delta(\Xi^n) = \set{\mb{u} \in \Xi^n : u(1) \leq u(2) \leq \ldots \leq u(n)}$. Similarly,
\[
\set{Z(\xi_{u(1)}) \ldots Z(\xi_{u(n)}) : n \geq 0, \mb{u} \in \Delta(\Xi^n)}
\]
is an orthogonal basis for $L^2(\mc{P}(\mc{K}), \mf{E})$. Thus the linear extension of the map $\mc{S} : \W{h_1, \ldots, h_n} \mapsto Z(h_1) \ldots Z(h_n)$ (the Segal-Bargmann transform) is an isomorphism between these spaces.

\section{The tensor algebra of symmetric groups}
\label{Sec:Group}

\subsection{A function on the symmetric group}

On the symmetric group $S_0(n)$, consider the function $\chi^{n+1}_q: \alpha \mapsto q^{\abs{\alpha}}$, and extend it linearly to a function on the group algebra $\mf{C}[S_0(n)]$. As is well-known (see, for example, \cite{Gnedin-Gorin-Kerov-Block-characters,Kostler-Nica-CLT-S-infty}), this function is positive semi-definite for
\[
q \in \mc{Z}_{n+1} = \left[ - \frac{1}{n}, \frac{1}{n} \right] \cup \set{\pm \frac{1}{N} : 1 \leq N \leq n-1}
\]
and is not positive semi-definite for other values of $q$. It follows that these functions are positive definite for all $n$ if
\[
q \in \mc{Z} = \cap_n \mc{Z}_n = \set{0} \cup \set{\pm \frac{1}{N} : N \in \mf{N}}.
\]
We will typically omit the superscript on $\chi^{n+1}_q$. For $q \in \mc{Z}$, the positivity of $\chi_q$ follows from the fact that $\chi_{1/N}$ is the normalized character of the standard representation
\[
\pi_{n, 1/N} : \mf{C}[S_0(n)] \rightarrow \End \left((\mf{C}^N)^{\otimes (n+1)} \right),
\]
while $\chi_0$ is the normalized character of its regular representation (and also the standard trace on $\mf{C}[S_0(n)]$).

It is well-known \cite{Gnedin-Gorin-Kerov-Block-characters} (see also \cite{Biane-Approx-factorization-characters,Kerov-book}) that
\[
\chi_{1/N}^{n+1} = \frac{1}{N^{n+1}} \sum_{\lambda \in \Par(n+1)} \abs{SS_N(\lambda)} \chi^\lambda,
\]
where $SS_N(\lambda)$ is the number of semistandard Young tableaux of shape $\lambda$ with entries belonging to the set $\set{1, \ldots, N}$. In particular, the coefficients are zero if $\lambda$ has more than $N$ parts, and non-zero if it has at most $N$ parts. Also
\[
\chi_0^{n+1} = \frac{1}{(n+1)!} \sum_{\lambda \in \Par(n+1)} d_\lambda \chi^\lambda.
\]

We now discuss the kernel of the normalized character $\chi_q$. See Section~4 of \cite{Procesi} for a related discussion.

\begin{Prop}
\label{Prop:Kernel-rep}
Let $q \in \mc{Z}$. Denote
\[
\mc{N}_{gr, q, n} = \set{\eta \in \mf{C}[S_0(n)] : \chi_q[\eta \eta^\ast] = 0} = \set{\eta \in \mf{C}[S_0(n)] : \pi_{n, q}(\eta) = 0}
\]
and
\[
\mc{N}_{gr, q} = \bigoplus_{n=0}^\infty \mc{N}_{gr, q, n} \subset \bigoplus_{n=0}^\infty \mf{C}[S_0(n)].
\]
\begin{enumerate}
\item
$\mc{N}_{gr, q, n} = \set{0}$ for $q = 0$ or for $q = \frac{1}{N}$, $n+1 \leq N$.
\item
The following are equivalent descriptions of $\mc{N}_{gr, 1/N, n}$ for $n+1 > N$.
\begin{itemize}
\item
$\mc{N}_{gr, 1/N, n}$ is the ideal of the group algebra of $S_0(n)$ spanned by the Young symmetrizers corresponding to diagrams with at least $N + 1$ rows.
\item
$\mc{N}_{gr, 1/N, n} = \mf{C}[S_0(n)]_{> N}$, so that on $\mf{C}[S_0(n)]_{\leq N}$, $\chi_{1/N}$ is faithful.
\item
More explicitly, $\mc{N}_{gr, 1/N, n}$ is the ideal generated by
\[
\sigma_n^{(1/N)} = \sum_{\sigma \in S_0(N)} (-1)^{\abs{\sigma}} \sigma
\]
under the natural embedding of $S_0(N)$ into $S_0(n)$. Similarly, $\mc{N}_{gr, -1/N, n}$ is the ideal generated by
\[
\sigma_n^{(-1/N)} = \sum_{\sigma \in S_0(N)} \sigma.
\]
\item
Let $\pi \in \mc{PS}_0(n)$ be a partial permutation of $[0, n]$ with $(N+1)$ linear orbits. Denote $a_0, \ldots, a_{N}$ and $b_0, \ldots, b_N$ the initial, respectively, final elements of these orbits (recall that if an element does not belong to any actual orbit of $\pi$, we consider it as a single-element linear orbit, in which case $a_i = b_i$). Here $a_0, b_0$ belong to the orbit containing $0$. For a permutation $\sigma \in S_0(N)$, define $\sigma \circ \pi \in S_0(n)$ by
\[
(\sigma \circ \pi)(x) =
\begin{cases}
    a_{\sigma(i)}, & x = b_i, \\
    \pi(x), & \text{otherwise}.
\end{cases}
\]
In other words, we use the cycles of $\sigma$ to concatenate the linear orbits of $\pi$. Then
\[
\mc{N}_{gr, 1/N, n} = \Span{\sum_{\sigma \in S_0(N)} (-1)^{\abs{\sigma}} \sigma \circ \pi : \pi \in \mc{PS}_0(n, N+1)}
\]
and
\[
\mc{N}_{gr, -1/N, n} = \Span{\sum_{\sigma \in S_0(N)} \sigma \circ \pi : \pi \in \mc{PS}_0(n, N+1)}.
\]
\end{itemize}
\end{enumerate}
\end{Prop}

\begin{proof}
(a) and the equivalence between the first three entries in (b) are well-known. For the equivalence between the final two entries, denote by $\bar{\pi}$ the permutation obtained by ``closing'' the orbits of $\pi$; that is, $\bar{\pi}(b_i) = a_i$ and $\bar{\pi}(x) = \pi(x)$ otherwise. Also let $\alpha \in S(n)$ be defined by $\alpha(a_i) = i$ and arbitrarily otherwise; thus, $\alpha$ maps $\set{a_0, \ldots, a_{N}}$ bijectively onto $[0, N]$. Then a calculation shows that $\alpha^{-1} \sigma \alpha \bar{\pi} = \sigma \circ \pi$. Therefore
\[
\sum_{\sigma \in S(N)} (-1)^{\abs{\sigma}} \sigma \circ \pi = \sum_{\sigma \in S(N)} (-1)^{\abs{\sigma}} \alpha^{-1} \sigma \alpha \bar{\pi}
\]
is in the ideal generated by $\sigma_n^{(1/N)}$. The argument for the opposite inclusion is very close to the proof of Theorem~4.5 in \cite{Procesi}, in a somewhat different language. The ideal is spanned by the elements of the form $\alpha^{-1} \sigma_n^{(1/N)} \beta \alpha$  for $\alpha, \beta \in S_0(n)$. Possibly by replacing $\beta$ by its multiple by appropriate transpositions, we may assume that $0, 1, \ldots, N$ lie in different cycles of $\beta$. Denote $a_i = \alpha^{-1}(i)$ for $0 \leq i \leq N$. Then $a_0, \ldots, a_{N}$ lie in different cycles of $\alpha^{-1} \beta \alpha$. So $\alpha^{-1} \beta \alpha = \bar{\pi}$, where linear orbits of $\pi$ are cycles of $\alpha^{-1} \beta \alpha$ with initial elements $a_0, a_1, \ldots, a_{N}$, and cyclic orbits of $\pi$ are the remaining cycles of $\alpha^{-1} \beta \alpha$. It follows that
\[
(\alpha^{-1} \sigma_n^{(1/N)} \alpha) (\alpha^{-1} \beta \alpha) = \sum_{\sigma \in S(N)} (-1)^{\abs{\sigma}} \sigma \circ \pi.
\]
The argument for $q = -1/N$ is similar.
\end{proof}

\subsection{Algebra structure}
\label{Subsec:Algebra}

\begin{Notation}
\label{Notation:Union}
Let $\alpha \in S_0(n)$, $\beta \in S_0(k)$. Define $\sigma_{n, k} \in S(n+k)$ by
\[
\sigma_{n, k}(i) =
\begin{cases}
i + k, & 1 \leq i \leq n, \\
i - n, & n+1 \leq i \leq n+k.
\end{cases}
\]
Thus in word notation, $\sigma_{n,k} = k+1, k+2, \ldots, k + n, 1, 2, \ldots, k$. Note for future reference that $\sigma_{k, n} = \sigma_{n, k}^{-1}$.

Identify $\alpha, \beta$ with elements in $S_0(n+k)$ by letting them act on the first $(n+1)$, respectively $(k+1)$ elements. Define $\alpha \cup \beta \in S_0(n + k)$ by
\begin{equation}
\label{Eq:Defn:cup}
\alpha \cup \beta = \sigma_{n, k}^{-1} \beta \sigma_{n, k} \alpha.
\end{equation}
That is, $\alpha \cup \beta$ is obtained by: combining the cycles of $\alpha$ and $\beta$ containing $0$ into a cycle
\[
(0, \alpha(0), \ldots, \alpha^{-1}(0), \beta(0) + n, \ldots, \beta^{-1}(0) + n),
\]
keeping the remaining cycles of $\alpha$, and letting the remaining cycles of $\beta$ act on the shifted set $\set{n+1, \ldots, n + k}$.
\end{Notation}

We will now define a version of tensor multiplication on $\bigoplus_{n=0}^\infty \mf{C}[S_0(n)]$ and its quotient by $\mc{N}_{gr, q}$. To distinguish this algebra structure from the usual multiplication on the group algebra, we will denote $\alpha$ by $\T{\alpha}$. We will use this identification to talk about $\chi_q$, $\mc{N}_{gr, q}$, etc., as applied to $\T{\eta}$.

\begin{Defn}
\label{Defn:T-multiplication}
Define the multiplication on $\bigoplus_{n=0}^\infty \mf{C}[S_0(n)]$ by the linear extension of
\[
\T{\alpha} \T{\beta} = \T{{\alpha \cup \beta}}.
\]
We use the ordinary adjoint on the group algebra, defined by the anti-linear extension of the relation $\T{\alpha}^\ast = \T{{\alpha^{-1}}}$.
\end{Defn}

\begin{Remark}
The subalgebra $\bigoplus_{n=0}^\infty \mf{C}[S(n)]$ is called the generic tensor algebra in \cite{Raicu-Generic-tensor}.
\end{Remark}

\begin{Prop}
\label{Prop:Multiplication-group}
\

\begin{enumerate}
\item
The multiplication in Definition~\ref{Defn:T-multiplication} is associative.
\item
${(0)}$ is the identity.
\item
$\mc{N}_{gr, q}$ is an ideal for this multiplication. Consequently, for $q \in \mc{Z}$, the multiplication factors through to the quotient
\[
\mc{TP}_q = \left(\bigoplus_{n=0}^\infty \mf{C}[S_0(n)] \right) / \mc{N}_{gr, q} \simeq \bigoplus_{n=0}^\infty \left(\mf{C}[S_0(n)] / \mc{N}_{gr, q, n} \right).
\]
For $q = \frac{1}{N}$,
\[
\mc{TP}_q = \bigoplus_{n=0}^{N-1} \mf{C}[S_0(n)] \oplus \bigoplus_{n=N}^\infty \mf{C}[S_0(n)]_{\leq N}.
\]
\end{enumerate}
\end{Prop}

\begin{proof}
Associativity of multiplication follows from the description of the operation $\cup$ after equation~\eqref{Eq:Defn:cup}.  (b) is immediate. For (c), since $\mc{N}_{gr, q}$ is an ideal for the usual group algebra multiplication, it follows from \eqref{Eq:Defn:cup} that it is also an ideal for the Definition~\ref{Defn:T-multiplication} multiplication. The final expression follows from Proposition~\ref{Prop:Kernel-rep}.
\end{proof}

\begin{Notation}
\label{Notation:Restriction}
For $\alpha \in S_0(n)$ and $S \subset [n]$, the restriction $\alpha|_{S^c}$ of a permutation is the permutation on $[0, n] \setminus S$ defined by $\alpha|_{S^c}(x) = \alpha^m(x)$, where $m = \min \set{k > 0 \ |\ \alpha^k(x) \in S^c}$.
\end{Notation}

\begin{Notation}
For $A, B \subseteq \mf{Z}$, $\abs{A} = \abs{B}$, denote $P^A_B$ the unique order-preserving bijection from $A$ to $B$. By abuse of notation, we will also denote $P^A_B$ the corresponding bijection between the collections of permutations $S(A)$ and $S(B)$ (it should really be $\sigma \mapsto (P^A_B) \sigma (P^A_B)^{-1}$).
\end{Notation}

\begin{Example}
For $\alpha = (13524)$ and $S = \set{2,5}$, $\alpha|_{S^c} = (134)$ and $P^{[5] \setminus S}_{[3]} \alpha|_{S^c} = (1 2 3)$.
\end{Example}

For $\pi \in \mc{P}_{1,2}(n)$, denote $\supp{\pi} = [n] \setminus \Sing{\pi}$.

\begin{Defn}
\label{Defn:Contraction-gr}
Let $q \neq 0$. For a transposition $\tau = (ij) \in S(n)$ and $\alpha \in S_0(n)$, define the $\tau$-contraction by the linear extension of
\[
C_\tau (\alpha)
= q^{\cyc_0((\tau \alpha)|_{\supp{\tau}^c}) - \cyc_0(\tau \alpha) + 1} {P^{[0, n] \setminus \set{i, j}}_{[0, n-2]} (\tau \alpha)|_{\supp{\tau}^c}}
\]
More generally, for $\pi \in \Part_{1,2}(n)$, define the contraction
\[
C_\pi (\alpha)
= q^{\cyc_0((\pi \alpha)|_{\supp{\pi}^c}) - \cyc_0(\pi \alpha) + \ell}
{P^{[0, n] \setminus \supp{\pi}}_{[0, n-2 \ell]} (\pi \alpha)|_{\supp{\pi}^c}},
\]
where $\ell = \abs{\Pair{\pi}} = \abs{\pi}$. Extend $C_\pi$ linearly to $\mf{C}[S_0(n)]$.
\end{Defn}

\begin{Remark}
It is easy to check that for a transposition $\tau = (ij)$,
\[
q^{\cyc_0((\tau \alpha)|_{\supp{\tau}^c}) - \cyc_0(\tau \alpha) + 1} =
\begin{cases}
q^{-1}, & (ij) \text{ is a cycle in } \alpha, \\
1, & (i), (j) \text{ are cycles in } \alpha, \\
1,  & i, j \text{ are consecutive elements} \\
& \text{in the same cycle of $\alpha$ of length at least } 3, \\
q, & \text{otherwise.}
\end{cases}
\]
In particular, $C_\tau$ is defined for $q=0$ unless $(ij)$ is a cycle in $\alpha$. See Section~\ref{Subsec:q=0}.
\end{Remark}

\begin{Lemma}
\label{Lemma:Contraction-kernel-gr}
Let $q \in \mc{Z} \setminus \set{0}$, and $\tau \in S(n)$ a transposition. Each $C_\tau$ maps $\mc{N}_{gr, q, n}$ to $\mc{N}_{gr, q, n-2}$. 
\end{Lemma}

\begin{proof}
We consider the case $q = \frac{1}{N}$; for $q = - \frac{1}{N}$, the argument is similar. We will use the representation in Proposition~\ref{Prop:Kernel-rep}. Let $\eta \in \mc{N}_{gr, q, n}$. Since $\mc{N}_{gr, q, n}$ is an ideal, $\tau \eta \in \mc{N}_{gr, q, n}$. So it suffices to show that if $\pi \in \mc{PS}_0(n)$ and $\eta = \sum_{\sigma \in S_0(N)} (-1)^{\abs{\sigma}} \sigma \circ \pi$, then for any $S \subset [0, n]$,
\[
\sum_{\sigma \in S_0(N)} q^{\cyc_0((\sigma \circ \pi)|_{S^c}) - \cyc_0(\sigma \circ \pi)} (-1)^{\abs{\sigma}} (\sigma \circ \pi)|_{S^c} \in \mc{N}_{gr, q, n - \abs{S}}.
\]
By induction, it suffices to take $S$ to be a singleton, $S = \set{s}$. We consider two cases.

Suppose $s = a_i = b_i$. Then $(\sigma \circ \pi)|_{\set{s}^c} = (\sigma|_{\set{i}^c}) \circ (\pi|_{\set{s}^c})$ and $\cyc_0((\sigma \circ \pi)|_{\set{s}^c}) = \cyc_0(\sigma|_{\set{i}^c})$. Each $\sigma' \in S_0(N-1)$ appears as $\sigma|_{\set{i}^c}$ $(N+1)$ times, once when $(i)$ is a cycle in $\sigma$ (so that $\sigma'$ has one less cycle than $\sigma$), and $N$ times corresponding to $\sigma(i) = j$ for each $j \in [0,N] \setminus \set{i}$ (so that $\sigma'$ has the same number of cycles as $\sigma$). Thus
\[
\sum_{\sigma \in S_0(N)} q^{\cyc_0((\sigma \circ \pi)|_{S^c}) - \cyc_0(\sigma \circ \pi)} (-1)^{\abs{\sigma}} (\sigma \circ \pi)|_{S^c}
= \sum_{\sigma' \in S_0(N-1)} (N - q^{-1}) (-1)^{\abs{\sigma'}} \sigma' \circ (\pi|_{\set{s}^c})
= 0.
\]
If $\set{s}$ is not a single-element linear orbit of $\pi$, then $(\sigma \circ \pi)|_{\set{s}^c} = \sigma \circ (\pi|_{\set{s}^c})$, and has the same number of cycles as $\sigma \circ \pi$.
\end{proof}

\begin{Notation}
\label{Notation:Laplacian}
By the preceding lemma, the operator
\[
\mc{L}_n = \sum_{\tau \text{ a transposition in } S(n)} C_\tau
\]
is well defined as a map from $\mf{C}[S_0(n)] / \mc{N}_{gr, q, n}$ to $\mf{C}[S_0(n-2)] / \mc{N}_{gr, q, n-2}$. Denote by $\mc{L}$ the direct sum of these operators, which we may consider as a map from $\mc{TP}_q$ to itself. Note that for $\pi \in \Part_{1,2}(n)$, $C_\pi$ is a product of several transposition-type contractions, and
\[
\mc{L}^\ell_n = \ell! \sum_{\substack{\pi \in \Part_{1,2}(n) \\ \abs{\pi} = \ell}} C_\pi.
\]
Denote also
\[
\Part_{1,2}(n, k) = \set{\pi \in \Part_{1,2}(n+k) : \text{ if } (ij) \in \pi, i < j, \text{ then } i \leq n < j}
\]
the inhomogeneous partitions, and
\[
\mc{L}_{n, k} = \sum_{\substack{(ij) \in S(n+k) \\ i \leq n < j}} C_{(ij)}.
\]
Finally, for $\ell \leq n \wedge k$, denote
\[
\mc{L}_{n, k}^{(\ell)} = \ell! \sum_{\substack{\pi \in \Part_{1,2}(n,k) \\ \abs{\pi} = \ell}} C_\pi.
\]
\end{Notation}

\begin{Defn}
\label{Defn:I-gr}
Let $\eta \in \bigoplus_{n=0}^\infty \mf{C}[S_0(n)]$. By analogy with equation~\ref{Eq:Wick-Gaussian} from Section~\ref{Section:Gaussian-Hilbert}, define the Wick product
\[
\W{\eta}
= \T{e^{- \mc{L}} \eta}
= \sum_{\ell=0}^\infty (-1)^\ell \frac{1}{\ell!} \T{\mc{L}^\ell (\eta)}
= \sum_{\pi \in \Part_{1,2}(n)} (-1)^{\abs{\pi}} \T{C_\pi (\eta)}.
\]
For $q \in \mc{Z} \setminus \set{0}$, $\W{\eta}$ is also well-defined for $\eta \in \mc{TP}_q$.
\end{Defn}

\begin{Prop}
\label{Prop:T-I-gr}
\

\begin{enumerate}
\item
\[
\T{\eta}
= \W{e^{\mc{L}} \eta}
= \sum_{\ell=0}^\infty \frac{1}{\ell!} \W{\mc{L}^\ell (\eta)}
= \sum_{\pi \in \Part_{1,2}(n)} \W{C_\pi (\eta)}.
\]
\item
For $\alpha \in S_0(n)$, $\beta \in S_0(k)$,
\[
\begin{split}
\W{\alpha} \W{\beta}
& = \W{{\alpha \cup \beta}} + \sum_{\ell=1}^{\min(n, k)} \frac{1}{\ell!} \W{\mc{L}_{n,k}^{(\ell)} (\alpha \cup \beta)} \\
& = \sum_{\pi \in \Part_{1,2}(n, k)} \W{C_\pi({\alpha \cup \beta})}.
\end{split}
\]
\end{enumerate}
\end{Prop}

\begin{proof}
(a) follows by composing (terminating) power series in $\mc{L}$. For (b), the argument is very similar to the standard one for Hermite polynomials \cite{dSCViennot}. We expand
\[
\begin{split}
\W{\alpha} \W{\beta}
& = \sum_{{\sigma_1} \in \Part_{1,2}(n)} (-1)^{\abs{{\sigma_1}}} \T{C_{\sigma_1} (\alpha)} \sum_{{\sigma_2} \in \Part_{1,2}(k)} (-1)^{\abs{{\sigma_2}}} \T{C_{\sigma_2} (\beta)} \\
& = \sum_{{\sigma_1} \in \Part_{1,2}(n)} \sum_{{\sigma_2} \in \Part_{1,2}(k)} (-1)^{\abs{{\sigma_1}} + \abs{{\sigma_2}}} \T{C_{\sigma_1} (\alpha) \cup C_{\sigma_2} (\beta)} \\
& = \sum_{{\sigma_1} \in \Part_{1,2}(n)} \sum_{{\sigma_2} \in \Part_{1,2}(k)} (-1)^{\abs{{\sigma_1}} + \abs{{\sigma_2}}} \sum_{\tau \in \Part_{1,2}([\abs{\Sing{{\sigma_1}}} + \abs{\Sing{{\sigma_2}}}])} \W{C_\tau (C_{\sigma_1} (\alpha) \cup C_{\sigma_2} (\beta))} \\
& = \sum_{\pi \in \Part_{1,2}(n+k)} \W{C_\pi({\alpha \cup \beta})} \sum_{\substack{S_1 \subset \Pair{\pi|_{[n]}} \\ S_2 \subset \Pair{\pi|_{[n+1,n+k]}}}} (-1)^{\abs{S_1} + \abs{S_2}}.
\end{split}
\]
The final sum is zero unless both $\Pair{\pi|_{[n]}} = \emptyset = \Pair{\pi|_{[n+1,n+k]}}$, in which case $\pi \in \Part_{1,2}(n,k)$.
\end{proof}

\section{Fock space}
\label{Sec:Fock}

\begin{Construction}
Let $\mc{H}_{\mf{R}}$ be a real Hilbert space, and $\mc{H}$ its complexification. We will denote by $\mc{H}^{\odot n}$ the algebraic tensor product, which is spanned by simple tensors, and by $\mc{H}^{\otimes n}$ the Hilbert space tensor product. For each $n$, form the algebraic Fock space
\[
\mf{C} {(0)} \oplus \bigoplus_{n=1}^\infty \left(\mf{C}[S_0(n)] \otimes \mc{H}^{\otimes n} \right).
\]
On each component of this space, we have a natural action of $S(n)$: for $\sigma \in S(n)$
\[
\alpha \otimes F \mapsto {\sigma \alpha \sigma^{-1}} \otimes U_\sigma F,
\]
where
\[
U_\sigma (h_1 \otimes \ldots \otimes h_n) = h_{\sigma^{-1}(1)} \otimes \ldots \otimes h_{\sigma^{-1}(n)}
\]
extends to $\mc{H}^{\otimes n}$ as an isometry. We denote by
\[
\overline{\mc{TP}}(\mc{H}) = \mf{C} {(0)} \oplus \bigoplus_{n=1}^\infty \left(\mf{C}[S_0(n)] \otimes_s \mc{H}^{\otimes n} \right)
\]
the direct sum of vector space quotients under these actions, which may be identified with (the direct sum of) the fixed point subspace(s), the image of the direct sum of the projections
\begin{equation}
\label{Eq:Symm-proj}
P_n : \alpha \otimes F \mapsto \frac{1}{n!} \sum_{\sigma \in S(n)} {\sigma \alpha \sigma^{-1}} \otimes U_\sigma F
\end{equation}
Thus
\begin{equation}
\label{Eq:Quotient}
\mf{C}[S_0(n)] \otimes_s \mc{H}^{\otimes n} = \set{\sum_{\alpha \in S_0(n)} \alpha \otimes F_\alpha \in \mf{C}[S_0(n)] \otimes \mc{H}^{\otimes n} \ :\ \frac{1}{n!} \sum_{\sigma \in S(n)} U_\sigma F_{\sigma^{-1} \alpha \sigma} = F_\alpha}.
\end{equation}
To simplify notation, we will denote
\[
\alpha \otimes_s F = P_n (\alpha \otimes F) = \frac{1}{n!} \sum_{\sigma \in S(n)} {\sigma \alpha \sigma^{-1}} \otimes U_\sigma F.
\]

On $\mf{C} {(0)} \oplus \bigoplus_{n=1}^\infty \left(\mf{C}[S_0(n)] \otimes \mc{H}^{\otimes n} \right)$, we have the canonical (``free'') inner product
\begin{equation}
\label{Eq:Zero-inner-product}
\ip{\alpha \otimes F}{\beta \otimes G}_f
= \delta_{n=k} \chi_0[\alpha \beta^{-1}] \ip{F}{G}_{\mc{H}^{\otimes n}}
= \delta_{n=k} \delta_{\alpha = \beta} \ip{F}{G}_{\mc{H}^{\otimes n}},
\end{equation}
where $\chi_0$ is the canonical trace on each $\mf{C}[S_0(n)]$. On this space, define the operator
\begin{equation}
\label{Eq:K}
\mc{K}_q \left( \alpha \otimes F \right) = n! \sum_{\beta \in S_0(n)} \chi_q[\alpha \beta^{-1}] \beta \otimes F,
\end{equation}
for $\alpha \in S_0(n)$ and $F \in \mc{H}^{\otimes n}$. Note that $\mc{K}_0$ is a scalar multiple of the identity operator. It is easy to check that for $q \in \mc{Z}$, $\mc{K}_q$ is a positive operator. Moreover,
\[
\begin{split}
\mc{K}_q P_n (\alpha \otimes F)
& = \sum_{\beta \in S_0(n)} \sum_{\sigma \in S(n)} \chi_q[\sigma \alpha \sigma^{-1} \beta^{-1}] \beta \otimes U_\sigma F \\
& = \sum_{\beta \in S_0(n)} \sum_{\sigma \in S(n)} \chi_q[\alpha \beta^{-1}] {\sigma \beta \sigma^{-1}} \otimes U_\sigma F \\
& = P_n \mc{K}_q (\alpha \otimes F),
\end{split}
\]
so $\mc{K}_q$ restricts to the subspace $\overline{\mc{TP}}(\mc{H})$. We thus define a (degenerate) inner product on $\mf{C} {(0)} \oplus \bigoplus_{n=1}^\infty \left(\mf{C}[S_0(n)] \otimes \mc{H}^{\otimes n} \right)$ by
\begin{equation}
\label{Eq:Unsymmetrized-IP}
\begin{split}
\ip{\alpha \otimes F}{\beta \otimes G}_q
& = \ip{\alpha \otimes F}{\mc{K}_q P(\beta \otimes G)}_f \\
& = \delta_{n=k} \sum_{\sigma \in S(n)} \chi_q(\alpha \sigma \beta^{-1} \sigma^{-1}) \ip{F}{U_\sigma G}_{\mc{H}^{\otimes n}}
\end{split}
\end{equation}
for $F \in \mc{H}^{\otimes n}$ and $G \in \mc{H}^{\otimes k}$. The map $P_n$ is an orthogonal projection for this inner product, but (because of the degeneracy) it is also an isometry. The resulting inner product on $\overline{\mc{TP}}(\mc{H})$ takes a simpler form
\begin{equation}
\label{Eq:IP-centralizer}
\begin{split}
\ip{\sum_{\alpha \in S_0(n)} \alpha \otimes F_\alpha}{\sum_{\beta \in S_0(n)} \beta \otimes G_\beta}_q
& = n! \sum_{\alpha, \beta \in S_0(n)} \chi_q(\alpha \beta^{-1}) \ip{F_{\alpha}}{G_{\beta}},
\end{split}
\end{equation}
where we recall that elements of $\overline{\mc{TP}}(\mc{H})$ satisfy \eqref{Eq:Quotient}. The inner product \eqref{Eq:IP-centralizer} is positive semi-definite on $\overline{\mc{TP}}(\mc{H})$ for $q \in \mc{Z}$. It is not in general positive definite. However, for $q=0$,
\begin{equation}
\label{Eq:q=0}
\ip{\sum_{\alpha \in S_0(n)} \alpha \otimes F_\alpha}{\sum_{\beta \in S_0(n)} \beta \otimes G_\beta}_0
= n! \sum_{\alpha \in S_0(n)} \ip{F_{\alpha}}{G_{\alpha}}.
\end{equation}
and so the $0$-inner product on $\overline{\mc{TP}}(\mc{H})$ is positive definite.
\end{Construction}

\begin{Remark}
In \cite{Guta-Maassen-BM}, Gu\c{t}\u{a} and Maassen considered a general Fock space construction. Let $\mc{H}$ and $\set{V_n : n \in \mf{N}}$ be Hilbert spaces, such that $S(n)$ acts unitarily on $V_n$. Then one can define $V_n \otimes_s \mc{H}^{\otimes n}$ as the fixed point subspace of the action
\[
v \otimes F \mapsto (\sigma \cdot v) \otimes U_\sigma F,
\]
and a symmetrized Fock space as the orthogonal sum of these subspaces. In our case, $V_n = \mf{C}[S_0(n)]$, with the inner product induced by $\chi_q$, on which $S(n)$ acts by conjugation (which preserves $\chi_q$). One can then define the creation operator based on a sequence of maps $j_n : V_n \rightarrow V_{n+1}$ which commute with the action of $S(n)$; the annihilation operator as its adjoint; and study the algebra generated by field operators. In our setting, there are several natural equivariant choices of the map $j_n$. One possibility is the standard embedding $\mf{C}[S_0(n)] \hookrightarrow \mf{C}[S_0(n+1)]$. Another possibility is the linear extension of the map $\alpha \mapsto (0 \ n+1) \alpha$. See Sections~\ref{Subsec:Gaussian} and \ref{Subsec:Polynomial-aubalgebra}. The algebra considered throughout most of the article is much larger than the subalgebras generated by these field operators; in fact, the vacuum vector is cyclic for it.
\end{Remark}

\begin{Prop}
\label{Prop:kernel-vs}
For $q \in \mc{Z}$, the kernel of the inner product is
\[
\mc{N}_{vs, q}
= \set{\xi \in \overline{\mc{TP}}(\mc{H}) : \ip{\xi}{\xi}_q = 0}
= \Span{\eta \otimes_s F : \eta \in \mc{N}_{gr, q}, F \in \mc{F}_{f}(\mc{H})},
\]
where $\mc{F}_{f}(\mc{H})$ is the full Fock space of $\mc{H}$. Here $\mc{N}_{gr, q}$ is the kernel of the inner product induced by $\chi_q$ on the (direct sum of) group algebra(s) $\bigoplus_{n=0}^\infty \mf{C}[S_0(n)]$, while $\mc{N}_{vs, q}$ is the kernel of the $q$-inner product on the vector space $\overline{\mc{TP}}(\mc{H})$.
\end{Prop}

\begin{proof}
The kernel of $\mc{K}_q$ as an operator on $\mf{C}[S_0(n)] \otimes \mc{H}^{\otimes n}$ with the (tensor) inner product \eqref{Eq:Zero-inner-product} is clearly $\mc{N}_{gr, q, n} \otimes \mc{H}^{\otimes n}$. The kernel of the inner product on $\overline{\mc{TP}}(\mc{H})$ is the intersection of the kernel of $\mc{K}_q$ and $\overline{\mc{TP}}(\mc{H})$.
\end{proof}

\begin{Defn}
\label{Defn:Fock-space}
For $q \in \mc{Z}$, denote
\[
\overline{\mc{TP}}_q(\mc{H}) = \overline{\mc{TP}}(\mc{H}) /\mc{N}_{vs, q}.
\]
In particular,
\[
\overline{\mc{TP}}_{1/N}(\mc{H})
= \bigoplus_{n=0}^{N-1} \left( \mf{C}[S_0(n)] \otimes_s \overline{\mc{H}^{\otimes n}} \right) \oplus \bigoplus_{n=N}^\infty \left( \mf{C}[S_0(n)]_{\leq N} \otimes_s \overline{\mc{H}^{\otimes n}} \right).
\]
Denote by $\mc{F}_q(\mc{H})$ the completion of $\overline{\mc{TP}}_q(\mc{H})$ with respect to the inner product $\ip{\cdot}{\cdot}_q$.

Note that $\overline{\mc{TP}}_q(\mf{C})$ is not $\mc{TP}_q$ from Proposition~\ref{Prop:Multiplication-group}, but its symmetrized version,
\[
\overline{\mc{TP}}_q(\mf{C}) = \bigoplus_{n=0}^\infty Z(\mf{C}[S_0(n)] : \mf{C}[S(n)])/\mc{N}_{gr, q}.
\]
\end{Defn}

\begin{Lemma}
\label{Lemma:L2-approximation}
\

\begin{enumerate}
\item
If $F_i, F \in \mc{H}^{\otimes n}$ and $\norm{F_i - F} \rightarrow 0$, then $\norm{(\alpha \otimes F_i) - (\alpha \otimes F)}_q^2 \rightarrow 0$.
\item
If $\mc{H}$ is infinite-dimensional, the linear span of the elements of the form $\alpha \otimes (h_1 \otimes \ldots \otimes h_n)$ for mutually orthogonal $h_1, \ldots, h_n$ is dense in $\mc{F}_q(\mc{H})$.
\end{enumerate}
\end{Lemma}

\begin{proof}
For part (a),  note that
\[
\norm{(\alpha \otimes F)}_q^2
= \sum_{\sigma \in S(n)} \chi_q(\alpha \sigma \alpha^{-1} \sigma^{-1}) \ip{F}{U_\sigma F}
\leq n! \norm{F}^2.
\]
Part (b) follows from the fact that for infinite-dimensional $\mc{H}$, the linear span of the elements of the form $h_1 \otimes \ldots \otimes h_n$ for mutually orthogonal $h_1, \ldots, h_n$ is dense in $\mc{H}^{\otimes n}$ (for $\mc{H} = L^2([0,1], dx)$, this is a well-known fact in stochastic analysis).
\end{proof}

\begin{Remark}
Recall that we denote by ${\mc{H}^{\odot n}}$ the algebraic tensor product. Then we could equally well consider the Fock space
\[
\mf{C} {(0)} \oplus \bigoplus_{n=1}^\infty \left(\mf{C}[S_0(n)] \otimes {\mc{H}^{\odot n}} \right),
\]
its symmetrized subspace
\[
{\mc{TP}}(\mc{H}) = \mf{C} {(0)} \oplus \bigoplus_{n=1}^\infty \left(\mf{C}[S_0(n)] \otimes_s {\mc{H}^{\odot n}} \right),
\]
and its quotient ${\mc{TP}}_q(\mc{H})$ by the kernel of $\mc{K}_q$ for $q \in \mc{Z}$. By the preceding lemma, $\mc{TP}(\mc{H})$ is dense in $\overline{\mc{TP}}(\mc{H})$ (with respect to the seminorm), and $\mc{TP}_q(\mc{H})$ is dense in $\overline{\mc{TP}}_q(\mc{H})$. In particular, $\mc{F}_q(\mc{H})$ is the completion of either set.
\end{Remark}
\begin{Notation}
Let  $\pi = (I_1, \ldots, I_k) \in \Int(n)$ be an interval partition of $n$ elements. Denote $S(I_j)$ the permutations of the set $I_j$. We can then identify $S(I_1) \times \ldots \times S(I_k)$ with a subgroup of $S(n)$, and so also of $S_0(n)$. Its centralizer, which we denote
\[
\begin{split}
Z(\mf{C}[S_0(n)] : \pi)
& = Z(\mf{C}[S_0(n)] : S(I_1) \times \ldots \times S(I_k)),
\end{split}
\]
consists of all elements of $S_0(n)$ commuting with this subgroup.
\end{Notation}

\begin{Thm}[Chaos decomposition I]
Let $\set{\xi_i : i \in \Xi}$ be an orthonormal basis for $\mc{H}$, where $\Xi = [d]$ or $\Xi = \mf{N}$. Recall the notation
\[
\Delta(\Xi^n) = \set{\mb{u} \in \Xi^n : u(1) \leq u(2) \leq \ldots \leq u(n)}.
\]
For $\mb{u} \in \Xi^n$, denote $\ker(\mb{u})$ the interval partition $\pi = (I_1, \ldots, I_k) \in \Int(n)$ such that $u(i) = u(j) \Leftrightarrow i \stackrel{\pi}{\sim} j$, and $Z(\mf{C}[S_0(n)] : \pi)$ the centralizer of the corresponding subgroup.
\begin{enumerate}
\item
We have a decomposition
\begin{equation}
\label{Eq:Chaos-I}
\overline{\mc{TP}}(\mc{H})
= \bigoplus_{n = 0}^\infty \bigoplus_{\mb{u} \in \Delta(\Xi^n)} Z(\mf{C}[S_0(n)] : \ker(\mb{u})) \otimes_s (\xi_{u(1)} \otimes \ldots \otimes \xi_{u(n)}).
\end{equation}
If $q \in \mc{Z}$, this decomposition is orthogonal with respect to any $q$-inner product.
\item
Any $A \in \mc{F}_{1/N}(\mc{H})$ has a unique decomposition
\[
A = \sum_{n=0}^\infty \sum_{\mb{u} \in \Delta(\Xi^n)} \eta_{\mb{u}} \otimes_s \xi_{\mb{u}},
\]
where $\eta_{\mb{u}} \in Z(\mf{C}[S_0(n)] : \ker(\mb{u}))$ for $n < N$ and $\eta_{\mb{u}} \in Z(\mf{C}[S_0(n)] : \ker(\mb{u})) \cap \mf{C}[S_0(n)]_{\leq N}$ for $n \geq N$, and
\begin{equation}
    \label{Eq:q-norm-decomposition}
\ip{A}{A}_q = \sum_{n=0}^\infty \sum_{\mb{u} \in \Delta(\Xi^n)} \ker(\mb{u})! \chi_q[\eta_{\mb{u}}^\ast \eta_{\mb{u}}] < \infty,
\end{equation}
where as usual $\pi! = \prod_{V \in \pi} \abs{V}!$.
\item
For each $n$, for sufficiently large $N$, $\mf{C}[S_0(n)] \otimes_s \mc{H}^{\otimes n}$ is complete with respect to the $\frac{1}{N}$-norm.
\end{enumerate}
\end{Thm}

\begin{proof}
The span of the vectors of the form $\eta \otimes (\xi_{u(1)} \otimes \ldots \otimes \xi_{u(n)})$ is dense in the left-hand side of equation~\eqref{Eq:Chaos-I}. Using invariance~\eqref{Eq:Quotient}, we first see that we may take $u(1) \leq u(2) \leq \ldots \leq u(n)$. Choosing $\pi = \ker(\mb{u})$, we further see that $\xi_{\mb{u}}$ is invariant under the action of $S(I_1) \times \ldots \times S(I_k)$. So we may take $\eta$ to be invariant under the corresponding action.

For orthogonality, we observe that if $\mb{u}, \mb{v} \in \Delta(\Xi^n)$ and $\eta, \zeta \in Z(\mf{C}[S_0(n)] : \ker(\mb{u}))$,
\[
\begin{split}
\ip{\eta \otimes \xi_{\mb{u}}}{\zeta \otimes \xi_{\mb{v}}}_q
& = \sum_{\sigma \in S(n)} \chi_q[\eta \sigma \zeta^\ast {\sigma^{-1}}] \ip{\xi_{u(1)} \otimes \ldots \otimes \xi_{u(n)}}{\xi_{v(\sigma^{-1}(1))} \otimes \ldots \otimes \xi_{v(\sigma^{-1}(n))}} \\
& = \delta_{\mb{u} = \mb{v}} \sum_{\sigma \in S(I_1) \times \ldots \times S(I_k)} \chi_q[\eta \sigma \zeta^\ast {\sigma^{-1}}] \\
& = \ker(\mb{u})! \chi_q[\eta \zeta^\ast].
\end{split}
\]
Here the second equality holds because for any $\sigma$ exchanging elements of different $I_i$ and $I_j$, the inner product $\ip{\xi_{\mb{u}}}{U_\sigma \xi_{\mb{v}}}$ is zero. The third equality follows from the choice that $\zeta$ is in the centralizer. 

Part (b)  for $\overline{\mc{TP}}_{1/N}(\mc{H})$ follows from part (a) and Definition~\ref{Defn:Fock-space}.  To extend the decomposition to $\mc{F}_{1/N}(\mc{H})$, we note that each subspace $Z(\mf{C}[S_0(n)] : \ker(\mb{u})) \otimes (\xi_{u(1)} \otimes \ldots \otimes \xi_{u(n)})$ is finite dimensional, and thus closed, and the norm-squared \eqref{Eq:q-norm-decomposition} is the sum of norms-squared on these subspaces.

 For (c), note that
\[
\begin{split}
\frac{1}{n!} \norm{\sum_{\alpha \in S_0(n)} \alpha \otimes_s F_\alpha}_q^2
& = \sum_{\alpha, \beta \in S_0(n)} \chi_q[\alpha \beta^{-1}] \ip{F_\alpha}{F_\beta} \\
& \geq \sum_{\alpha \in S_0(n)} \norm{F_\alpha}^2 - \abs{q} \sum_{\alpha \neq \beta \in S_0(n)} \norm{F_\alpha} \norm{F_\beta} \\
& \geq (1 - \abs{q} (n+1)!) \sum_{\alpha \in S_0(n)} \norm{F_\alpha}^2 + \frac{1}{2} \abs{q} \sum_{\alpha, \beta \in S_0(n)} \left( \norm{F_\alpha} - \norm{F_\beta} \right)^2,
\end{split}
\]
and so for $\abs{q} < \frac{1}{(n+1)!}$, the norm $\norm{\sum_{\alpha \in S_0(n)} \alpha \otimes_s F_\alpha}_q$ is equivalent to $\sqrt{\sum_{\alpha \in S_0(n)} \norm{F_\alpha}^2}$.
\end{proof}

In two special cases, we have alternative chaos decompositions. The first one follows from the comments in Section~\ref{Sec:Centralizer}.

\begin{Prop}[Chaos decomposition II]
\label{Prop:Chaos-II}
Let $\mc{H} = \mf{C}$, so that $\mc{TP}(\mf{C}) = Z(\mf{C}[S_0(n)] : \mf{C}[S(n)])$. Then
\[
\set{\chi^{\lambda' : \lambda} : \lambda \in \Par(n), \lambda' = \lambda + \square}
\]
are orthogonal and span $\mc{TP}(\mf{C})$. Moreover,
\[
\set{\chi^{\lambda' : \lambda} : \lambda' = \lambda + \square, \lambda' \in \Par(n+1; \leq N)}
\]
is an orthogonal basis for $\mc{F}_{1/N}(\mf{C})$.
\end{Prop}

\begin{Notation}
In the case $\mc{H}_{\mf{R}} = L^2(\mf{R}_+, dx)$, denote
\[
\Delta(\mf{R}^n_+) = \set{(t_1, \ldots, t_n) \in \mf{R}^n : t_1 \leq t_2 \leq \ldots \leq t_n}.
\]
\end{Notation}

\begin{Prop}[Chaos decomposition III]
\label{Chaos:III}
Let $\mc{H} = L^2(\mf{R}_+, dx)$.
\begin{enumerate}
\item
We have a decomposition
\[
\overline{\mc{TP}}(\mc{H}) = \bigoplus_{n=0}^\infty \bigoplus_{\lambda \in \Par(n+1)} \bigoplus_{i, j = 1}^{d_\lambda} \mc{W}(E_{ij}^\lambda) \otimes_s L^2(\Delta(\mf{R}_+^n), dx^{\otimes n})
\]
which is orthogonal with respect to any $q$-inner product. Here we use the notation from Section~\ref{Sec:Centralizer}.
\item
Any $A \in \mc{F}_{1/N}(\mc{H})$ has a unique decomposition
\[
A = \sum_{n=0}^\infty \sum_{\lambda \in \Par(n+1; \leq N)} \sum_{i, j = 1}^{d_\lambda} \mc{W}(E_{ij}^\lambda) \otimes F_{i j}^\lambda,
\]
where $F_{ij}^\lambda \in L^2(\Delta(\mf{R}_+^n), dx^{\otimes n})$ and
\[
\sum_{n=0}^\infty \sum_{\lambda \in \Par(n+1; \leq N)} n_\lambda \sum_{i, j = 1}^{d_\lambda} \norm{F_{i j}^\lambda}^2 < \infty
\]
for $n_\lambda = \frac{\abs{SS_N(\lambda)}}{N^{n+1}}$. For $A \in \mc{F}_0(\mc{H})$ the same decomposition holds with no restrictions on $\lambda$ and $n_\lambda = \frac{d_\lambda}{(n+1)!}$. However, in that case we also have the simpler isometry~\eqref{Eq:q=0}.
\end{enumerate}
\end{Prop}

\begin{proof}
Every element in the $n$'th component of $\overline{\mc{TP}}(\mc{H})$ is equivalent to a unique element of the form
\[
\sum_{\lambda \in \Par(n+1)} \sum_{i, j = 1}^{d_\lambda} \mc{W}(E_{ij}^\lambda) \otimes_s F_{i j}^\lambda
\]
for some $F_{i j}^\lambda \in L^2(\Delta(\mf{R}_+^n), dx^{\otimes n})$. For $F, G \in L^2(\Delta(\mf{R}_+^n), dx^{\otimes n})$,
\[
\begin{split}
\ip{\mc{W}(E_{ij}^\lambda) \otimes_s F}{\mc{W}(E_{k \ell}^\mu) \otimes_s G}
& = \sum_{\sigma \in S(n)} \chi_q[\mc{W}(E_{ij}^\lambda) \sigma \mc{W}(E_{k \ell}^\mu)^\ast \sigma^{-1}] \ip{F}{U_\sigma G} \\
& = \chi_q[\mc{W}(E_{ij}^\lambda) \mc{W}(E_{k \ell}^\mu)^\ast] \ip{F}{G} \\
& = \delta_{i = k} \delta_{j = \ell} \delta_{\lambda = \mu} n_\lambda \ip{F}{G}. \qedhere
\end{split}
\]
\end{proof}

\section{The operator algebra}
\label{Sec:Algebra}

We now define a star-algebra structure on $\mc{TP}(\mc{H}_{\mf{R}})$, and eventually on $\overline{\mc{TP}}(\mc{H}_{\mf{R}})$. To distinguish the elements of the algebra from the corresponding elements of the inner product space, we will denote the algebra element $\W{\alpha \otimes_s F}$ to emphasize its interpretation as a Wick product. Note that this identification differs from the one in Section~\ref{Subsec:Algebra}. In Remark~\ref{Remark:GNS} we will interpret the Fock space vector $\alpha \otimes_s F$ as the evaluation of the operator $\W{\alpha \otimes_s F}$ on the vacuum vector in the appropriate GNS representation.

\begin{Defn}
For a transposition $\tau = (ij) \in S(n)$, define the $\tau$-contraction on $\mc{H}_{\mf{R}}^{\odot n}$ by the linear extension of
\[
C_\tau(h_1 \otimes \ldots \otimes h_n) = \ip{h_i}{h_j} h_1 \otimes \ldots \otimes \hat{h}_i \otimes \ldots \otimes \hat{h}_j \otimes \ldots \otimes h_n,
\]
More generally, for $\pi \in \Part_{1,2}(n)$,
\[
\pi = \set{(v_1, w_1), \ldots, (v_\ell, w_\ell), (u_1) \ldots (u_{n - 2 \ell})},
\]
with $u_1 < u_2 < \ldots < u_{n - 2 \ell}$, define the contraction
\[
C_\pi(h_1 \otimes \ldots \otimes h_n) = \prod_{i=1}^\ell \ip{h_{v(i)}}{h_{w(i)}} h_{u_1} \otimes \ldots \otimes h_{u_{n - 2 \ell}}
\]
For $q \neq 0$, denote $C_\tau(\alpha \otimes F) = C_\tau(\alpha) \otimes C_\tau(F)$ following Definition~\ref{Defn:Contraction-gr}.
\end{Defn}

\begin{Remark}
The (tensor) contraction is not a bounded operator and so does not extend to the Hilbert space tensor product $\mc{H}_{\mf{R}}^{\otimes n}$. But if $\pi \in \Part_{1,2}(n,k)$ with $\Pair{\pi} = \ell$, it is easy to check that
\[
C_\pi : \mc{H}_{\mf{R}}^{\otimes n} \times \mc{H}_{\mf{R}}^{\otimes k} \rightarrow \mc{H}_{\mf{R}}^{\otimes (n+k-2 \ell)}
\]
is a contraction. Indeed, we have the following. 
\end{Remark}

\begin{Lemma}
Let $F \in \mc{H}_{\mf{R}}^{\otimes n}$ and $G \in \mc{H}_{\mf{R}}^{\otimes k}$. Then
\[
\norm{C_\pi(F \otimes G)} \leq \norm{F} \ \norm{G}.
\]
\end{Lemma}

\begin{proof}
We will prove the result for the transposition $\pi = (1, n+1)$; the general result is similar with more indices. Let $\set{\xi_i : i \in \Xi}$ be an orthonormal basis for $\mc{H}_{\mf{R}}$, so that $\set{\xi_\mb{u} = \otimes_{k=1}^n \xi_u(k) : \mb{u} \in \Xi^n}$ is an orthonormal basis for $\mc{H}_{\mf{R}}^{\otimes n}$. Let $F = \sum_{i, \mb{u}} f_{i, \mb{u}} \xi_i \otimes \xi_{\mb{u}}$ and $G = \sum_{j, \mb{v}} g_{j, \mb{v}} \xi_j \otimes \xi_{\mb{v}}$. Then by the Cauchy-Schwarz inequality,
\[
\begin{split}
\norm{C_\pi(F \otimes G)}^2 
& = \norm{\sum_{i, \mb{u}, \mb{v}} f_{i, \mb{u}} g_{i, \mb{v}}\xi_{\mb{u}} \otimes \xi_{\mb{v}}}^2 \\
& = \sum_{\mb{u}, \mb{v}} \left( \sum_{i} f_{i, \mb{u}} g_{i, \mb{v}} \right)^2
\leq \sum_{\mb{u}, \mb{v}} \sum_i f_{i, \mb{u}}^2 \sum_j g_{j, \mb{v}}^2
= \norm{F}^2 \ \norm{G}^2.
\end{split}
\]
\end{proof}

Next we note that contraction is compatible with the action of the symmetric group.

\begin{Lemma}
\label{Lemma:Conjugation}
Let $\alpha \in S_0(n)$, $\sigma \in S(n)$, $F \in \mc{H}_{\mf{R}}^{\odot n}$ and $\pi \in \mc{P}_{1,2}(n)$, with $\abs{\Pair{\pi}} = \ell$. Then
\[
C_\pi(\sigma \alpha \sigma^{-1} \otimes U_\sigma F) = \tilde{\sigma} C_{\tilde{\pi}}(\alpha) \tilde{\sigma}^{-1} \otimes U_{\tilde{\sigma}} C_{\tilde{\pi}} (F),
\]
where $\tilde{\pi} = \sigma^{-1} \pi \sigma$ and
\[
\tilde{\sigma} = P^{[0, n] \setminus \supp{\tilde{\pi}}}_{[0, n-2 \ell]} \sigma|_{[0,n] \setminus \supp{\tilde{\pi}}} \in S(n - 2 \ell).
\]
\end{Lemma}

\begin{proof}
The relation
\[
C_\pi(U_\sigma F) =  U_{\tilde{\sigma}} C_{\sigma^{-1} \pi \sigma} F
\]
is not hard to check. For the second relation, we compute
\[
\begin{split}
C_\pi(\sigma \alpha \sigma^{-1})
& = q^{\cyc_0((\pi \sigma \alpha \sigma^{-1})|_{\supp{\pi}^c}) - \cyc_0(\pi \sigma \alpha \sigma^{-1}) + \ell}
P^{[0, n] \setminus \supp{\pi}}_{[0, n-2 \ell]} (\pi \sigma \alpha \sigma^{-1})|_{\supp{\pi}^c} \\
& = q^{\cyc_0((\sigma \tilde{\pi} \alpha \sigma^{-1})|_{\supp{\pi}^c}) - \cyc_0(\sigma \tilde{\pi} \alpha \sigma^{-1}) + \ell}
P^{[0, n] \setminus \supp{\pi}}_{[0, n-2 \ell]} (\sigma \tilde{\pi} \alpha \sigma^{-1})|_{\supp{\pi}^c} \\
& = q^{\cyc_0((\tilde{\pi} \alpha)|_{\supp{\tilde{\pi}}^c}) - \cyc_0(\tilde{\pi} \alpha) + \ell}
\tilde{\sigma} P^{[0, n] \setminus \supp{\tilde{\pi}}}_{[0, n-2 \ell]} (\tilde{\pi} \alpha)|_{\supp{\tilde{\pi}}^c} \tilde{\sigma}^{-1} \\
& = \tilde{\sigma} C_{\tilde{\pi}}(\alpha) \tilde{\sigma}^{-1}.
\end{split}
\]
\end{proof}

\begin{Lemma}
We keep the notation $\mc{L}_n$, $\mc{L}$, $\mc{L}_{n, k}$, $\mc{L}_{n, k}^{(\ell)}$ as in Notation~\ref{Notation:Laplacian}. Then $\mc{L}$ descends to a map on $\mc{TP}(\mc{H}_{\mf{R}})$, and $\mc{L}_{n, k}^{(\ell)}$ to a map
\[
\left( \mf{C}[S_0(n)] \otimes_s \mc{H}_{\mf{R}}^{\odot n} \right) \times \left( \mf{C}[S_0(k)] \otimes_s \mc{H}_{\mf{R}}^{\odot k} \right) \rightarrow \mf{C}[S_0(n+k- 2\ell)] \otimes_s \mc{H}_{\mf{R}}^{\odot n+k- 2\ell}.
\]
\end{Lemma}

\begin{proof}
$\mc{L}$ is invariant under the action of $S(n)$, and $\mc{L}_{n,k}$ under the action of $S(n) \times S(k)$.
\end{proof}

\begin{Defn}
\label{Defn:T-I}
For $\eta \otimes F \in \mf{C}[S_0(n)] \otimes \mc{H}_{\mf{R}}^{\odot n}$, define
\[
\T{\eta \otimes F} = \W{e^{\mc{L}} (\eta \otimes F)} = \sum_{k=0}^\infty \frac{1}{k!} \W{\mc{L}^k (\eta \otimes F)} = \sum_{\pi \in \Part_{1,2}(n)} \W{C_\pi (\eta \otimes F)}.
\]
Then $T$ is also well-defined on $\mc{TP}(\mc{H}_{\mf{R}})$. Note that we cannot in general extend it to $\overline{\mc{TP}}(\mc{H}_{\mf{R}})$.
\end{Defn}

The following result follows immediately from Proposition~\ref{Prop:T-I-gr}.

\begin{Prop}
\label{Prop:T-I}
For $\alpha \in S_0(n)$ and $F \in \mc{H}_{\mf{R}}^{\odot n}$,
\[
\begin{split}
\W{\alpha \otimes F} = \T{e^{- \mc{L}} (\alpha \otimes F)} & = \sum_{\ell=0}^\infty (-1)^\ell \frac{1}{\ell!} T \left( \mc{L}^\ell (\alpha \otimes F) \right) \\
& = \sum_{\pi \in \Part_{1,2}(n)} (-1)^{\abs{\pi}} \T{C_\pi (\alpha \otimes F)}.
\end{split}
\]
\end{Prop}

\begin{Defn}
\label{Defn:Product-tensors}
Define the star-algebra structure on $\mc{TP}(\mc{H}_{\mf{R}})$ by
\[
\T{\alpha \otimes_s F} \T{\beta \otimes_s G} = \T{(\alpha \cup \beta) \otimes_s (F \otimes G)}
\]
and
\[
\T{\alpha \otimes_s (h_1 \otimes \ldots \otimes h_n)}^\ast
= \T{{\alpha^{-1}} \otimes_s ({h}_1 \otimes \ldots \otimes {h}_n)}.
\]
Here $\alpha \in S_0(n)$ and $\beta \in S_0(k)$, for any $n, k \geq 0$.
\end{Defn}

\begin{Prop}
The multiplication on $\mc{TP}(\mc{H}_{\mf{R}})$ is well defined, and
\[
[\T{\alpha \otimes_s F} \T{\beta \otimes_s G}]^\ast = \T{\beta \otimes_s G}^\ast \T{\alpha \otimes_s F}^\ast.
\]
\end{Prop}

\begin{proof}
Let $\alpha \in S_0(n)$,  $\beta \in S_0(k)$, $F \in \mc{H}_{\mf{R}}^{\odot n}$, $G \in \mc{H}_{\mf{R}}^{\odot k}$, $\sigma \in S(n)$, $\tau \in S(k)$. Embed $\alpha, \beta$ in $S_0(n+k)$ and $\sigma, \tau$ in $S(n+k)$ as in Notation~\ref{Notation:Union}, and denote $\rho = \sigma \sigma_{n,k}^{-1} \tau \sigma_{n,k} \in S(n+k)$. Then
\[
\begin{split}
& \T{{\sigma \alpha \sigma^{-1}} \otimes U_\sigma F} \T{{\tau \beta \tau^{-1}} \otimes U_\tau G} \\
&\quad = \T{{\sigma_{n,k}^{-1} \tau \beta \tau^{-1} \sigma_{n,k} \sigma \alpha \sigma^{-1}} \otimes (U_\sigma F \otimes U_\tau G)} \\
&\quad = \T{{(\sigma_{n,k}^{-1} \tau \sigma_{n,k}) (\sigma_{n,k}^{-1} \beta \sigma_{n,k}) (\sigma_{n,k}^{-1} \tau^{-1} \sigma_{n,k}) \sigma \alpha \sigma^{-1}} \otimes (U_\sigma F \otimes U_\tau G)} \\
&\quad = \T{{ \sigma (\sigma_{n,k}^{-1} \tau \sigma_{n,k}) (\sigma_{n,k}^{-1} \beta \sigma_{n,k}) \alpha (\sigma_{n,k}^{-1} \tau^{-1} \sigma_{n,k}) \sigma^{-1}} \otimes (U_\sigma F \otimes U_\tau G)} \\
&\quad = \T{{\rho(\alpha \cup \beta) \rho^{-1}} \otimes U_\rho(F \otimes G)}.
\end{split}
\]
Similarly,
\[
\begin{split}
& \T{{\beta^{-1}} \otimes_s ({g}_1 \otimes \ldots \otimes {g}_k)} \T{{\alpha^{-1}} \otimes_s ({f}_1 \otimes \ldots \otimes {f}_n)} \\
&\quad = \T{{\sigma_{n, k} \alpha^{-1} \sigma_{n, k}^{-1} \beta^{-1}} \otimes_s ({g}_1 \otimes \ldots \otimes {g}_k \otimes {f}_1 \otimes \ldots \otimes {f}_n)} \\
&\quad = \T{{\sigma_{n, k} (\alpha \cup \beta)^{-1} \sigma_{n, k}^{-1}} \otimes_s ({g}_1 \otimes \ldots \otimes {g}_k \otimes {f}_1 \otimes \ldots \otimes {f}_n)} \\
&\quad = \T{{(\alpha \cup \beta)^{-1}} \otimes_s ({f}_1 \otimes \ldots \otimes {f}_n \otimes {g}_1 \otimes \ldots \otimes {g}_k)}. \qedhere
\end{split}
\]
\end{proof}

\begin{Prop}
\label{Prop:Contraction-kernel-vs}
Let $q \in \mc{Z} \setminus \set{0}$.
\begin{itemize}
    \item $\mc{N}_{vs, q}$ is an ideal for multiplication in Definition~\ref{Defn:Product-tensors}.
    \item For $\pi \in \mc{P}_{1,2}(n)$ with $\abs{\Pair{\pi}} = \ell$, $C_\pi$ maps $\mc{N}_{vs, q, n}$ to $\mc{N}_{vs, q, n - 2 \ell}$. 

\end{itemize} 
Consequently, $\mc{L}$ is defined as a map on $\mc{TP}_q(\mc{H}_{\mf{R}})$, and $\mc{L}_{n,k}^{(\ell)}$ on the appropriate quotient. Also, $\T{\eta}$ is well-defined for $\eta \in \mc{TP}_q(\mc{H}_{\mf{R}})$.
\end{Prop}

\begin{proof}
Apply Lemma~\ref{Lemma:Contraction-kernel-gr} and Propositions~\ref{Prop:Multiplication-group} and \ref{Prop:kernel-vs}.
\end{proof}

The following result also follows from Proposition~\ref{Prop:T-I-gr}.

\begin{Prop}
\label{Prop:Product-I}
Let $\alpha \in S_0(n)$, $\beta \in S_0(k)$, $F \in \mc{H}_{\mf{R}}^{\odot n}$, $G \in \mc{H}_{\mf{R}}^{\odot k}$. Then
\[
\begin{split}
\W{\alpha \otimes_s F} \W{\beta \otimes_s G}
& = \W{{\alpha \cup \beta} \otimes_s (F \otimes G)} + \sum_{\ell =1}^{\min(n, k)} \frac{1}{\ell!} \W{\mc{L}_{n,k}^{(\ell)} ((\alpha \cup \beta) \otimes_s (F \otimes G))} \\
& = \sum_{\pi \in \Part_{1,2}(n, k)} \W{C_\pi({\alpha \cup \beta} \otimes_s (F \otimes G))}.
\end{split}
\]
\end{Prop}

We obtain an extension of the familiar factorization property of Hermite polynomials. Compare also with Proposition~\ref{Prop:Linearization}.

\begin{Cor}
\label{Cor:Product}
If the components of $\set{F_j: 1 \leq j \leq k}$ are mutually orthogonal,
\[
\prod_{j=1}^k \W{\alpha_j \otimes_s F_j} = \W{(\alpha_1 \cup \ldots \cup \alpha_k) \otimes_s (F_1 \otimes \ldots \otimes F_k)}.
\]
\end{Cor}
\begin{proof}
    Use the preceding proposition recursively, and note that a contraction which pairs mutually orthogonal vectors produces a zero term. 
\end{proof}
The next proposition is the analog of the property that ordinary Hermite polynomial satisfies the differential equation
\[
x H_n' - 2 H_n'' = n H_n.
\]
Note that the Hermite polynomials in Section~\ref{Section:Gaussian-Hilbert} have non-standard normalization.
\begin{Prop}

\label{Prop:DE}
\

\begin{enumerate}
\item
By abuse of notation, define $\mc{L} \ \T{\eta \otimes F} = \T{\mc{L}(\eta \otimes F)}$. Then also $\mc{L} \ \W{\eta \otimes F} = \W{\mc{L}(\eta \otimes F)}$.
\item
Define the Euler operator on $\mc{TP}(\mc{H}_{\mf{R}})$ by
\[
E \T{\eta \otimes F} = n \T{\eta \otimes F}
\]
for $\eta \in \mf{C}[S_0(n)]$ and $F \in \mc{H}_{\mf{R}}^{\odot n}$. Then for such $\eta$,
\[
(E - 2 \mc{L}) \W{\eta \otimes F} = n \W{\eta \otimes F},
\]
and it is the unique eigenfunction of this operator with eigenvalue $n$ and leading term $\T{\eta \otimes F}$. In particular, it follows that $E$ maps $\mc{TP}_q(\mc{H}_{\mf{R}})$ to itself.
\end{enumerate}
\end{Prop}

\begin{proof}
Part (a) follows from the expansion in Proposition~\ref{Prop:T-I}. For part (b), we first note that
\[
\begin{split}
& E \W{\eta \otimes F} - 2 \W{\mc{L}(\eta \otimes F)} \\
&\qquad = E \sum_{k=0}^n \frac{(-1)^k}{k!} \T{\mc{L}^k(\eta \otimes F)} - 2 \sum_{k=0}^n \frac{(-1)^k}{k!} \T{\mc{L}^{k+1}(\eta \otimes F)} \\
&\qquad = \sum_{k=0}^n \frac{(-1)^k (n - 2k)}{k!} \T{\mc{L}^k(\eta \otimes F)} + 2 \sum_{k=1}^n \frac{(-1)^k k}{k!} \T{\mc{L}^{k}(\eta \otimes F)} \\
&\qquad = \sum_{k=0}^n \frac{(-1)^k n}{k!} \T{\mc{L}^k(\eta \otimes F)} \\
&\qquad = n \W{{\eta} \otimes F}.
\end{split}
\]
In particular, anything of lower degree is in the sum of eigenspaces with eigenvalues $0, 1, \ldots, n-1$. It follows that specifying the leading term of an eigenfunction with a given eigenvalue determines it.
\end{proof}

\begin{Notation}
If $C_\pi(\eta \otimes F)$ is a scalar multiple of ${(0)}$ (which is the identity for the algebra), we will identify it with a scalar.
\end{Notation}

\begin{Thm}
\label{Thm:State}
Define a unital linear functional $\phi$ on $\bigoplus_{n=0}^\infty \mf{C}[S_0(n)] \otimes \mc{H}_{\mf{R}}^{\odot n}$ by
\[
\state{\W{{(0)}}} = 1, \quad \state{\W{\alpha \otimes F}} = 0
\]
for any $\alpha \in \mf{C}[S_0(n)]$, $F \in \mc{H}_{\mf{R}}^{\odot n}$, and $n \geq 1$. 
Then
\begin{enumerate}
\item
\[
\state{\W{\beta \otimes_s G}^\ast \W{\alpha \otimes_s F}} = \ip{(\alpha \otimes_s F)}{(\beta \otimes_s G)}_q.
\]
In particular,  $\phi$ is well-defined on the quotient $\mc{TP}_q(\mc{H}_{\mf{R}})$.
\item
$\varphi$ is tracial.
\item
$\varphi$ is positive for $q \in \mc{Z} \setminus \set{0}$.
\item
For $q \in \mc{Z} \setminus \set{0}$, $\set{A : \state{A^\ast A} = 0} = \Span{\W{\zeta} : \zeta \in \mc{N}_{vs, q}}$, and $\phi$ is faithful on $\mc{TP}_q(\mc{H}_{\mf{R}})$.
\item
For $\alpha \in S_0(2n)$,
\begin{equation}
\label{Eq:GUE-moment}
\state{\T{\alpha \otimes F}}
= \sum_{\pi \in \Part_2(2n)} q^{n - \cyc_0(\pi \alpha)} C_\pi(F)
= \sum_{\pi \in \Part_2(2n)} q^{\abs{\pi \alpha} - n} C_\pi(F)
\end{equation}
and it is zero if $\alpha \in S_0(2n+1)$.
\end{enumerate}
\end{Thm}

\begin{proof}
For $\alpha \in S_0(n)$, $\beta \in S_0(k)$, $\state{\W{\alpha \otimes F} \W{\beta \otimes G}} = 0$ if $n \neq k$. If $n=k$, using Proposition~\ref{Prop:Product-I}, the definition of $\varphi$, and Definition~\ref{Defn:Contraction-gr},
\[
\begin{split}
\state{\W{{\beta} \otimes_s G}^\ast \W{\alpha \otimes_s F}}
& = \sum_{\pi \in \Part_2(n, n)} \W{C_\pi((\beta^{-1} \cup \alpha) \otimes_s ({G} \otimes F))} \\
& = \sum_{\pi \in \Part_2(n, n)} q^{n - \cyc_0(\pi (\beta^{-1} \cup \alpha))} C_\pi({G} \otimes F) .
\end{split}
\]
The map $\sigma \mapsto \sigma^{-1} \sigma_{n,n} \sigma$ maps $S(n)$ bijectively onto $\Part_2(n, n)$. Clearly
\[
C_{\sigma^{-1} \sigma_{n,n} \sigma}({G} \otimes F) = \ip{G}{U_{\sigma^{-1}} F} = \ip{F}{U_\sigma G}.
\]
Moreover, each cycle of
\[
\sigma^{-1} \sigma_{n,n} \sigma (\beta^{-1} \cup \alpha)
= \sigma^{-1} \sigma_{n,n} \sigma \sigma_{n, n} \alpha \sigma_{n,n} \beta^{-1}
\]
intersects $[0, n]$. A calculation shows that its restriction to $[0,n]$ is $\sigma^{-1} \alpha \sigma \beta^{-1}$. Therefore
\[
q^{n - \cyc_0(\sigma^{-1} \sigma_{n,n} \sigma (\beta^{-1} \cup \alpha))}
= q^{n - \cyc_0(\sigma^{-1} \alpha \sigma \beta^{-1})}
= \chi_q(\sigma^{-1} \alpha \sigma \beta^{-1})
\]
It follows that
\[
\state{\W{{\beta} \otimes_s F}^\ast \W{\alpha \otimes_s G}}
= \sum_{\sigma \in S(n)} \chi_q(\sigma^{-1} \alpha \sigma \beta^{-1}) \ip{F}{U_\sigma G}_{\mc{H}_{\mf{R}}^{\otimes n}}
= \ip{(\alpha \otimes_s F)}{(\beta \otimes_s G)}_q.
\]
Since $\chi_q[\beta^{-1} \alpha] = \chi_q[\alpha^{-1} \beta]$, (b) follows from (a), as do (c) and (d). For (e), using the expansion in Proposition~\ref{Prop:T-I},
\[
\state{\T{\alpha \otimes F}}
= \sum_{\pi \in \Part_2(2n)} \W{C_\pi (\alpha \otimes F)}
= \sum_{\pi \in \Part_2(2n)} q^{n - \cyc_0(\pi \alpha)} C_\pi(F). \qedhere
\]
\end{proof}

\begin{Prop}
\label{Prop:Extended-algebra}
Let $\alpha \in S_0(n)$, $\beta \in S_0(k)$, $F \in \mc{H}_{\mf{R}}^{\odot n}$, $G \in \mc{H}_{\mf{R}}^{\odot k}$. Then
\[
\norm{\W{\alpha \otimes_s F} \W{\beta \otimes_s G}}_\varphi \leq (n + k)! \min[(k+1)^n, (n+1)^k] \norm{F} \ \norm{G}.
\]
Consequently, for $q \in \mc{Z} \setminus \set{0}$, the star-algebra structure extends to $\overline{\mc{TP}}(\mc{H}_{\mf{R}})$ and $\overline{\mc{TP}}_q(\mc{H}_{\mf{R}})$.
\end{Prop}

\begin{proof}
Combining Proposition~\ref{Prop:Product-I} with the estimate in Lemma~\ref{Lemma:L2-approximation},
\[
\norm{\W{\alpha \otimes_s F} \W{\beta \otimes_s G}}_\varphi \leq (n + k)! \abs{\Part_{1,2}(n, k)} \norm{F} \ \norm{G}.
\]
$\abs{\Part_{1,2}(n, k)}$ is sequence A086885 in OEIS. Choosing the locations of the left and right endpoints of $\ell$ pairs, and $\ell!$ ways to pair them, gives  the identity
\[
\abs{\Part_{1,2}(n, k)} = \sum_{\ell=0}^{\min(n,k)} \ell! \binom{n}{\ell} \binom{k}{\ell}.
\]
More roughly, noting that each of the first $n$ elements may be paired with one of the final $k$ elements, or none, 
\[
\abs{\Part_{1,2}(n, k)} \leq (k+1)^n.
\]
The stated estimate follows from symmetry.
\end{proof}

The proof of the following proposition is very similar to Proposition~\ref{Prop:Product-I}, and is omitted.

\begin{Prop}
\label{Prop:Linearization}
For $\alpha_i \in S_0(n_i)$, and $F_i \in \mc{H}_{\mf{R}}^{\otimes n_i}$,
\[
\state{\W{{\alpha_1} \otimes_s F_1} \ldots \W{{\alpha_k} \otimes_s F_k}}
= \sum_{\pi \in \Part_2(n_1, \ldots, n_k)} \W{C_\pi ((\alpha_1 \cup \ldots \cup \alpha_k) \otimes_s (F_1 \otimes \ldots \otimes F_k))}.
\]
Here $\Part_2(n_1, \ldots, n_k)$ are the inhomogeneous pair partitions \cite{dSCViennot}. That is, denote by $I(n_1, \ldots, n_k)$ the interval partition of $n_1 + \ldots + n_k$ with blocks of sizes $n_1, \ldots, n_k$. Then $\pi \in \Part_2(n_1, \ldots, n_k)$ if $\pi \in \Part_2(n_1 + \ldots + n_k)$ and the elements of each pair in $\pi$ belong to different blocks of $I(n_1, \ldots, n_k)$.
\end{Prop}

\begin{Remark}
\label{Remark:GNS}
Let $q \in \mc{Z} \setminus \set{0}$,. By Theorem~\ref{Thm:State}(a,b,c), the GNS Hilbert space $L^2(\mc{TP}(\mc{H}_\mf{R}), \phi)$ is isomorphic to $\mc{F}_q(\mc{H})$, and  have the (right) star-representation of $\mc{TP}(\mc{H}_\mf{R})$ on this Hilbert space, for which the state vector ${(0)}$ is cyclic, and $\W{\alpha \otimes_s F} (0) = \alpha \otimes_s F \in \mc{F}_q(\mc{H})$. For the corresponding representation of $\mc{TP}_q(\mc{H}_\mf{R})$, it is cyclic and separating. It follows from Proposition~\ref{Prop:Extended-algebra} that this representation extends to $\overline{\mc{TP}}_q(\mc{H}_\mf{R})$. Similarly, we have the left representation, which commutes with the right one, and for which $(0)$ is also cyclic.
\end{Remark}

\begin{Prop}
\label{Prop:CE}
Let $\mc{H}'_{\mf{R}} \subset \mc{H}_{\mf{R}}$ be a closed subspace, let $\mc{H}' = \mf{C} \mc{H}_{\mf{R}}'$, and let $P_{\mc{H}'}$ be the orthogonal projection onto $\mc{H}'$. Fix $q \in \mc{Z} \setminus \set{0}$.
\begin{enumerate}
\item
The map defined by $\mc{F}(P_{\mc{H}'}) (\alpha \otimes F) = (\alpha \otimes (P_{\mc{H}'}^{\otimes n} F))$ extends to the orthogonal projection $\mc{F}(P_{\mc{H}'}) : \mc{F}_q(\mc{H}) \rightarrow \mc{F}_q(\mc{H}')$.
\item
The map $\Gamma(P_{\mc{H}'}) : \overline{\mc{TP}}_q(\mc{H}_{\mf{R}}) \rightarrow \overline{\mc{TP}}_q(\mc{H}'_{\mf{R}})$ obtained by the linear extension of
\[
\Gamma(P_{\mc{H}'})(\W{\alpha \otimes F}) = \W{\alpha \otimes (P_{\mc{H}'}^{\otimes n} F)}
\]
is an algebraic conditional expectation, which we will denote by $\state{\cdot \ |\ \mc{H}'}$. In the GNS representation on $\mc{F}_q(\mc{H}')$, it is implemented by
\[
\Gamma(P_{\mc{H}'})(\W{\alpha \otimes F}) = \mc{F}(P_{\mc{H}'}) \W{\alpha \otimes F} \mc{F}(P_{\mc{H}'}).
\]
\end{enumerate}
\end{Prop}

\begin{proof}
We first verify that for $F \in \mc{H}^{\otimes n}$ and $G \in (\mc{H}')^{\otimes k}$,
\[
\begin{split}
\ip{\alpha \otimes F}{\beta \otimes G}_q
& = \delta_{n=k} \sum_{\sigma \in S(n)} \chi_q(\alpha \sigma \beta \sigma^{-1}) \ip{F}{U_\sigma G} \\
& = \delta_{n=k} \sum_{\sigma \in S(n)} \chi_q(\alpha \sigma \beta \sigma^{-1}) \ip{P_{\mc{H}'}^{\otimes n} F}{U_\sigma G} \\
& = \ip{\alpha \otimes (P_{\mc{H}'}^{\otimes n} F)}{\beta \otimes G}_q,
\end{split}
\]
which implies part (a). Part (b) follows from Proposition~\ref{Prop:Algebraic-CE}.
\end{proof}

\begin{Prop}
\label{Prop:Single}
In the single-variable case, for $\alpha \in S_0(n)$, we have
\[
\state{\T{\alpha \otimes h^{\otimes n}} \ |\ \mc{H}'}
= \sum_{\pi \in \Part_{1,2}(n)} \norm{P_{(\mc{H}')^\perp} h}^{2 \abs{\Pair{\pi}}} \T{C_\pi(\alpha) \otimes (P_{\mc{H}'} h)^{\otimes \abs{\Sing{\pi}}}}.
\]
\end{Prop}

\begin{proof}
We compute
\[
\begin{split}
& \state{\T{\alpha \otimes h^{\otimes n}} \ |\ \mc{H}'} \\
&\quad = \sum_{\pi \in \Part_{1,2}(n)} \norm{h}^{n - \abs{\Sing{\pi}}} \W{C_\pi(\alpha) \otimes (P_{\mc{H}'} h)^{\otimes \abs{\Sing{\pi}}}} \\
&\quad = \sum_{\pi \in \Part_{1,2}(n)} \norm{h}^{2 \abs{\Pair{\pi}}} \\
&\quad\qquad \sum_{\sigma \in \Part_{1,2}(\Sing{\pi})} (-1)^{\abs{\Pair{\sigma}}} \norm{P_{\mc{H}'} h}^{2 \abs{\Pair{\sigma}}}
\T{C_\sigma C_\pi(\alpha) \otimes (P_{\mc{H}'} h)^{\otimes \abs{\Sing{\sigma}}}} \\
&\quad = \sum_{\rho \in \Part_{1,2}(n)} \sum_{S \subset \Pair{\rho}} (-1)^{\abs{S}} \norm{h}^{2 \abs{\Pair{\rho}} - 2 \abs{S}} \norm{P_{\mc{H}'} h}^{2 \abs{S}} \T{C_\rho(\alpha) \otimes (P_{\mc{H}'} h)^{\otimes \abs{\Sing{\rho}}}} \\
&\quad = \sum_{\rho \in \Part_{1,2}(n)}\left(\norm{h}^2 - \norm{P_{\mc{H}'} h}^2 \right)^{\abs{\Pair{\rho}}} \T{C_\rho(\alpha) \otimes (P_{\mc{H}'} h)^{\otimes \abs{\Sing{\rho}}}} \\
&\quad = \sum_{\rho \in \Part_{1,2}(n)} \norm{P_{(\mc{H}')^\perp} h}^{2 \abs{\Pair{\rho}}} \T{C_\rho(\alpha) \otimes (P_{\mc{H}'} h)^{\otimes \abs{\Sing{\rho}}}}. \qedhere
\end{split}
\]
\end{proof}

The following could have been taken as the definition of $\W{\alpha \otimes F}$, with the more explicit expression in terms of contractions obtained a corollary; compare with Section~\ref{Section:Gaussian-Hilbert}. Since our algebraic definition works on $\overline{\mc{TP}}(\mc{H}_{\mf{R}})$, on which $\phi$ is not faithful and may not be positive, it is a little more general.

\begin{Prop}
\label{Prop:Projection}
For $q \in \mc{Z} \setminus \set{0}$, let $\overline{\mc{TP}}_n(\mc{H}_{\mf{R}})$ be the image of $\bigoplus_{k=0}^n \mf{C}[S_0(k)]\otimes \mc{H}_{\mf{R}}^{\otimes k}$ in $\mc{TP}_q(\mc{H}_{\mf{R}})$. Denote by $P_n$ the orthogonal projection
\[
P_n : \overline{\mc{TP}}_n(\mc{H}_{\mf{R}}) \rightarrow \overline{\mc{TP}}_n(\mc{H}_{\mf{R}}) \ominus \overline{\mc{TP}}_{n-1}(\mc{H}_{\mf{R}}).
\]
Then for $\alpha \otimes F \in \mf{C}[S_0(n)]\otimes \mc{H}_{\mf{R}}^{\otimes n}$, 
\[
\W{\alpha \otimes F} = P_n \T{\alpha \otimes F}
\]
\end{Prop}

\begin{proof}
For $\alpha \in S_0(n)$ and $F \in \mc{H}_{\mf{R}}^{\odot n}$, it suffices to note that $\W{\alpha \otimes F} - \T{\alpha \otimes F} \in \overline{\mc{TP}}_{n-1}(\mc{H}_{\mf{R}})$ by Proposition~\ref{Prop:T-I}, and $\W{\alpha \otimes F} \perp \overline{\mc{TP}}_{n-1}(\mc{H}_{\mf{R}})$ by definition of the $q$-inner product. For general $F \in \mc{H}_{\mf{R}}^{\otimes n}$, the result follows by approximation.
\end{proof}

We finish the section with the observation that all the constructions in the article are isomorphic for $q$ and $-q$. In particular, for $q \in \mc{Z} \setminus \set{0}$, it suffices to consider $q = \frac{1}{N}$. We will use the notation $C_\pi^{(q)}(\alpha)$, $T^{(q)}[\alpha \otimes F]$, $\phi_q$, and the multiplication $\cdot_q$ to indicate the dependence on $q$.

\begin{Prop}
\label{Prop:Negative}
Define the map $R : \mc{TP}(\mc{H}_{\mf{R}}) \rightarrow \mc{TP}(\mc{H}_{\mf{R}})$ by
\[
R(\alpha \otimes F) = (-1)^{\abs{\alpha}} \alpha \otimes F,
\]
so that $R(\W{\alpha \otimes F}) = (-1)^{\abs{\alpha}} \W{\alpha \otimes F}$. Then
\begin{itemize}
\item
$T^{(-q)}(R(\alpha \otimes F)) = R(T^{(q)}(\alpha \otimes F))$
\item
$R$ is a star-isomorphism from $(\mc{TP}(\mc{H}_{\mf{R}}), \cdot_q)$ to $(\mc{TP}(\mc{H}_{\mf{R}}), \cdot_{-q})$
\item
For $\alpha \in S_0(2n)$, $R$ satisfies
\[
\phi_{-q}[R(T^{(q)} (\alpha \otimes F))] = \varphi_{q}[T^{(q)}(\alpha \otimes F)].
\]
\item
For $q = 1/N$, $R$ is an isometry from $\mc{F}_{1/N}(\mc{H}_{\mf{R}})$ onto $\mc{F}_{-1/N}(\mc{H}_{\mf{R}})$.
\end{itemize}
\end{Prop}

\begin{proof}
It suffices to consider the case $\mc{H} = \mf{C}$. $R$ is clearly a bijection. We first compute
\[
\begin{split}
T^{(-q)}(\alpha)
& = \sum_{\pi \in \Part_{1,2}(n)} \W{C_\pi^{(-q)}(\alpha)} \\
& = \sum_{\pi \in \Part_{1,2}(n)} (-1)^{\cyc_0((\pi \alpha)|_{\supp{\pi}^c}) - \cyc_0(\pi \alpha) + \abs{\pi}} \W{C_\pi^{(q)}(\alpha)} \\
& = \sum_{\pi \in \Part_{1,2}(n)} (-1)^{\left(n - 2 \abs{\pi} - \abs{(\pi \alpha)|_{\supp{\pi}^c})}\right) -(n - \abs{\pi \alpha}) + \abs{\pi}} \W{C_\pi^{(q)}(\alpha)} \\
& =  \sum_{\pi \in \Part_{1,2}(n)} (-1)^{\abs{\pi} + \abs{\pi \alpha}} (-1)^{\abs{(\pi \alpha)|_{\supp{\pi}^c})}} \W{C_\pi^{(q)}(\alpha)} \\
& = (-1)^{\abs{\alpha}} \sum_{\pi \in \Part_{1,2}(n)} (-1)^{\abs{(\pi \alpha)|_{\supp{\pi}^c})}} \W{C_\pi^{(q)}(\alpha)} \\
& = (-1)^{\abs{\alpha}} \sum_{\pi \in \Part_{1,2}(n)} R \left(\W{C_\pi^{(q)}(\alpha)} \right) \\
& = (-1)^{\abs{\alpha}} R(T^{(q)}(\alpha))
\end{split}.
\]
Here we used the fact that for any $\alpha, \pi$, $-\abs{\alpha} + \abs{\pi \alpha} + \abs{\pi}$ is even. Since
\[
\begin{split}
R(T^{(q)}(\alpha)) \cdot_{-q} R(T^{(q)}(\beta))
& = (-1)^{\abs{\alpha} + \abs{\beta}} T^{(-q)}(\alpha) \cdot_{-q} T^{(-q)}(\beta) \\
& = (-1)^{\abs{\alpha \cup \beta}} T^{(-q)}(\alpha \cup \beta)
= R(T^{(q)}(\alpha) T^{(q)}(\alpha)),
\end{split}
\]
$R$ is a homomorphism. It clearly commutes with the adjoint operation. Since the action of the linear functional $\phi$ on $\W{\alpha \otimes F}$ does not depend on $q$, and $\ip{A}{B}_q = \phi_q[B^\ast A]$, the isometry property follows from the homomorphism property. Finally, for $\alpha \in S_0(2n)$,
\[
\begin{split}
\phi_{q}\left[(-1)^{\abs{\alpha}} T^{(q)}(\alpha \otimes F)\right]
& = \sum_{\pi \in \Part_2(2n)} (-1)^{\abs{\alpha}} q^{\abs{\pi \alpha} - n} C_\pi(F) \\
& = \sum_{\pi \in \Part_2(2n)} (-q)^{\abs{\pi \alpha} - n} C_\pi(F) \\
& = \phi_{-q}\left[T^{(-q)}(\alpha \otimes F)\right]. \qedhere
\end{split}
\]
\end{proof}

\subsection{Three subalgebras}

For $h \in \mc{H}_{\mf{R}}$, denote $r^+(h)$ the standard \emph{right} creation operator on the full Fock space $\bigoplus_{n=0}^\infty \overline{\mc{H}^{\otimes n}}$, and by $r^-_k(h)$ the annihilation operators
\[
r^-_k(h) F = C_{(k \ n+1)} (F \otimes h).
\]

\subsubsection{Gaussian subalgebra}
\label{Subsec:Gaussian}

\begin{Defn}
The Gaussian subalgebra of $\mc{TP}(\mc{H}_{\mf{R}})$ is
\[
\mc{G}(\mc{H}_{\mf{R}}) = \Alg{\T{{(0)(1)} \otimes h} : h \in \mc{H}_{\mf{R}}} = \Span{\T{{(0)(1)\ldots (n)} \otimes F} : h \in \mc{H}_{\mf{R}^{\otimes n}}}.
\]
\end{Defn}

\begin{Thm}
\label{Thm:Gaussian}
In the right GNS representation from Remark~\ref{Remark:GNS}, we may decompose
\[
\W{{(0)(1)} \otimes h} = \T{{(0)(1)} \otimes h} = a^+_{(0)(1)}(h) + a^-_{(0)(1)}(h),
\]
where
\[
a^+_{(0)(1)}(h) (\alpha \otimes F)
= {\alpha} \otimes r^+(h) F
\]
and
\[
\begin{split}
a^-_{(0)(1)}(h) (\alpha \otimes F)
&= \sum_{k : (k) \in \alpha} {P^{[0, n] \setminus \set{k}}_{[0, n-1]} \alpha|_{\set{k}^c}} \otimes r_k^-(h) F \\
&\qquad + q \sum_{k : (k) \not \in \alpha}  {P^{[0, n] \setminus \set{k}}_{[0, n-1]} \alpha|_{\set{k}^c}} \otimes r_k^-(h) F
\end{split}
\]
These operators are adjoints of each other.

The distribution of $\W{{(0)(1)} \otimes h}$ is Gaussian with mean $0$ and variance $\norm{h}^2$.
\end{Thm}

\begin{proof}
According to Proposition~\ref{Prop:Product-I},
\[
\begin{split}
& \W{\alpha \otimes F} \ \W{{(0)(1)} \otimes h} \\
&\quad = \W{{\alpha} \otimes r^+(h) F} \\
&\quad \qquad + \sum_{k=1}^n q^{\cyc_0(\alpha|_{\set{k}^c}) - \cyc_0(\alpha) + 1}  \W{{P^{[0, n] \setminus \set{k}}_{[0, n-1]} \alpha|_{\set{k}^c}} \otimes r^-_k(h) F} \\
&\quad = \W{{\alpha} \otimes r^+(h) F} \\
&\quad \qquad + \sum_{k : (k) \in \alpha} \W{{P^{[0, n] \setminus \set{k}}_{[0, n-1]} \alpha|_{\set{k}^c}} \otimes r^-_k(h) F}
+ q \sum_{k : (k) \not \in \alpha} \W{{P^{[0, n] \setminus \set{k}}_{[0, n-1]} \alpha|_{\set{k}^c}} \otimes r^-_k(h) F}
\end{split}
\]
Also
\[
\state{\W{{(0)(1)} \otimes h}^n}
= \state{\T{{(0)(1) \ldots (n)} \otimes h^{\otimes n}}}
=
\begin{cases}
\abs{\Part_2(n)} \ \norm{h}^{n}, & n \text{ even}, \\
0, & n \text{ odd}.
\end{cases}
\]
Thus the distribution of $\W{{(0)(1)} \otimes h}$ is Gaussian. Finally, since $\T{{(0)(1)} \otimes h}$ is symmetric, $a^+_{(0)(1)}(h)$ maps the $n$'th graded component into the $(n+1)$'st, and $a^-_{(0)(1)}(h)$ maps it into the $(n-1)$'st component, these operators have to be each other's adjoints.
\end{proof}

\begin{Remark}
The subspace generated by this algebra's action on ${(0)}$ is
\[
\Span{{(0)(1) \ldots (n)} \otimes_s F : F \in \mc{H}_{\mf{R}}^{\otimes n}, n \in \mf{N}}.
\]
Each element of this subspace has the same permutation component, which can be dropped. Moreover the action of $a^{\pm}_{(0)(1)}(h)$ on this subspace is independent of $q$, and the induced inner product is the usual symmetric inner product. Therefore this space is isomorphic to the symmetric Fock space $\mc{F}_s(\mc{H}_{\mf{R}})$, with the usual Bosonic creation and annihilation operators.
\end{Remark}

\begin{Remark}
For $q=1$, the inner product on the Fock space is
\[
\ip{(\alpha \otimes F)}{(\beta \otimes G)}_1 = \delta_{n=k} \sum_{\sigma \in S(n)} \ip{F}{U_\sigma G}_{\mc{H}_{\mf{R}}^{\otimes n}}.
\]
So the non-degenerate quotient of the space is isomorphic to the symmetric Fock space $\mc{F}_s(\mc{H})$, and $\mc{TP}_1(\mc{H}_{\mf{R}}) = \mc{G}(\mc{H}_{\mf{R}})$.
\end{Remark}

\subsubsection{Pure trace polynomial subalgebra}

\begin{Defn}
The pure trace polynomial subalgebra  of $\mc{TP}(\mc{H}_{\mf{R}})$ is $\Span{\W{\alpha \otimes_s F} : \alpha(0) = 0}$.
\end{Defn}

\begin{Remark}
Recall that
\[
\overline{\mc{TP}}(\mf{C}) = \bigoplus_{n=0}^\infty Z(\mf{C}[S_0(n)] : \mf{C}[S(n)])
\]
is commutative.
\end{Remark}

\begin{Thm}
\label{Thm:Center}
Let $\dim \mc{H} _{\mf{R}}\geq 2$. Then the pure trace polynomial subalgebra is the center of $\overline{\mc{TP}}(\mc{H}_{\mf{R}})$.
\end{Thm}

\begin{proof}
Recall from Notation~\ref{Notation:Union} that
\[
\sigma_{n, k} (\alpha \cup \beta) \sigma_{n, k}^{-1} = \beta \sigma_{n, k} \alpha \sigma_{n, k}^{-1}.
\]
If $\alpha \in S(n)$ and $\beta \in S_0(k)$, then $\beta$ and $\sigma_{n, k} \alpha \sigma_{n, k}^{-1}$ commute, and the expression above is $\beta \cup \alpha$. Since also $U_{\sigma_{n, k}}(F \otimes G) = G \otimes F$,
\[
(\alpha \cup \beta) \otimes_s (F \otimes G) = (\beta \cup \alpha) \otimes_s (G \otimes F).
\]
So using Lemma~\ref{Lemma:Conjugation}, for such $\alpha$,
\[
\begin{split}
\W{\alpha \otimes_s F} \W{\beta \otimes_s G}
& = \sum_{\pi \in \mc{P}_{1,2}(n, k)} \W{C_\pi(\alpha \cup \beta) \otimes_s C_\pi(F \otimes G)} \\
& = \sum_{\pi \in \mc{P}_{1,2}(n, k)} \W{C_\pi(\sigma_{n,k} ((\beta \cup \alpha) \sigma_{n,k}^{-1} \otimes_s U_{\sigma_{n,k}} (G \otimes F))} \\
& = \sum_{\pi \in \mc{P}_{1,2}(n, k)} \W{C_{ \sigma_{n,k}^{-1} \pi \sigma_{n,k}} ((\beta \cup \alpha) \otimes_s (G \otimes F))} \\
& = \W{\beta \otimes_s G} \W{\alpha \otimes_s F}.
\end{split}
\]

For the converse, denote $\set{\xi_i : i \in \Xi}$ an orthonormal basis for $\mc{H}_{\mf{R}}$. Suppose that for some $A \in \overline{\mc{TP}}(\mc{H}_{\mf{R}})$, $A \ \W{(01) \otimes \xi_i} = \W{(01) \otimes \xi_i} A$ for all $i \in \Xi$. $A$ has the form
\[
A = \sum_{n=0}^{k} \sum_{\alpha \in S_0(n)} \W{\alpha \otimes F_\alpha},
\]
where for each $n$ and $F_\alpha \in \mc{H}_{\mf{R}}^{\otimes n}$,
\begin{equation}
\label{Eq:Symmetrizer}
\frac{1}{n!} \sum_{\sigma \in S(n)} U_\sigma F_{\sigma^{-1} \alpha \sigma} = F_\alpha.
\end{equation}
Comparing only the terms in the $(n + 1)$'th component, it follows that
\[
\sum_{\alpha \in S_0(n)} \W{(\alpha \cup (01)) \otimes_s (F_\alpha \otimes \xi_i)} = \sum_{\beta \in S_0(n)} \W{((01) \cup \beta) \otimes_s (\xi_i \otimes F_\beta)}.
\]
Recall that $U_{\sigma_{1,n}}(h \otimes F_\beta) = F_\beta \otimes h$ and
\[
\sigma_{1,n} ((01) \cup \beta) \sigma_{1,n}^{-1} = \beta \sigma_{1, n} (01) \sigma_{1,n}^{-1} = \beta (0 \ n+1).
\]
Therefore
\[
\sum_{\alpha \in S_0(n)} \W{((0 \ n+1) \alpha) \otimes_s (F_\alpha \otimes \xi_i)}
= \sum_{\beta \in S_0(n)} \W{(\beta (0 \ n+1)) \otimes_s (F_\beta \otimes \xi_i)}.
\]
That is,
\begin{equation}
\label{Eq:Symmetrized}
\sum_{\sigma \in S(n+1)} \sum_{\alpha \in S_0(n)} \W{\sigma ((0 \ n+1) \alpha) \sigma^{-1} \otimes U_\sigma (F_\alpha \otimes \xi_i)}
= \sum_{\tau \in S(n+1)} \sum_{\beta \in S_0(n)} \W{\tau (\beta (0 \ n+1)) \tau^{-1} \otimes U_\tau (F_\beta \otimes \xi_i)}.
\end{equation}
Fix $\tilde{\alpha} \in S_0(n)$ with  $\tilde{\alpha}(0) \neq 0$,  and denote $\gamma = (0 \ n+1) \tilde{\alpha} $. Note that if $\gamma = \sigma (0 \ n+1) \alpha \sigma^{-1} $ with $\alpha \in S_0(n)$ and $\sigma \in S(n+1)$, then in fact $\sigma  \in S(n)$. This follows since 
\[
\sigma (n+1) = (\tilde{\alpha}^{-1} (0 \ n+1) \sigma (0 \ n+1) \alpha)(n+1) = n+1.
\]
Next, from \eqref{Eq:Symmetrized},
\[
\sum_{\substack{\sigma \in S(n+1), \alpha \in S_0(n) \\ \sigma (0 \ n+1) \alpha \sigma^{-1} = \gamma}} U_\sigma (F_\alpha \otimes \xi_i)
= \sum_{\substack{\tau \in S(n+1), \beta \in S_0(n) \\ \tau \beta (0 \ n+1) \tau^{-1} = \gamma}}  U_\tau (F_\beta \otimes \xi_i).
\]
The left-hand simplifies to
\[
\sum_{\substack{\sigma \in S(n), \alpha \in S_0(n) \\ \sigma \alpha \sigma^{-1} = \tilde{\alpha}}} U_\sigma (F_\alpha \otimes \xi_i)
= \sum_{\sigma \in S(n)} U_\sigma (F_{\sigma^{-1} \tilde{\alpha} \sigma} \otimes \xi_i)
=  n! F_{\tilde{\alpha}} \otimes \xi_i,
\]
where in the last step we used \eqref{Eq:Symmetrizer}. Note also that for  $\tau, \beta$ as in the sum above,
\begin{equation}
\label{Eq:Zeros}
\gamma(0) = \tilde{\alpha}(0) = \tau(n+1),
\end{equation}
where we recall that $\tilde{\alpha}(0) \neq 0, n+1$. Therefore
\[
n! F_{\tilde{\alpha}} \otimes \xi_i = \sum_{\substack{\tau \in S(n+1), \beta \in S_0(n) \\ \tau \beta (0 \ n+1) \tau^{-1} = \gamma}}  U_\tau (F_\beta \otimes \xi_i) \in U_{(\tilde{\alpha}(0) \ n+1)} (\mc{H}_{\mf{R}}^{\otimes n} \otimes \Span{\xi_i}).
\]
It follows that $F_{\tilde{\alpha}} \in  U_{(\tilde{\alpha}(0) \ n)} (\mc{H}_{\mf{R}}^{\otimes (n-1)} \otimes \Span{\xi_i})$. Since this is true for each $i \in \Xi$, which has at least 2 elements, we conclude that $F_{\tilde{\alpha}} = 0$. 
\end{proof}

\subsubsection{Polynomial subalgebra}
\label{Subsec:Polynomial-aubalgebra}

\begin{Defn}
The polynomial subalgebra  of $\mc{TP}(\mc{H}_{\mf{R}})$ is
\[
\mc{P}(\mc{H}_{\mf{R}})  = \Alg{\T{{(01)} \otimes h}} = \Span{\T{\alpha \otimes_s F} : \cyc_0(\alpha) = 0}
\]
\end{Defn}

\begin{Remark}
Denote $X(h) = \T{{(01)} \otimes h}$, so that $\mc{P}(\mc{H}_{\mf{R}})$ is generated by these operators explaining the name. Also, 
\[
X(h_1) \ldots X(h_n) = \T{{(0 1 \ldots n)} \otimes (h_1 \otimes \ldots \otimes h_n)}
\]
and any other long cycle is conjugate to the standard one, so the two expressions in the preceding definition are equivalent.  $\mc{P}(\mc{H}_{\mf{R}})$ is a unital star-subalgebra of $\mc{TP}(\mc{H}_{\mf{R}})$. It is not closed under contractions or conditional expectations. In particular, the corresponding $\W{{(0 1 \ldots n)} \otimes (h_1 \otimes \ldots \otimes h_n)}$ operators do not in general belong to this subalgebra.

Clearly $\mc{P}(\mc{H}_{\mf{R}})$ and the center $\Span{\T{\alpha \otimes_s F} : \alpha(0) = 0}$ together generate $\mc{TP}(\mc{H}_{\mf{R}})$.
\end{Remark}

\begin{Thm}
Suppose $\mc{H}_{\mf{R}}$ is infinite-dimensional. Then $\mc{TP}(\mc{H}_{\mf{R}})$ is generated (as an algebra) by the conditional expectations
\[
\set{\state{A \ |\ \mc{H}'} : A \in \mc{P}(\mc{H}_{\mf{R}}), \mc{H}'_{\mf{R}} \subset \mc{H}_{\mf{R}}}.
\]
\end{Thm}

\begin{proof}
By the preceding remark, it suffices to show that the algebra generated by conditional expectations contains the center $\Span{\T{\alpha \otimes_s F} : \alpha(0) = 0}$. By Definition~\ref{Defn:Product-tensors} and linearity, the center is generated by $\T{\beta \otimes_s F}$, where $\beta \in S(n)$ is a single cycle not containing $0$, and $F$ is a simple tensor. We will use induction on $n$.  Suppose that for each $\beta \in S(k)$, $k < n$, $\T{\beta \otimes_s F}$ is in the algebra generated by the conditional expectations of elements of $\mc{P}(\mc{H}_{\mf{R}})$.  It then suffices to show that $\T{{(0)(1 \ldots n)} \otimes (h_1 \otimes \ldots \otimes h_n)}$ is in it. Let $h \perp \mc{H}'_{\mf{R}} = \Span{h_1, \ldots, h_n}$ be a non-zero vector. Then
\[
\begin{split}
& \state{X(h) \W{{(0 1 \ldots n)} \otimes (h_1 \otimes \ldots \otimes h_n)} X(h) \ |\ \mc{H}'} \\
&\quad = \state{C_{(1 \ n+2)} \W{{(0 1 \ldots n+2)} \otimes (h \otimes h_1 \otimes \ldots \otimes h_n \otimes h)} \ |\ \mc{H}'} \\
&\quad = q \norm{h}^2 \W{{(0)(1 \ldots n)} \otimes (h_1 \otimes \ldots \otimes h_n)}.
\end{split}
\]
Therefore using Definition~\ref{Defn:T-I},
\[
\begin{split}
& \state{X(h) \T{{(0 1 \ldots n)} \otimes (h_1 \otimes \ldots \otimes h_n)} X(h) \ |\ \mc{H}'} \\
&\quad \state{X(h) (\W{{(0 1 \ldots n)} \otimes (h_1 \otimes \ldots \otimes h_n)} + \text{ lower order terms}) X(h) \ |\ \mc{H}'} \\
&\quad = q \norm{h}^2 \W{{(0)(1 \ldots n)} \otimes (h_1 \otimes \ldots \otimes h_n)} + \text{ lower order terms} \\
&\quad = q \norm{h}^2 \T{{(0)(1 \ldots n)} \otimes (h_1 \otimes \ldots \otimes h_n)} + \text{ lower order terms}.
\end{split}
\]
The result follows by induction.
\end{proof}

It is also clear that $\mc{TP}(\mc{H}_{\mf{R}})$ is the smallest algebra generated by $\mc{P}(\mc{H}_{\mf{R}})$ which is closed under the center-valued trace operation in the next corollary. The proof is similar to the one of the preceding or the following result.

\begin{Cor}
\label{Cor-Center-trace}
The $\phi$-preserving conditional expectation from $\overline{\mc{TP}}(\mc{H}_{\mf{R}})$ onto its center is the map
\[
C: \W{\alpha \otimes_s F} \mapsto 
\begin{cases}
\W{\alpha \otimes_s F}, &  \alpha(0) = 0, \\
q \W{{\alpha|_{\set{0}^c}} \otimes_s F}, & \text{otherwise}.
\end{cases}
\]
Here for $\alpha \in S_0(n)$, the permutation $\alpha|_{\set{0}^c}$ is obtained by removing $0$ from its cycle in $\alpha$, similarly to Notation~\ref{Notation:Restriction}. On $\mc{TP}(\mc{H}_{\mf{R}})$ this conditional expectation is implemented by
\[
\state{X(h) \W{ \alpha \otimes_s (h_1 \otimes \ldots \otimes h_n)} X(h) \ |\ \mc{H}'},
\]
where $\mc{H}'_{\mf{R}} = \Span{h_1, \ldots, h_n}$ and $h \perp \mc{H}'_{\mf{R}}$ is a unit vector.
\end{Cor}

A similar representation holds for general contractions.

\begin{Lemma}
\label{Lemma:Orthogonal-contractions}
Let $\pi = \set{(v_1, w_1), \ldots, (v_\ell, w_\ell), (u_1) \ldots (u_{n - 2 \ell} )} \in \Part_{1,2}(n)$, and $\mc{H}'_{\mf{R}} \subset \mc{H}_{\mf{R}}$ a closed subspace. Let $h_1, \ldots, h_n \in \mc{H}_{\mf{R}}$ be vectors such that
\begin{itemize}
\item
$h_{v_i} = h_{w_i}$, $1 \leq i \leq \ell$.
\item
The vectors $\set{h_{v_1}, \ldots, h_{v_\ell}}$ are an orthonormal subset of $(\mc{H}'_{\mf{R}})^\perp$.
\item
$\set{h_{u_1}, \ldots, h_{u_{n - 2 \ell}}} \subset \mc{H}'_{\mf{R}}$.
\end{itemize}
Then
\[
\state{\T{\alpha \otimes (h_1 \otimes \ldots \otimes h_n)} \ |\ \mc{H}'} = \T{C_\pi(\alpha \otimes (h_1 \otimes \ldots \otimes h_n))}.
\]
\end{Lemma}

\begin{proof}
Using the assumptions on the vectors,
\[
\begin{split}
& \state{\T{\alpha \otimes (h_1 \otimes \ldots \otimes h_n)} \ |\ \mc{H}'} \\
&\quad = \state{\sum_{\sigma \in \Part_{1,2}(n)} \W{C_\sigma(\alpha \otimes (h_1 \otimes \ldots \otimes h_n))} \ |\ \mc{H}'} \\
&\quad = \sum_{\substack{\sigma \in \Part_{1,2}(n) \\ \Pair{\pi} \subset \Pair{\sigma}}} \W{C_\sigma(\alpha \otimes (h_1 \otimes \ldots \otimes h_n))} \\
&\quad = \sum_{\rho \in \Part_{1,2}([\abs{\Sing{\pi}}])} \W{C_\rho C_\pi(\alpha \otimes (h_1 \otimes \ldots \otimes h_n))} \\
&\quad = \T{C_\pi(\alpha \otimes (h_1 \otimes \ldots \otimes h_n))}. \qedhere
\end{split}
\]
\end{proof}

\begin{Prop}
For $q \in \mc{Z} \setminus \set{0}$, in its representation on $\mc{F}_q(\mc{H})$, for $h \in \mc{H}_{\mf{R}}$, $X(h)$ is essentially self-adjoint.
\end{Prop}

\begin{proof}
Clearly $X({h})$ is symmetric. So by Nelson's analytic vector theorem, it suffices to show that its domain contains a subset of analytic vectors which is total (i.e. their span is dense). We verify that for each $\alpha \in S_0(k)$, $\T{\alpha \otimes F} {(0)}$ is an analytic vector for it. Indeed, using the relation in Theorem~\ref{Thm:State} between the inner product on $\mc{F}_q(\mc{H})$ and the state on $\mc{T P}(\mc{H}_{\mf{R}})$,
\[
\begin{split}
& \frac{1}{n} \norm{X(h)^n \T{\alpha \otimes F} {(0)}}_q^{1/n} \\
&\quad = \frac{1}{n} \state{\T{{\alpha^{-1} \cup (0 1 \ldots 2n) \cup \alpha} \otimes (\bar{F} \otimes h^{\otimes 2 n} \otimes F)}}^{1/2n} \\
&\quad = \frac{1}{n} \left( \sum_{\pi \in \Part_2(2n + 2 k)} q^{(n + k) - \cyc_0(\pi (\alpha^{-1} \cup (0 1 \ldots 2n) \cup \alpha))} C_\pi(\bar{F} \otimes h^{\otimes 2 n} \otimes F) \right)^{1/2n} \\
&\quad \leq \frac{1}{n} \left( \frac{(2 n + 2k)!}{2^{n+k} (n+k)!} \max(1, \abs{q}^{-(n+k)}) \norm{F}^2 \norm{h}^{2n} \right)^{1/2n} \\
&\quad \sim \frac{1}{n} 2^{1/2} (n + k)^{1/2} e^{-1/2} \max(1, \abs{q}^{-1/2}) \norm{h} \rightarrow 0,
\end{split}
\]
where we used the fact that $\abs{q} \leq 1$ and $\abs{\Part_2(2n)} = \frac{(2n!)}{2^n n!}$.
\end{proof}

\begin{Thm}
\label{Thm:GUE}
In the right GNS representation, we may decompose $X(h) = a^+_{(01)}(h) + a^-_{(01)}(h)$, where
\[
a^+_{(01)}(h) (\alpha \otimes F)
= {(0 \ n+1) \alpha} \otimes r^+(h) F
\]
and
\[
\begin{split}
a^-_{(01)}(h) (\alpha \otimes F)
&= q \sum_{k \neq \alpha^{-1}(0)} P^{[0, n] \setminus \set{k}}_{[0, n-1]} ((0k)\alpha)|_{\set{k}^c} \otimes r^-_k(h) F \\
&\qquad+ \delta_{\alpha(0) \neq 0}  P^{[0, n] \setminus \set{\alpha^{-1}(0)}}_{[0, n-1]} \alpha|_{\set{\alpha^{-1}(0)}^c} \otimes r^-_{\alpha^{-1}(0)}(h) F.
\end{split}
\]
The distribution of $\W{{(01)} \otimes h}$ is the unnormalized average empirical distribution of a GUE matrix with mean $0$ and variance $\norm{h}^2$.
\end{Thm}

\begin{proof}
\[
\begin{split}
& \W{\alpha \otimes F} \ \W{(01) \otimes h} \\
&\quad = \W{(0 \ n+1) \alpha \otimes r^+(h) F} \\
&\quad \quad + \sum_{k=1}^n q^{\cyc_0(((0k)\alpha)|_{\set{k}^c}) - \cyc_0((0k)\alpha) + 1} \W{P^{[0, n] \setminus \set{k}}_{[0, n-1]} ((0k)\alpha)|_{\set{k}^c} \otimes r^-_k(h) F} \\
&\quad = \W{(0 \ n+1) \alpha \otimes r^+(h) F} \\
&\quad \quad + q \sum_{k \neq \alpha^{-1}(0)} \W{P^{[0, n] \setminus \set{k}}_{[0, n-1]} ((0k)\alpha)|_{\set{k}^c} \otimes r^-_k(h) F} \\
&\quad \quad + \delta_{\alpha(0) \neq 0}  \W{P^{[0, n] \setminus \set{\alpha^{-1}(0)}}_{[0, n-1]} \alpha|_{\set{\alpha^{-1}(0)}^c} \otimes r^-_{\alpha^{-1}(0)}(h) F}.
\end{split}
\]
Also,
\[
\state{\W{{(01)} \otimes h}^n}
= \state{\T{{(0 1 \ldots n)} \otimes h^{\otimes n}}}
= \begin{cases}
\sum_{\pi \in \Part_2(n)} q^{(n/2) - \cyc_0((0 \ldots n) \pi)} \norm{h}^{n}, & n \text{ even}, \\
0, & n \text{ odd},
\end{cases}
\]
which should be compared with Theorem 22.12 in \cite{Nica-Speicher-book}.
\end{proof}

\subsection{The relation to a construction by Bo\.{z}ejko and Gu\c{t}\u{a}}

We contrast the algebra $\mc{P}(\mc{H}_{\mf{R}})$ with a construction from \cite{Bozejko-Guta}. In section 5 of that paper, Bo\.{z}ejko and Gu\c{t}\u{a} considered the Fock space with the inner product
\[
\ip{f_1 \otimes \ldots \otimes f_n}{g_1 \otimes \ldots \otimes g_k}_q = \delta_{n=k} \sum_{\sigma \in S(n)} \chi_q[\sigma] \prod_{i=1}^n \ip{f_i}{g_{\sigma(i)}}
\]
for $q = \pm \frac{1}{N}$. On this space, they defined the creation operator $a^+(h)$ in the usual way, and the annihilation operator as its adjoint, which comes out to be
\[
a^-(h)(h_1 \otimes \ldots \otimes h_n) = \ip{h_1}{h} + q \sum_{k=2}^n \ip{h_k}{h} (h_2 \otimes \ldots \otimes h_{i-1} \otimes h_1 \otimes h_{i+1} \otimes \ldots \otimes h_n)
\]
(compare with Theorem~\ref{Thm:GUE}). Then (Lemma~5.1) the operators $\omega(h) = a^+(h) + a^-(h)$ satisfy (with our notation)
\[
\ip{\Omega}{\omega(h_1) \ldots \omega(h_{2n}) \Omega} = \sum_{\pi \in \Part_2(2n)} q^{n - c(\pi)} C_\pi(h_1 \otimes \ldots \otimes h_{2n})
\]
and the corresponding expression is zero for an odd number of factors (compare with equation~\eqref{Eq:GUE-moment}). Here $c(\pi)$ is again the number of cycles of a permutation corresponding to a partition $\pi$, but this correspondence is more subtle. For a pair partition $\pi$, there is a unique non-crossing partition $\tilde{\pi}$ with the same openers and closers as $\pi$. If $i \stackrel{\pi}{\sim} j$ and $i \stackrel{\tilde{\pi}}{\sim} k$, then for the corresponding permutation $\sigma$, $\sigma(i) = j$ if $i < j$, and $\sigma(i) = k$ if $k < i$ (it is easy to check that this is an alternative). In other words, $\sigma$ is a permutation with an upper partition $\pi$ and the lower partition $\tilde{\pi}$ in the sense of Corteel \cite{Corteel-Crossings-permutations}. Then $c(\pi)$ is the number of cycles of $\sigma$.

Moreover (Lemma~5.1), the creation and annihilation operators satisfy a commutation relation
\[
a^-(f) a^+(g) = \ip{f}{g} + q \ d\Gamma(|g \rangle \langle f |),
\]
where $d\Gamma(A)$ is the standard second quantization operator. We prove an analog of this relation in our context below. Note however that in other aspects, our construction behaves quite differently. For example, there is no simple commutation relation between $a^+_{(01)}$ and $d\Gamma(A)$ below. Also, for $q = - \frac{1}{N}$, the distribution of $\omega(h)$ only has finite support, in contrast to Theorem~\ref{Thm:State} and Propositions~\ref{Prop:Negative}.

\begin{Remark}
Both our construction and the one in \cite{Bozejko-Guta} are quite different from the more familiar $q$-deformed free Fock space, for which $q = \pm \frac{1}{N}$ hold no special meaning.
\end{Remark}

\begin{Remark}
Combining our construction with \cite{Bozejko-Guta}, one could consider a Fock space with the inner product
\[
\ip{f_1 \otimes \ldots \otimes f_n}{g_1 \otimes \ldots \otimes g_k}_q = \delta_{n=k} \sum_{\sigma \in S(n)} \chi_q[\sigma \alpha \sigma^{-1} \alpha^{-1}] \prod_{i=1}^n \ip{f_i}{g_{\sigma(i)}}
\]
for a fixed permutation $\alpha$, for example for $\alpha = (0 1 \ldots n)$. It is easy to see that for $q \in \mc{Z}$, this inner product is positive semi-definite for any $\alpha$. For particular choices of $\alpha$, it is positive semi-definite for a wider range of $q$.
\end{Remark}

\begin{Defn}
Let $A$ be a (bounded for simplicity) linear operator on $\mc{H}$. Define its differential second quantization
\[
d\Gamma(A)(\alpha \otimes (h_1 \otimes \ldots \otimes h_n))
= \sum_{i=1}^n {(0 i) \alpha} \otimes (h_1 \otimes \ldots \otimes A h_i \otimes \ldots \otimes h_n).
\]
\end{Defn}

Note that
\[
d\Gamma(I)(\alpha \otimes F)
=  \sum_{i=1}^n {(0 i)} \alpha \otimes F,
\]
where $\sum_{i=1}^n {(0 i)}$ is the Jucys-Murphy element.

\begin{Lemma}
\

\begin{enumerate}
\item
For $P_n$ the symmetrizing projection from equation~\eqref{Eq:Symm-proj}, $d\Gamma(A) \ P_n = P_n \ d\Gamma(A)$. Therefore $d\Gamma(A)$ restricts to an operator on $\mc{TP}(\mc{H})$ and $\mc{TP}_q(\mc{H})$.
\item
$(d\Gamma(A))^\ast = d\Gamma(A^\ast)$.
\end{enumerate}
\end{Lemma}

\begin{proof}
For part (a), we note that
\[
\begin{split}
& d\Gamma(A) \ P_n (\alpha \otimes (h_1 \otimes \ldots \otimes h_n)) \\
&\quad = \sum_{\sigma \in S(n)} \sum_{i=1}^n {(0i) \sigma \alpha \sigma^{-1}} \otimes (h_{\sigma^{-1}(1)} \otimes \ldots \otimes A h_{\sigma^{-1}(i)} \otimes \ldots \otimes h_{\sigma^{-1}(n)}) \\
&\quad = \sum_{\sigma \in S(n)} \sum_{i=1}^n {\sigma (0 \sigma^{-1}(i)) \alpha \sigma^{-1}} \otimes (h_{\sigma^{-1}(1)} \otimes \ldots \otimes A h_{\sigma^{-1}(i)} \otimes \ldots \otimes h_{\sigma^{-1}(n)}) \\
&\quad = P_n \ d\Gamma(A) (\alpha \otimes (h_1 \otimes \ldots \otimes h_n)).
\end{split}
\]
The restriction to $\mc{TP}_q(\mc{H})$ follows since $\mc{N}_{gr, q}$ is an ideal. Similarly, for part (b),
\[
\begin{split}
& \ip{ d\Gamma(A) (\alpha \otimes (f_1 \otimes \ldots \otimes f_n))}{\beta \otimes (g_1 \otimes \ldots \otimes g_n)}_q \\
&\quad = \sum_{\sigma \in S(n)} \sum_{i=1}^n \chi_q((0 i) \alpha \sigma \beta^{-1} \sigma^{-1}) \ip{A f_{i}}{g_{\sigma^{-1}(i)}} \prod_{j \neq i} \ip{f_j}{g_{\sigma^{-1}(j)}} \\
&\quad = \sum_{\sigma \in S(n)} \sum_{i=1}^n \chi_q(\alpha \sigma ((0 \sigma^{-1}(i)) \beta)^{-1} \sigma^{-1}) \ip{f_i}{A^\ast g_{\sigma^{-1}(i)}} \prod_{j \neq i} \ip{f_j}{g_{\sigma^{-1}(j)}} \\
&\quad = \ip{\alpha \otimes (f_1 \otimes \ldots \otimes f_n)}{ d\Gamma(A^\ast) (\beta \otimes (g_1 \otimes \ldots \otimes g_n))}. \qedhere
\end{split}
\]
\end{proof}

\begin{Prop}
Splitting $X(h)$ into the creation operator $a^+(h)$ and annihilation operator $a^-(h)$ as in Theorem~\ref{Thm:GUE}, we have
\[
a^-(f) a^+(g) = \ip{f}{g} + q \ d\Gamma(|g \rangle \langle f |).
\]
\end{Prop}

\begin{proof}
With the notation from Theorem~\ref{Thm:GUE},
\[
\begin{split}
& a^-(f) a^+(g) (\alpha \otimes_s (h_1 \otimes \ldots \otimes h_n)) \\
&\quad = a^-(f) ({(0 \ n+1) \alpha} \otimes_s h_1 \otimes \ldots \otimes h_n \otimes g) \\
&\quad = q \sum_{k = 1}^n \ip{h_k}{f} {P^{[0, n+1] \setminus \set{k}}_{[0, n]} ((0k)(0 \ n+1) \alpha)|_{\set{k}^c}} \otimes_s (h_1 \otimes \ldots \otimes \hat{h}_k \otimes \ldots \otimes h_n \otimes g) \\
&\quad \qquad + \ip{f}{g} {\alpha} \otimes_s (h_1 \otimes \ldots \otimes h_n)
\end{split}
\]
Denoting $\sigma_k = (k \ k+1 \ldots n-1 \ n)$, we have
\[
U_{\sigma_k} (h_1 \otimes \ldots \otimes \hat{h}_k \otimes \ldots \otimes h_n \otimes g) = h_1 \otimes \ldots \otimes h_{k-1} \otimes g \otimes h_{k+1} \otimes \ldots \otimes h_n.
\]
On the other hand, the bijection
\[
\tau_k = P^{[0, n+1] \setminus \set{k}}_{[0, n]} : i \mapsto
\begin{cases}
i, & 1 \leq i \leq k - 1 \\
i-1, & k+1 \leq i \leq n+1.
\end{cases}
\]
and so $\sigma_k \tau_k(i) = i$ for $i \neq n+1$, $\sigma_k \tau_k(n+1) = k$. Therefore
\[
\sigma_k \tau_k ((0k)(0 \ n+1) \alpha)|_{\set{k}^c} \tau_k^{-1} \sigma_k^{-1}
= (0 k) \alpha. \qedhere
\]
\end{proof}

\begin{Remark}
\[
\exp(d\Gamma(I)) = \sum_{k=0}^\infty \frac{1}{k!} (d\Gamma(I))^k
= \sum_{\beta \in S_0(n)} \beta \sum_{k=0}^\infty \frac{1}{k!} \abs{\set{\mb{i} \in [n]^k : \beta = (0 i(k)) \ldots (0 i(1))}}.
\]
Here the coefficient of $\beta$ is the generating function of the number of primitive factorizations of $\beta$. See \cite{Matsumoto-Novak-primitive}.
\end{Remark}

\section{Trace polynomials in GUE matrices}
\label{Sec:GUE}

\subsection{Background}

\label{Subsec:Trace-background}

Abstract trace polynomials, and in particular the expression $\tr_\alpha[x_1, \ldots, x_n]$ for $\alpha \in S_0(n)$, were defined in the introduction. Let $\mc{A}$ be a unital algebra, $\mc{C}$ its center, and $F: \mc{A} \rightarrow \mc{C}$ a unital, tracial, $\mc{C}$-bimodule linear map. For any $\alpha \in S_0(n)$, we can similarly form $F_\alpha(a_1, \ldots, a_n) \in \mc{A}$, and consider it as the application of the trace monomial $\tr_\alpha[x_1, \ldots, x_n]$ to the elements $a_1, \ldots, a_n$. See \cite{Cebron-Free-convolution}. We extend the notation to $F_\eta$ for $\eta \in \mf{C}[S_0(n)]$ by linearity.

\subsubsection{Invariant theory of $N \times N$ matrices}

Let $\set{x_{ij}^{(k)} : k \in S, 1 \leq i, j \leq N}$ be formal commuting variables subject to the relation $x_{ji}^{(k)} = (x_{ij}^{(k)})^\ast$. For each $k$, form a matrix $X^{(k)} = (x_{ij}^{(k)})_{i, j = 1}^N$. Let $\mc{A}_{N, S}$ be the collection of all matrices with polynomial entries
\[
\mc{A}_{N, S} = M_N(\mf{C}) \otimes \mf{C}[x_{ij}^{(k)} : k \in S, 1 \leq i, j \leq N].
\]

Let $Y \in \mc{A}_{N, S}$, $Y = P(X^{(k)} : k \in S)$, where each entry $P_{ab}$ is a polynomial in the entries of its argument. We say that $Y$ is equivariant if for any $U \in U_N(\mf{C})$,
\[
P(U X^{(k)} U^\ast : k \in S) = U Y U^\ast.
\]
Denote
\[
\mc{A}_{N, S}^{\text{equiv}}
= \set{\text{equivariant } Y \in \mc{A}_{N, S}}.
\]
Then
\begin{equation}
\label{Eq:Equivariant}
\mc{A}_{N, S}^{\text{equiv}} = \Span{\Tr_\alpha(X^{(k(1))}, \ldots, X^{(k(n))}) : \alpha \in S_0(n), n \geq 0, k(1), \ldots, k(n) \in S},
\end{equation}
where $\Tr$ is the (un-normalized) trace on $M_N(\mf{C})$. Indeed, since for the purposes of this expansion, $x_{ij}^{(k)}$ and $(x_{ij}^{(k)})^\ast$ can be considered as independent variables, this follows directly from the first Procesi-Razmyslov theorem \cite{Procesi,Razmyslov}. The result is usually formulated using $GL(n)$-invariance. However, the argument ultimately reduces to Schur-Weyl duality, for which (in the case of inner product spaces) unitary invariance is sufficient.

\subsubsection{Hermitian Brownian motion}

Let $\set{b_{ij}(h) : h \in \mc{H}_{\mf{R}}}$ be $N^2$ standard Gaussian processes indexed by the same real Hilbert space as in Section~\ref{Section:Gaussian-Hilbert}, independent for different $(i,j)$, represented on the same probability space. Define the $N \times N$ Hermitian Gaussian process $\set{X(h) : h \in \mc{H}_{\mf{R}}}$ by
\[
X(h)_{ij} =
\begin{cases}
\frac{1}{\sqrt{2N}} (b_{ij}(h) + \sqrt{-1} b_{ji}(h)), & i < j, \\
\frac{1}{\sqrt{N}} b_{ij}(h), & i = j, \\
\frac{1}{\sqrt{2N}} (b_{ij}(h) - \sqrt{-1} b_{ji}(h)), & i > j. \\
\end{cases}
\]
Equivalently, each $X(h)$ is a Hermitian random matrix, whose entries are centered Gaussian variables with the joint covariance
\[
\Exp{X(f)_{ij} X(g)_{k \ell}} = \frac{1}{N} \delta_{i=\ell} \delta_{j=k} \ip{f}{g},
\]
so that
\[
(I \otimes \mf{E})[X(f) X(g)] = \ip{f}{g} I_N.
\]
Note that $(I \otimes \mf{E})[\Tr_\alpha(X(h_1), \ldots, X(h_n))]$ is always a scalar. Indeed, since the distribution of this random matrix is unitarily invariant, so is the distribution of its entry-wise expectation, which then has to be a multiple of identity. By a slight abuse of notation, we will denote this scalar-valued functional by $\mf{E}$ again.

\begin{Remark}
\label{Remark:Tensor-Hilbert-space}
Let $\mc{H}_{\mf{R}}$ be a real Hilbert space, $M_N^{sa}(\mf{C})$ the space of complex Hermitian matrices, $\mc{K}_{\mf{R}} = M_N^{sa}(\mf{C}) \otimes \mc{H}_{\mf{R}}$ a real Hilbert space, and $\mc{K} = M_N(\mf{C}) \otimes \mc{H}_{\mf{R}}$ its complexification, with the inner product
\[
\ip{A \otimes f}{B \otimes g} = \frac{1}{N} \Tr[A B^\ast] \ip{f}{g}.
\]
Let $\set{X(f) : f \in \mc{K}}$ be the complex Gaussian Hilbert space indexed by $\mc{K}$. Then for $h \in \mc{H}_{\mf{R}}$ and $X(h)$ the corresponding Gaussian random matrix, we may identify
\[
X(h)_{ij} = X(E_{ij} \otimes h).
\]
\end{Remark}

The following Proposition is a slight extension of a well-known result (see for example Lemma~22.31 of \cite{Nica-Speicher-book}) to the case of a general permutation $\alpha \in S_0(n)$ and GUE matrices indexed by a Hilbert space. We provide the proof for completeness.

\begin{Prop}
\label{Prop:GUE-moments}
Let $\set{D_i: i \in I}$ be non-random $N \times N$ matrices. For even $n$,
\[
D^{(0)} (I \otimes \mf{E})\left[\Tr_\alpha[X(h_1) D^{(1)}, X(h_2) D^{(2)}, \ldots, X(h_n) D^{(n)}]\right]
= \frac{1}{N^{n/2}} \sum_{\pi \in \Part_2(n)} C_\pi(h_1 \otimes \ldots \otimes h_n) D^{(0)} \Tr_{\pi \alpha}[D^{(1)}, D^{(2)}, \ldots, D^{(n)}].
\]
In particular,
\[
\Exp{\Tr_\alpha[X(h_1), X(h_2), \ldots, X(h_n)]}
=  \sum_{\pi \in \Part_2(n)} \frac{1}{N^{n/2 - \cyc_0(\pi \alpha)}} C_\pi(h_1 \otimes \ldots \otimes h_n).
\]
\end{Prop}

\begin{proof}
\[
\begin{split}
& \left(D^{(0)} (I \otimes \mf{E})\left[\Tr_\alpha[X(h_1) D^{(1)}, X(h_2) D^{(2)}, \ldots, X(h_n) D^{(n)}]\right]\right)_{\ell r} \\
&\quad = \sum_{\substack{\mb{u}, \mb{v} \in [0,n]^N \\ u(0) = r, v(0) = \ell}}  \prod_{i=0}^n D_{v(i), u(\alpha(i))}^{(i)} \Exp{\prod_{i=1}^n X_{u(i), v(i)}(h_i)} \\
&\quad = \sum_{\pi \in \Part_2(n)} \sum_{\substack{\mb{u}, \mb{v} \in [0,n]^N \\ u(i) = v(\pi(i)) \\ u(0) = r, v(0) = \ell}}  \prod_{i=0}^n D_{v(i), u(\alpha(i))}^{(i)} \prod_{(i, \pi(i)) \in \pi} \Exp{X_{v(\pi(i)), v(i)}(h_i) X_{v(i), v(\pi(i))}(h_{\pi(i)})} \\
&\quad = \sum_{\pi \in \Part_2(n)} \sum_{\substack{\mb{u}, \mb{v} \in [0,n]^N \\ u(i) = v(\pi(i)) \\ u(0) = r, v(0) = \ell}}  \prod_{i=0}^n D_{v(i), u(\alpha(i))}^{(i)} \frac{1}{N^{n/2}} C_\pi(h_1 \otimes \ldots \otimes h_n) \\
&\quad = \sum_{\pi \in \Part_2(n)} \sum_{\substack{\mb{v} \in [0,n]^N \\ u(0) = r, v(0) = \ell}}  \prod_{i=0}^n D_{v(i), v(\pi \alpha(i))}^{(i)} \frac{1}{N^{n/2}} C_\pi(h_1 \otimes \ldots \otimes h_n) \\
&\quad = \sum_{\pi \in \Part_2(n)} \left( D^{(0)} \Tr_{\pi \alpha}[D^{(1)}, D^{(2)}, \ldots, D^{(n)}]\right)_{\ell r} \frac{1}{N^{n/2}} C_\pi(h_1 \otimes \ldots \otimes h_n). \qedhere
\end{split}
\]
\end{proof}

\begin{Notation}
We now restrict the index set $S$ to be a real Hilbert space $\mc{H}_{\mf{R}}$. Denote
\[
\begin{split}
\mc{A}_{N}(\mc{H}) & = M_N(\mf{C}) \otimes \mf{C}[b_{ij}(h) : h \in \mc{H}_{\mf{R}}, 1 \leq i, j \leq N] \\
& = \set{P(b_{ij}(h): h \in \mc{H}_{\mf{R}}, 1 \leq i, j \leq N) : P \in \mc{A}_{N, \mc{H}_{\mf{R}}}}
\end{split}
\]
and
\[
\begin{split}
\mc{A}_{N}^{\text{equiv}}(\mc{H})
& = \set{P(b_{ij}(h): h \in \mc{H}_{\mf{R}}, 1 \leq i, j \leq N) : P \in \mc{A}_{N, \mc{H}_{\mf{R}}}^{\text{equiv}}} \\
& = \Span{\Tr_\alpha(X(h_1), \ldots, X(h_n) : \alpha \in S_0(n), n \geq 0, h_1, \ldots, h_n \in \mc{H}_{\mf{R}})}.
\end{split}
\]
Denote $\mc{A}_{N}^{2, \text{equiv}}(\mc{H})$ the $L^2$ completion of $\mc{A}_{N}^{\text{equiv}}(\mc{H})$ with respect to the expectation functional $\mf{E}$.
\end{Notation}

\subsection{The isomorphisms}

\begin{Thm}
\label{Thm:E-map}
Let $q = 1/N$. Define the evaluation map $\mc{E}$ from $\bigoplus_{n=0}^\infty \mf{C}[S_0(n)] \otimes \mc{H}_{\mf{R}}^{\otimes n}$ to the algebra $\mc{A}_N^{\text{equiv}}(\mc{H})$ by the linear extension of
\[
\mc{E}[\T{\alpha \otimes (h_1 \otimes \ldots \otimes h_n)}] = \Tr_\alpha(X(h_1), \ldots, X(h_n)).
\]
\begin{enumerate}
\item
This map factors through to $\mc{TP}(\mc{H}_{\mf{R}})$.
\item
$\mc{E}$ is a star-homomorphism of algebras.
\item
$\mf{E} \circ \mc{E} = \phi$. It follows that $\mc{E}$ extends to a homomorphism from $\overline{\mc{TP}}(\mc{H}_{\mf{R}})$, and to an isometry from $\mc{F}_{1/N}(\mc{H})$ to $\mc{A}_N^{2, \text{equiv}}(\mc{H})$.
\item
$\mc{E}$ intertwines the center-valued trace and the normalized trace:
\[
\mc{E}[C(\W{\alpha \otimes_s (h_1 \otimes \ldots \otimes h_n)})]
= \frac{1}{N} \Tr{\mc{E}[\W{\alpha \otimes_s (h_1 \otimes \ldots \otimes h_n)}]} I.
\]
\item
For $A \in \overline{\mc{TP}}_{1/N}(\mc{H}_{\mf{R}})$,
\[
\Exp{\mc{E}[A] \ |\ X(h) : h \in \mc{H}'_{\mf{R}}}
= \mc{E}[\state{A \ |\ \mc{H}'}],
\]
so in particular,
\[
\Exp{\Tr_\alpha(X(h_1), \ldots, X(h_n)) \ |\ X(h) : h \in \mc{H}_{\mf{R}}'}
= \mc{E}[\state{\T{\alpha \otimes_s F} \ |\ \mc{H}'}].
\]
\end{enumerate}
\end{Thm}

\begin{proof}
(a) follows from
\[
\Tr_{\sigma \alpha \sigma^{-1}}(X(h_{\sigma^{-1}(1)}), \ldots, X(h_{\sigma^{-1}(n)})) = \Tr_\alpha(X(h_1), \ldots, X(h_n)),
\]
and (b) from
\[
\Tr_\alpha(X(h_1), \ldots, X(h_n)) \Tr_\beta(X(h_{n+1}), \ldots, X(h_{n+k})) = \Tr_{\alpha \cup \beta}(X(h_1), \ldots, X(h_{n+k}))
\]
and
\[
\Tr_\alpha(X(h_1), \ldots, X(h_n))^\ast = \Tr_{\alpha^{-1}} (X({h}_1), \ldots, X({h}_n)).
\]
(c) follows by comparing Theorem~\ref{Thm:State}(e) and Proposition~\ref{Prop:GUE-moments}. (d) is follows from Corollary~\ref{Cor-Center-trace}.

For part (e), for $h_1, \ldots, h_k \in \mc{H}'_{\mf{R}}$, using earlier parts and properties of conditional expectations,
\[
\begin{split}
& \Exp{\Exp{\mc{E}[A] \ |\ X(h) : h \in \mc{H}'_{\mf{R}}} \Tr_\beta(X(h_1), \ldots, X(h_k))} \\
&\quad = \Exp{\mc{E}[A] \Tr_\beta(X(h_1), \ldots, X(h_k))} \\
&\quad = \Exp{\mc{E}[A] \mc{E}[\T{\beta \otimes (h_1 \otimes \ldots \otimes h_n)}]} \\
&\quad = \state{A \T{\beta \otimes (h_1 \otimes \ldots \otimes h_n)}} \\
&\quad = \state{\state{A \ |\ \mc{H}'} \T{\beta \otimes (h_1 \otimes \ldots \otimes h_n)}} \\
&\quad = \Exp{\mc{E}[\state{A \ |\ \mc{H}'}] \Tr_\beta(X(h_1), \ldots, X(h_k))}.
\end{split}
\]
Since $\mf{E}$ is faithful on $\mc{A}_N^{\text{equiv}}(\mc{H})$, the result follows.
\end{proof}

Numerous corollaries follow by combining Theorem~\ref{Thm:E-map} with results from earlier in the article. We only list a few of them explicitly. Others include Proposition~\ref{Prop:Chaos-II} (chaos decomposition in the univariate case), \ref{Prop:T-I} (expansion of the Hermite polynomial and stochastic integral), \ref{Prop:Product-I} (product formula for Hermite polynomials and stochastic integrals), \ref{Prop:Linearization} (linearization coefficients), and Corollary~\ref{Cor:Product} (Factorization of the Hermite polynomial for orthogonal arguments).

\begin{Remark}
\label{Remark:Hermite}
Define the $\alpha$-Hermite polynomial
\[
H_\alpha(X(h_1), \ldots, X(h_n)) = \mc{E}[\W{\alpha \otimes (h_1 \otimes \ldots \otimes h_n)}],
\]
considered either as a random matrix or a trace polynomial in formal variables $X(h_1), \ldots, X(h_n)$. To drop the dependence on the Hilbert space, we would only consider arguments belonging to some orthonormal basis. For example, for $\alpha = (012)(34)$,
\begin{align*}
H_\alpha(x_1, x_1, x_1, x_1) & = x_1^2 \Tr[x_1^2] - \Tr[x_1^2] - (N + 4\frac{1}{N}) x_1^2 + (N + 2 \frac{1}{N}), \\
H_\alpha(x_1, x_1, x_2, x_2) & = x_1^2 \Tr[x_2^2] - \Tr[x_2^2] - N x_1^2 + N
= (x_1^2 - 1) (\Tr[x_2^2] - N), \\
H_\alpha(x_1, x_2, x_1, x_2) & = x_1 x_2 \Tr[x_1 x_2] - \frac{1}{N} x_1^2 - \frac{1}{N} x_2^2 + \frac{1}{N}, \\
H_\alpha(x_1, x_2, x_3, x_4) & = x_1 x_2 \Tr[x_3 x_4].
\end{align*}
\end{Remark}

\begin{Cor}
(Compare with Proposition~\ref{Prop:CE}.) Let $\mc{H} = L^2(\mf{R}_+, dx)$. Then for each $\alpha$, $H_\alpha(X(\chf{[0,t]}))$ is a martingale.
\end{Cor}

For Hermite polynomials of matrix argument, the martingale property was proved in \cite{Lawi} by generating function methods. See Section~\ref{Subsec:Hermite} for the connection.

\begin{Cor}
(Compare with Proposition~\ref{Prop:Single})
\[
\Exp{\Tr_\alpha(X(h)) \ |\ X(h) : h \in \mc{H}'_{\mf{R}}}
= \sum_{\rho \in \Part_{1,2}(n)} \norm{P_{(\mc{H}')^\perp} h}^{2 \abs{\Pair{\rho}}} \Tr_{C_\rho(\alpha)}(X(P_{\mc{H}'} h)).
\]
In particular, for $\mc{H} = L^2(\mf{R}_+, dx)$ and $\mc{H}_s = L^2([0,s], dx)$, for $s \leq t$,
\[
\Exp{\Tr_\alpha(X(\chf{[0,t]})) \ |\ X(h) : h \in \mc{H}_s}
= \sum_{\rho \in \Part_{1,2}(n)} (t-s)^{\abs{\Pair{\rho}}} \Tr_{C_\rho(\alpha)}(X(\chf{[0,s]})).
\]
It follows that $\mc{L}$ is the generator of the process $\set{X(t) = X(\chf{[0,t]}) : t \geq 0}$, in the sense that for a formal univariate trace polynomial $\tr_\alpha(x)$,
\[
\left.\frac{d}{dt}\right|_{t=s} \Exp{\tr_\alpha(X(t)) \ |\ X(h) : h \in \mc{H}_s}
= \tr_{\mc{L} \alpha}(X(s)).
\]
\end{Cor}

\begin{Remark}
In the case $\mc{H}_{\mf{R}} = L^2(\mf{R}_+, dx)$ and $F \in L^2(\mf{R}_+^n, dx^{\otimes n})$, we may identify $\mc{E}[\W{\eta \otimes_s F}]$ with a stochastic integral
\[
\int F(t_1, \ldots, t_n) \,\Tr_\eta[dX(t_1), \ldots, dX(t_n)].
\]
Indeed, consider first $F = \chf{J_1} \otimes \ldots \otimes \chf{J_n}$, where all $J_j$ are disjoint. Then from Proposition~\ref{Prop:T-I},
\[
\W{\eta \otimes (\chf{J_1} \otimes \ldots \otimes \chf{J_n})} = \T{\eta \otimes (\chf{J_1} \otimes \ldots \otimes \chf{J_n})}
\]
and so
\[
\mc{E}[\W{\eta \otimes (\chf{J_1} \otimes \ldots \otimes \chf{J_n})}]
 = \Tr_\eta[X(\chf{J_1}), \ldots, X(\chf{J_n})] \\
\]
which we define to be
\[
\int \chf{J_1}(t_1) \ldots \chf{J_n}(t_n) \,\Tr_\eta[dX(t_1), \ldots, dX(t_n)].
\]
A general $F \in L^2(\mf{R}_+^n, dx^{\otimes n})$ can be approximated by linear combinations of such function in the $L^2$ norm, and by Lemma~\ref{Lemma:L2-approximation} we also get the approximation of $\W{\eta \otimes_s F}$.
\end{Remark}

\begin{Cor}[Chaos decomposition IV; compare with Proposition~\ref{Chaos:III}]
Each element $A \in \mc{A}_N^{2, \text{equiv}}(L^2(\mf{R}_+, dx))$ has a unique decomposition
\[
A = \sum_{n=0}^\infty \sum_{\lambda \in \Par(n+1; \leq N)} \sum_{i, j = 1}^{d_\lambda} \int F_{ij}^{\lambda}(t_1, \ldots, t_n) \,\Tr_{\mc{W}(E_{ij}^\lambda)}[dX(t_1), \ldots, dX(t_n)],
\]
where $F_{ij}^\lambda \in L^2(\Delta(\mf{R}_+^n), dx^{\otimes n})$ and
\[
\sum_{n=0}^\infty \sum_{\lambda \in \Par(n+1; \leq N)} n_\lambda \sum_{i, j = 1}^{d_\lambda} \norm{F_{i j}^\lambda}^2 < \infty
\]
for $n_\lambda = \frac{\abs{SS_N(\lambda)}}{N^{n+1}}$.
\end{Cor}

Finally, we derive an explicit formula for the entries of the matrix $H_\alpha(X(h_1), \ldots, X(h_n))$.

\begin{Lemma}
\label{Lemma:Equivariant}
Let $\mc{H}_{\mf{R}}$ be a real Hilbert space, $\mc{K}_{\mf{R}} = M_N^{sa}(\mf{C}) \otimes \mc{H}_{\mf{R}}$, and $\mc{K} = M_N(\mf{C}) \otimes \mc{H}_{\mf{R}}$ its complexification. Let $Y \in \mc{A}^{\mathrm{equiv}}_{N}(\mc{K})$, with the inner product induced by the (Gaussian) expectation. For a polynomial variable $X$ in the Gaussian Hilbert space indexed by $\mc{H}_{\mf{R}}$, recall the notation $\Wick{X}$ from Section~\ref{Section:Gaussian-Hilbert}. Extend this operation to $Y$ entry-wise: 
\[
(\Wick{Y})_{ij} = \Wick{(Y_{ij})}.
\]
Then $\Wick{Y}$ is also equivariant.
\end{Lemma}

\begin{proof}
It suffices to prove the result when each $Y_{ij}$ is a homogeneous polynomial of degree $n$. Let $Y = P(X(h): h \in \mc{K})$ and $\Wick{Y} = Q(X(h): h \in \mc{K})$. Let $U \in U_N(\mf{C})$. Since each $:(Y_{ij}):$ is orthogonal to $\mc{P}_{n-1}(\mc{K})$, for any such polynomial $R$,
\[
\Exp{(U^\ast Q(U X(h) U^\ast : h \in \mc{K}) U)_{k \ell} R(X(h) : h \in \mc{K})}
= \sum_{i, j = 1}^N \bar{U}_{ik} U_{j \ell} \Exp{\Wick{(Y_{ij})} R(U^\ast X(h) U : h \in \mc{K})} = 0.
\]
Since $Y$ is equivariant,
\[
\begin{split}
U^\ast Q(U X(h) U^\ast : h \in \mc{K}) U
& = U^\ast \Bigl(P(U X(h) U^\ast : h \in \mc{K}) + (\text{terms of degree } < n) \Bigr) U \\
& = Y + (\text{terms of degree } < n).
\end{split}
\]
By the uniqueness of entry-wise projection, it follows that $U^\ast Q(U X(h) U^\ast : h \in \mc{K}) U  = \Wick{Y}$, i.e. it is equivariant.
\end{proof}

\begin{Thm}
\label{Thm:Entrywise}
Using the notation from Section~\ref{Section:Gaussian-Hilbert} and Remark~\ref{Remark:Tensor-Hilbert-space},
\[
H_\alpha(X(h_1), \ldots, X(h_n))_{\ell r}
= \sum_{\substack{\mb{u} \in [N]^{[0,n]} \\ u(0) = r, \\ u(\alpha(0)) = \ell}} \W{\otimes_{i=1}^n (E_{u(i), u(\alpha(i))} \otimes h_i)}.
\]
That is, matrix entries of the Hermite polynomial over $\mc{H}_{\mf{R}}$ are the Hermite polynomials over $M_N(\mf{C}) \otimes \mc{H}_{\mf{R}}$.
\end{Thm}

\begin{proof}
It suffices to show that
\[
\mc{E}[\W{\alpha \otimes (h_1 \otimes \ldots \otimes h_n)}]_{\ell r}
= \Wick{(\Tr_\alpha[X(h_1), \ldots, X(h_n)])_{\ell r}}.
\]
The matrix with entries $\Wick{(\Tr_\alpha[X(h_1), \ldots, X(h_n)])_{\ell r}}$ has leading term $\Tr_\alpha[X(h_1), \ldots, X(h_n)]$, polynomial entries, and (by Lemma~\ref{Lemma:Equivariant}) is equivariant. Therefore it is of the form
\[
\mc{E}[\W{\alpha \otimes (h_1 \otimes \ldots \otimes h_n)} + R]
\]
for some $R \in \overline{\mc{TP}}_{n-1}(\mc{H}_{\mf{R}})$ (see Proposition~\ref{Prop:Projection}). Since each entry of this matrix is orthogonal to $\mc{P}_{n-1}(\mc{H}_{\mf{R}})$, while $\W{\alpha \otimes (h_1 \otimes \ldots \otimes h_n)}$ is orthogonal to $\overline{\mc{TP}}_{n-1}(\mc{H}_{\mf{R}})$, it follows that $\Exp{\mc{E}[R]^\ast \mc{E}[R]} = 0$, and so $\mc{E}[R] = 0$.
\end{proof}

\begin{Remark}
$\mc{E}$ also intertwines the contraction $C_\pi$ with the ordinary entry-wise contraction on Hermite polynomials, as can be seen directly or using Lemma~\ref{Lemma:Orthogonal-contractions}.
\end{Remark}

\begin{Remark}
In \cite{Bia97b}, Biane defined the matricial Segal-Bargmann transform. We outline the construction, and send the reader to the original article for details. The ordinary Segal-Bargmann transform is defined in Section~\ref{Section:Gaussian-Hilbert}. Apply that construction with $\mc{H}_{\mf{R}} = M_N^{sa}(\mf{C})$. Upgrade the Segal-Bargmann transform to its matricial version by applying it entry-by-entry. Then it follows from Theorem~\ref{Thm:Entrywise} that, as in Lemma 7 and Theorem 7 of \cite{Bia97b},
\[
\mc{S}(H_\alpha(X(h_1), \ldots, X(h_n))) = \Tr_\alpha[Z(h_1), \ldots, Z(h_n)].
\]
\end{Remark}

\subsubsection{Hermite polynomials of matrix argument.}
\label{Subsec:Hermite}

\begin{Remark}
If $\alpha \in S(n)$ rather than $S_0(n)$, in the case of a single variable, $\Tr_\alpha[X]$ depends only on the conjugacy class of $\alpha$, in other words on the number partition $\lambda \in \Par(n)$. Moreover
\[
\Tr_\lambda[X] = p_\lambda(x_1, \ldots, x_N),
\]
where $\set{x_1, \ldots, x_N}$ are (random) eigenvalues of $X$ and $p_\lambda$ is the power sum symmetric polynomial. For $X = X(h)$, we also get non-homogeneous symmetric polynomials
\begin{equation}
\label{Eq:Hermite-matrix}
h_\lambda(x_1, \ldots, x_N) = \mc{E}[\W{\lambda \otimes h^{\otimes n}}].
\end{equation}
With respect to the inner product induced from $\mc{F}_{1/N}(\mf{C})$, these polynomials are orthogonal for different $n$ but not necessarily for different $\lambda \in \Par(n)$. We now recall a different and more familiar basis of polynomials which are fully orthogonal with respect to this inner product.
\end{Remark}

\begin{Defn}
Fix $N \in \mf{N}$, and denote
\[
D^\ast = \sum_{i=1}^N \frac{\partial^2}{\partial x_i^2} + \sum_{i \neq j} \frac{1}{x_i - x_j} \left( \frac{\partial}{\partial x_i} - \frac{\partial}{\partial x_j} \right)
\]
and
\[
E^\ast = \sum_{i=1}^N x_i \frac{\partial}{\partial x_i}.
\]
For $\lambda \in \Par(n)$, the Hermite polynomial of matrix argument (for $\beta = 2$) is the symmetric polynomial in $\set{x_1, \ldots, x_N}$ with leading term $\frac{\abs{\lambda}!}{c_\lambda} s_\lambda$ which is an eigenfunction of the operator $D^\ast - E^\ast$ with eigenvalue $-n$ (note a misprint in \cite{Dumitriu-MOPS}). Here (see Corollary~7.1.7.4 in \cite{Stanley-volume-2})
\[
s_\lambda = \frac{1}{n!}\sum_{\nu \in \Par(n)} \frac{n!}{z_\nu} \chi^\lambda(\nu) p_\nu = \frac{1}{n!} \sum_{\alpha \in S(n)} \chi^\lambda(\alpha) p_\alpha
\]
is the Schur polynomial and $c_\lambda = \prod (\lambda_i + \lambda_j' - i - j + 1)$ is the hook length. See \cite{Baker-Forrester-CS-model,Dumitriu-MOPS,Forrester-book} for more details.
\end{Defn}

\begin{Prop}
For $\eta \in \mf{C}[S(n)]$,
\[
\norm{h}^2 D^\ast \mc{E}[\T{\eta \otimes h^{\otimes n}}]
= 2 N \mc{E}[\T{\mc{L}(\eta \otimes h^{\otimes n}}].
\]
and
\[
E^\ast \mc{E}[\T{\eta \otimes h^{\otimes n}}]
= \mc{E}[ E\T{\eta \otimes h^{\otimes n}}].
\]
\end{Prop}

\begin{proof}
By linearity, it suffices to prove the result for $\eta = \alpha \in S(n)$ a single permutation. Suppose the cycle structure of $\alpha$ is given by the partition $\mu = 1^{k_1} 2^{k_2} \ldots m^{k_m}$. The second relation follows from $E^\ast p_\mu = n p_\mu$. Recall also that for a permutation $\alpha$, the Laplacian $\mc{L}(\alpha)$ is the sum of contractions of $\alpha$ with respect to all transpositions. Thus the first relation reduces to showing that
\[
D^\ast p_\mu(x_1, \ldots, x_N) = 2 N \sum_{\substack{\tau \in S(n) \\ \text{ a transposition}}} p_{\widetilde{C_\tau(\alpha)}}(x_1, \ldots, x_N),
\]
where $\widetilde{C_\tau(\alpha)}$ is the cycle type partition of $C_\tau(\alpha)$. For the left-hand side, a (long) computation (see \cite{Ans-Gai-Han-He-Mehl} for a similar computation in the unitary case) shows that
\begin{multline*}
D^\ast(p_\mu)
= 2 p_\mu \Biggl[ k_2 \frac{N^2}{p_2} + N \sum_{l=3}^{m} k_l l \frac{p_{l-2}}{p_l}
+ \sum_{l=2}^m \frac{k_l(k_l-1)l^2}{2} \frac{p_{2l-2}}{p_l^{2}} + \frac{k_1 (k_1 - 1)}{2} \frac{N}{p_1^2} \\
+ \sum_{2 \leq l_1 < l_2\leq m}^{ } k_{l_1}k_{l_2}l_1l_2 \frac{p_{l_1+l_2-2}} {p_{l_1}p_{l_2}}
+ \sum_{l = 2}^m k_1 k_l l \frac{p_{l-1}}{p_1 p_l}
+ \sum_{l=3}^{m} \frac{k_l l}{2} \sum_{a=1}^{l-3} \frac{p_ap_{l-2-a}}{p_l}\Biggr]
\end{multline*}
For the right-hand side, for a permutation $\alpha$ of cycle type $\mu = 1^{k_1} 2^{k_2} \ldots m^{k_m}$ and $q = 1/N$, a contraction by a transposition
\begin{itemize}
\item Eliminates a cycle of length $2$: in $k_2$ cases, with weight $N$.
\item Turns a cycle of length $l > 2$ into a cycle of length $l-2$: in $k_l l$ cases, with weight $1$.
\item Glues together two cycles of equal length $l \geq 2$: in $k_l (k_l - 1) l^2/2$ cases, with weight $N^{-1}$.
\item Eliminates two cycles of length $1$: in $k_1 (k_1 - 1)/2$ cases, with weight $1$.
\item Glues together two cycles of different lengths $2 \leq l_1 < l_2$: in $k_{l_1} l_{l_2} l_1 l_2$ cases, with weight $N^{-1}$.
\item Eliminates a cycle of length $1$ and turns a cycle of length $l$ into a cycle of length $l-1$: in $k_1 k_2 l$ cases, with weight $N^{-1}$.
\item Splits a cycle of length $l$ into those of different lengths $a < l-2-a$: in $k_l l$ cases, with weight $N^{-1}$.
\item Splits a cycle of length $l$ into those of equal lengths $(l-2)/2$: in $k_l l/2$ cases, with weight $N^{-1}$.
\end{itemize}

The expressions match up term-by-term up to the factor of $2N$. Applying the Laplacian to the tensor component gives an extra factor of $\norm{h}^2$.
\end{proof}

\begin{Cor}
Let $\norm{h}^2 = N$. Then $\mc{E}[\W{\chi^\lambda \otimes h^{\otimes n}}]$ is a multiple of the Hermite polynomial of matrix argument.
\end{Cor}

\begin{proof}
Since their leading terms differ by a factor of $\frac{\abs{\lambda}!}{c_\lambda n!}$, it suffices to verify that $\mc{E}[\W{\chi^\lambda \otimes h^{\otimes n}}]$ is an eigenfunction of $D^\ast - E^\ast$ with eigenvalue $-n$. Indeed, for $\norm{h}^2 = N$,
\[
(D^\ast - E^\ast) \mc{E}[\T{\eta \otimes h^{\otimes n}}]
= \mc{E}[(2 \mc{L} - E) \T{\eta \otimes h^{\otimes n}}].
\]
So using Proposition~\ref{Prop:DE},
\[
(D^\ast - E^\ast) \mc{E}[\W{\eta \otimes h^{\otimes n}}]
= \mc{E}[(2 \mc{L} - E) \W{\eta \otimes h^{\otimes n}}]
= - n \mc{E}[\W{\eta \otimes h^{\otimes n}}]. \qedhere
\]
\end{proof}

\subsubsection{The case of $q=0$}
\label{Subsec:q=0}

The scaling used throughout most of the article (corresponding to the un-normalized trace) gives well-defined inner products and Fock space structure for $q=0$, see equation~\eqref{Eq:q=0}. However the contractions, and so the operators $\T{\alpha \otimes F}$, may not be defined. In this section we consider a different scaling, corresponding to the normalized trace. Under this normalization, we get well-defined contractions, but a more degenerate Fock space structure.

\begin{Defn}
In this section, we will denote $\alpha \otimes_s F$ by $\tilde{W}(\alpha \otimes_s F)$. For $\ell = \abs{\Pair{\pi}} = \abs{\pi}$, define the normalized contractions
\[
\begin{split}
\tilde{C}_\pi (\alpha)
& = q^{\cyc_0(\alpha) - \cyc_0((\pi \alpha)|_{\supp{\pi}^c})} C_\pi(\alpha) \\
& = q^{\cyc_0(\alpha) - \cyc_0(\pi \alpha) + \ell}
{P^{[0, n] \setminus \supp{\pi}}_{[0, n-2 \ell]} (\pi \alpha)|_{\supp{\pi}^c}},
\end{split}
\]
and the operators
\begin{equation}
\label{Eq:Normalized-relation}
\tilde{T}(\alpha \otimes_s F)
= \sum_{\pi \in \Part_{1,2}(n)} \tilde{W}(\tilde{C}_\pi (\alpha) \otimes_s C_\pi(F)),
\end{equation}
so that
\begin{equation}
\label{Eq:Normalized-W}
\tilde{W}(\alpha \otimes_s F)
= \sum_{\pi \in \Part_{1,2}(n)} (-1)^{\abs{\pi}} \tilde{T}(\tilde{C}_\pi (\alpha) \otimes_s C_\pi(F))
\end{equation}
and
\[
\tilde{T}(\alpha \otimes_s F) \tilde{T}(\beta \otimes_s G) = \tilde{T}((\alpha \cup \beta) \otimes_s (F \otimes G)).
\]
Denote
\[
\widetilde{\mc{TP}}(\mc{H}_{\mf{R}}) = \Span{\set{\T{\eta \otimes_s F} : n \geq 0, \eta \in \mf{C}[S_0(n)], F \in \mc{H}_{\mf{R}}^{\odot n}}},
\]
and define the linear functional $\tilde{\varphi}_0$ on it by
\[
\statez{\tilde{W}(\alpha \otimes_s F)} = 0, \quad \statez{\tilde{W}({(0)})} = 1.
\]
Denote $\tilde{\mc{F}}_0(\mc{H}) = L^2(\widetilde{\mc{TP}}(\mc{H}_{\mf{R}}), \tilde{\varphi}_0)$ the corresponding GNS Hilbert space. 
\end{Defn}

\begin{Remark}
If we assume that $\tilde{T}(\alpha \otimes F) = q^{\cyc_0(\alpha)} \T{\alpha \otimes F}$, then for $q=0$, such a multiple is zero unless $\alpha$ is a single cycle (containing $0$). The definition above is a little more flexible.
\end{Remark}

\begin{Example}
With this normalization, the examples in Remark~\ref{Remark:Hermite}, for $\alpha = (012)(34)$, become
\begin{align*}
\tilde{H}_\alpha(x_1, x_1, x_1, x_1) & = x_1^2 \tr[x_1^2] - \tr[x_1^2] - (1 + 4\frac{1}{N^2}) x_1^2 + (1 + 2 \frac{1}{N^2}), \\
\tilde{H}_\alpha(x_1, x_1, x_2, x_2) & = x_1^2 \tr[x_2^2] - \tr[x_2^2] - x_1^2 + 1
= (x_1^2 - 1) (\tr[x_2^2] - 1), \\
\tilde{H}_\alpha(x_1, x_2, x_1, x_2) & = x_1 x_2 \tr[x_1 x_2] - \frac{1}{N^2} x_1^2 - \frac{1}{N^2} x_2^2 + \frac{1}{N^2}, \\
\tilde{H}_\alpha(x_1, x_2, x_3, x_4) & = x_1 x_2 \tr[x_3 x_4].
\end{align*}
\end{Example}

\begin{Remark}
\label{Remark:Standard-form}
Let $\lambda$ be an interval partition of $[0,n]$. As discussed in Section~\ref{Sec:Centralizer}, each permutation $\alpha \in S_0(n)$ is conjugate, under the action of $S(n)$, to a permutation with cycle structure $\lambda$ in which the elements in each cycle, as well as the cycles, appear in increasing order. Such a permutation is not unique. In the results below, we will only consider permutations of this type; the results are easily extended to general $\alpha \in S_0(n)$, but the notation gets heavier.
\end{Remark}

\begin{Remark}
Let $\mc{F}_f(\mc{H}) = \mf{C} \Omega \oplus \bigoplus_{n=1}^\infty \overline{\mc{H}^{\otimes n}}$ be the full Fock space of $\mc{H}$. For $h \in \mc{H}_{\mf{R}}$, denote by $S(h)$ the semicircular element corresponding to $h$ in its standard representation on $\mc{F}_f(\mc{H})$. Denote by $\Gamma_f(\mc{H}_{\mf{R}})$ the algebra generated by $\set{X(h) : h \in \mc{H}_{\mf{R}}}$, and by $\Phi$ the vacuum state on it. Denote by $U(h_1 \otimes \ldots \otimes h_n)$ the multivariate Chebyshev polynomials (of the second kind), which are the elements of $\Gamma_f(\mc{H}_{\mf{R}})$ that satisfy $U(h_1 \otimes \ldots \otimes h_n) \Omega = h_1 \otimes \ldots \otimes h_n$. They are also determined by the recursion
\begin{equation}
\label{Eq:Chebyshev}
U(f \otimes h_1 \otimes \ldots \otimes h_n) = S(f) U(h_1 \otimes \ldots \otimes h_n) - \ip{f}{h_1} U(h_2 \otimes \ldots \otimes h_n).
\end{equation}
More generally, for $F \in \overline{\mc{H}_{\mf{R}}^{\otimes n}}$, $U(F)$ is the unique element of the weak closure of $\Gamma_f(\mc{H}_{\mf{R}})$ satisfying $U(F) \Omega = F$.

See Lecture 7 of \cite{Nica-Speicher-book} and Section 5 of \cite{BiaSpeBrownian} for additional background on the full Fock space, Chebyshev polynomials, and free analysis.
\end{Remark}

\begin{Prop}
Let $q=0$, and $\lambda \in S_0(n)$ as in Remark~\ref{Remark:Standard-form}. Denote $V_0$ the block of $\lambda$ that contains $0$.
\begin{enumerate}
\item
$\tilde{C}_\pi ({\lambda}) = 0$ unless $\pi$ is non-crossing and $\pi \leq {\lambda}$ as a cycle partition, in which case
\[
\tilde{C}_\pi ({\lambda})
= {P^{[0, n] \setminus \supp{\pi}}_{[0, n-2 \ell]} (\pi {\lambda})|_{\supp{\pi}^c}}.
\]
\item Define the map $\mc{E}_0 : \widetilde{\mc{TP}}(\mc{H}_{\mf{R}}) \rightarrow \Gamma_f(\mc{H}_{\mf{R}})$ by
\begin{equation}
\label{Eq:T-tilde}
\mc{E}_0[\tilde{T}(\lambda \otimes (h_1 \otimes \ldots \otimes h_n))] = \Phi_\lambda[S(h_1), \ldots, S(h_n)] = S(h_1) \ldots S(h_{\abs{V_0}}) \prod_{\substack{V \in \lambda \\ V \neq V_0}} \sum_{\pi \in \mc{NC}_2(V)}  C_\pi(\bigotimes_{i \in V} h_i),
\end{equation}
where the notation $\Phi_\lambda$ is defined at the beginning of Section~\ref{Sec:GUE}. Then $\mc{E}_0$ extends to a homomorphism, and to an isometric isomorphism from $\tilde{\mc{F}}_0(\mc{H})$ to $L^2(\Gamma_f(\mc{H}_{\mf{R}}), \Phi) \simeq \mc{F}_f(\mc{H})$.
\item If $\lambda$ is not a single cycle (containing $0$), then $\mc{E}_0[\tilde{W}({\lambda} \otimes F)] = 0$. Otherwise,
\[
\mc{E}_0[\tilde{W}({(01 \ldots n)} \otimes (h_1 \otimes \ldots \otimes h_n))] = U(h_1 \otimes \ldots \otimes h_n)
\]
and more generally
\[
\mc{E}_0[\tilde{W}({(01 \ldots n)} \otimes F)] = U(F).
\]
\end{enumerate}
\end{Prop}

\begin{proof}
Note first that for $\ell = \abs{\pi}$,
\[
\cyc_0(\alpha) - \cyc_0(\pi \alpha) + \ell = - \abs{\alpha} + \abs{\pi^{-1} \alpha} + \abs{\pi} \geq 0.
\]
Moreover, for $\alpha = \lambda$ as in Remark~\ref{Remark:Standard-form}, this is equal to $0$ if and only if $\pi \leq \lambda$, and on each block of $\lambda$, $\pi$ is non-crossing (see Lecture 23 of \cite{Nica-Speicher-book}). Therefore equation~\eqref{Eq:Normalized-relation} reduces to 
\begin{equation}
\label{Eq:T-NC}
\tilde{T}({\lambda} \otimes F)
= \sum_{\substack{\pi \in \mc{NC}_{1,2}(n) \\ \pi \leq {\lambda}}} \tilde{W}(\tilde{C}_\pi ({\lambda}) \otimes C_\pi(F)),
\end{equation}
and so for $F = h_1 \otimes \ldots \otimes h_n$,
\[
\state{\tilde{T}({\lambda} \otimes F)}
= \sum_{\substack{\pi \in \mc{NC}_{2}(n) \\ \pi \leq {\lambda}}} C_\pi(F)
= \Phi[\Phi_\lambda[S(h_1), \ldots, S(h_n)]],
\]
where the second equality follows from the properties of semicircular elements. The homomorphism property follows as before, so $\mc{E}_0$ is an isometric isomorphism. In particular, to prove the second equality in~\eqref{Eq:T-tilde}, it suffices to note that if $(0)$ is a cycle in $\lambda$, then
\begin{equation}
\label{Eq:T-scalar}
\mc{E}_0[\tilde{T}({\lambda} \otimes F)] = \Phi_\lambda[S(h_1), \ldots, S(h_n)] = \sum_{\substack{\pi \in \mc{NC}_{2}(n) \\ \pi \leq {\lambda}}} C_\pi(F)
\end{equation}
is a scalar. Next, for $q=0$, the relation~\eqref{Eq:Normalized-W} reduces to
\[
\tilde{W}(\lambda \otimes F)
= \sum_{\substack{\pi \in \mc{NC}_{1,2}(n) \\ \pi \leq \lambda}} (-1)^{\abs{\pi}} \tilde{T}(\tilde{C}_\pi ({\lambda}) \otimes C_\pi(F)).
\]
Suppose $\lambda$ contains a cycle $V$ that does not contain $0$. Then  the sum above is
\[
\begin{split}
\tilde{W}(\lambda \otimes F)
& = \sum_{\substack{\pi \in \mc{NC}_{1,2}([n] \setminus V) \\ \pi \leq (\lambda|_{V^c})}} (-1)^{\abs{\pi}} \tilde{T}(\tilde{C}_{\pi}({\lambda|_{V^c}}) \otimes C_{\pi}(\bigotimes _{i \in [n] \setminus V} h_i)) \\
&\qquad \cdot \sum_{\tau \in \mc{NC}_{1,2}(V)} (-1)^{\abs{\tau}} \tilde{T}(\tilde{C}_{\tau}({V}) \otimes C_{\tau}( \bigotimes_{i \in V} h_i)).
\end{split}
\]
Since $V$ does not contain $0$, denoting $F_V = \bigotimes_{i \in V} h_i$ and using \eqref{Eq:T-scalar}, $\mc{E}_0$ applied to the second of these sums gives
\[
\sum_{\tau \in \mc{NC}_{1,2}(V)} (-1)^{\abs{\tau}} \sum_{\substack{\sigma \in \mc{NC}_{2}(\abs{V} - 2 \abs{\tau}) \\ \sigma \leq \tilde{C}_\tau(V)}} C_\sigma C_\tau F_V
= \sum_{\rho \in \mc{NC}_2(V)} \sum_{S \subset \Pair{\rho}} (-1)^{\abs{S}} C_\rho F_V = 0.
\]
Finally,  $\mc{E}_0[\tilde{W}({(01 \ldots n)} \otimes (h_1 \otimes \ldots \otimes h_n))]$ is a polynomial in $\set{S(h_1), \ldots, S(h_n)}$, with the leading term $S(h_1) \ldots S(h_n)$, and
\[
\begin{split}
S(f) \mc{E}_0[\tilde{W}((01\ldots n) \otimes (h_1 \otimes \ldots \otimes h_n))]
& = \mc{E}_0[\tilde{W}((01\ldots n (n+1)) \otimes (f \otimes h_1 \otimes \ldots \otimes h_n))] \\
&\quad - \sum_{k=1}^n \ip{f}{h_k} \mc{E}_0[\tilde{W}((0 k \ldots (n-1)) (1 \ldots (k-1)) \otimes (f \otimes h_1 \otimes \ldots \otimes \hat{h}_k \otimes \ldots \otimes h_n))] \\
& = \mc{E}_0[\tilde{W}((01\ldots n (n+1)) \otimes (f \otimes h_1 \otimes \ldots \otimes h_n))] \\
&\quad - \ip{f}{h_1} \mc{E}_0[\tilde{W}((0 1 \ldots (n-1)) \otimes (f \otimes h_2 \otimes \ldots \otimes h_n))]
\end{split}
\]
since the remaining terms in the sum are zero. Comparing with the recursion~\eqref{Eq:Chebyshev}, we conclude that these are indeed the Chebyshev polynomials.
\end{proof}

\begin{acknowledgments}
The first author has discussed various aspects of this project with a number of people. He is especially grateful to Todd Kemp, Andu Nica, JM Landsberg, Marek Bo{\.z}ejko, and Jurij Vol\v{c}i\v{c}. The authors also thank two anonymous referees and Elena Boguslavskaya for an incredibly careful reading of the manuscript and numerous comments. Their suggestions resulted in substantial improvements, including the Gaussian Hilbert spaces approach and Theorem~\ref{Thm:Entrywise}, as well as a better (in particular, correct) proof of Theorem~\ref{Thm:Center}. Both authors were supported by the Simons Foundation Collaboration Grant for Mathematicians \#527486 .
\end{acknowledgments}

\section*{Author declarations}

\subsection*{Conflict of interest}

The authors have no conflicts to disclose.

\subsection*{Author Contributions}

\textbf{Michael Anshelevich:} Conceptualization (equal), funding acquisition (lead), writing - original draft preparation (lead), writing - review \& editing (lead) . \textbf{David Buzinski:} Conceptualization (equal), writing - original draft preparation (supporting).

\section*{Data Availability Statement}

Data sharing not applicable to this article as no datasets were generated or analyzed during the current study.


\begin{thebibliography}{10}

\bibitem{Ans-Gai-Han-He-Mehl}
Michael Anshelevich, Matthew Gaikema, Madeline Hansalik, Songyu He, and Nathan
  Mehlhop, \emph{Expansion of permutations as products of transpositions},
  arXiv:1702.06093 [math.CO], 2017.

\bibitem{Baker-Forrester-CS-model}
T.~H. Baker and P.~J. Forrester, \emph{The {C}alogero-{S}utherland model and
  generalized classical polynomials}, Comm. Math. Phys. \textbf{188} (1997),
  no.~1, 175--216. \MR{1471336 (99c:33012)}

\bibitem{Bia97b}
Philippe Biane, \emph{Segal-{B}argmann transform, functional calculus on matrix
  spaces and the theory of semi-circular and circular systems}, J. Funct. Anal.
  \textbf{144} (1997), no.~1, 232--286.

\bibitem{Biane-Approx-factorization-characters}
\bysame, \emph{Approximate factorization and concentration for characters of
  symmetric groups}, Internat. Math. Res. Notices (2001), no.~4, 179--192.
  \MR{1813797}

\bibitem{BiaSpeBrownian}
Philippe Biane and Roland Speicher, \emph{Stochastic calculus with respect to
  free {B}rownian motion and analysis on {W}igner space}, Probab. Theory
  Related Fields \textbf{112} (1998), no.~3, 373--409. \MR{99i:60108}

\bibitem{Bozejko-Guta}
Marek Bo\.{z}ejko and M\u{a}d\u{a}lin Gu\c{t}\u{a}, \emph{Functors of white
  noise associated to characters of the infinite symmetric group}, Comm. Math.
  Phys. \textbf{229} (2002), no.~2, 209--227. \MR{1923173}

\bibitem{Structure-symmetric}
Murray~R. Bremner, Sara Madariaga, and Luiz~A. Peresi, \emph{Structure theory
  for the group algebra of the symmetric group, with applications to polynomial
  identities for the octonions}, Comment. Math. Univ. Carolin. \textbf{57}
  (2016), no.~4, 413--452. \MR{3583300}

\bibitem{Cebron-Free-convolution}
Guillaume C{\'e}bron, \emph{Free convolution operators and free {H}all
  transform}, J. Funct. Anal. \textbf{265} (2013), no.~11, 2645--2708.
  \MR{3096986}

\bibitem{Corteel-Crossings-permutations}
Sylvie Corteel, \emph{Crossings and alignments of permutations}, Adv. in Appl.
  Math. \textbf{38} (2007), no.~2, 149--163. \MR{2290808}

\bibitem{Dabrowski-Guionnet-Shl-Convex}
Yoann Dabrowski, Alice Guionnet, and Dima Shlyakhtenko, \emph{Free transport
  for convex potentials}, New Zealand J. Math. \textbf{52} (2021), 259--359.
  \MR{4378155}

\bibitem{dSCViennot}
Myriam de~Sainte-Catherine and G{\'e}rard Viennot, \emph{Combinatorial
  interpretation of integrals of products of {H}ermite, {L}aguerre and
  {T}chebycheff polynomials}, Orthogonal polynomials and applications
  (Bar-le-Duc, 1984), Lecture Notes in Math., vol. 1171, Springer, Berlin,
  1985, pp.~120--128. \MR{87g:05007}

\bibitem{Driver-Hall-Kemp}
Bruce~K. Driver, Brian~C. Hall, and Todd Kemp, \emph{The large-{$N$} limit of
  the {S}egal-{B}argmann transform on {$\Bbb{U}_N$}}, J. Funct. Anal.
  \textbf{265} (2013), no.~11, 2585--2644. \MR{3096985}

\bibitem{Dumitriu-MOPS}
Ioana Dumitriu, Alan Edelman, and Gene Shuman, \emph{M{OPS}: multivariate
  orthogonal polynomials (symbolically)}, J. Symbolic Comput. \textbf{42}
  (2007), no.~6, 587--620. \MR{2325917}

\bibitem{Forrester-book}
P.~J. Forrester, \emph{Log-gases and random matrices}, London Mathematical
  Society Monographs Series, vol.~34, Princeton University Press, Princeton,
  NJ, 2010. \MR{2641363 (2011d:82001)}

\bibitem{Gill-Rep}
Graham Gill, \emph{Representation theory of the symmetric group: Basic
  elements}, \\
  http://www.math.toronto.edu/murnaghan/courses/mat445/Symmetric.pdf, 2005.

\bibitem{Gnedin-Gorin-Kerov-Block-characters}
Alexander Gnedin, Vadim Gorin, and Sergei Kerov, \emph{Block characters of the
  symmetric groups}, J. Algebraic Combin. \textbf{38} (2013), no.~1, 79--101.
  \MR{3070121}

\bibitem{Graczyk-Vostrikova}
P.~Graczyk and L.~Vostrikova, \emph{The moments of {W}ishart processes via
  {I}t\^o calculus}, Teor. Veroyatn. Primen. \textbf{51} (2006), no.~4,
  732--751. \MR{2338064}

\bibitem{Guta-Maassen-BM}
M\u{a}d\u{a}lin Gu\c{t}\u{a} and Hans Maassen, \emph{Generalised {B}rownian
  motion and second quantisation}, J. Funct. Anal. \textbf{191} (2002), no.~2,
  241--275. \MR{1911186}

\bibitem{Huber-trace}
Felix Huber, \emph{Positive maps and trace polynomials from the symmetric
  group}, J. Math. Phys. \textbf{62} (2021), no.~2, Paper No. 022203, 27.
  \MR{4221255}

\bibitem{Janson-GHS}
Svante Janson, \emph{Gaussian {H}ilbert spaces}, Cambridge Tracts in
  Mathematics, vol. 129, Cambridge University Press, Cambridge, 1997.
  \MR{1474726}

\bibitem{Jekel-Elementary}
David Jekel, \emph{An elementary approach to free entropy theory for convex
  potentials}, Anal. PDE \textbf{13} (2020), no.~8, 2289--2374. \MR{4201981}

\bibitem{Jekel-Li-Shl-Wasserstein}
David Jekel, Wuchen Li, and Dimitri Shlyakhtenko, \emph{Tracial smooth
  functions of non-commuting variables and the free {W}asserstein manifold},
  Dissertationes Math. \textbf{580} (2022), 150. \MR{4451910}

\bibitem{Kemp-large-N-GL}
Todd Kemp, \emph{The large-{$N$} limits of {B}rownian motions on
  {$\Bbb{GL}_N$}}, Int. Math. Res. Not. IMRN (2016), no.~13, 4012--4057.
  \MR{3544627}

\bibitem{Kemp-Heat-kernel}
\bysame, \emph{Heat kernel empirical laws on {$\Bbb{U}_N$} and {$\Bbb{GL}_N$}},
  J. Theoret. Probab. \textbf{30} (2017), no.~2, 397--451. \MR{3647064}

\bibitem{Kerov-book}
S.~V. Kerov, \emph{Asymptotic representation theory of the symmetric group and
  its applications in analysis}, Translations of Mathematical Monographs, vol.
  219, American Mathematical Society, Providence, RI, 2003, Translated from the
  Russian manuscript by N. V. Tsilevich, With a foreword by A. Vershik and
  comments by G. Olshanski. \MR{1984868}

\bibitem{Klep-Pascoe-Volcic}
Igor Klep, James~Eldred Pascoe, and Jurij Vol\v{c}i\v{c}, \emph{Positive
  univariate trace polynomials}, J. Algebra \textbf{579} (2021), 303--317.
  \MR{4241240}

\bibitem{Klep-Spenko-Free-function-theory}
Igor Klep and {\v S}pela {\v S}penko, \emph{Free function theory through matrix
  invariants}, Canad. J. Math. \textbf{69} (2017), no.~2, 408--433.
  \MR{3612091}

\bibitem{Kostler-Nica-CLT-S-infty}
Claus K\"{o}stler and Alexandru Nica, \emph{A central limit theorem for
  star-generators of {$S _\infty$}, which relates to the law of a {GUE}
  matrix}, J. Theoret. Probab. \textbf{34} (2021), no.~3, 1248--1278.
  \MR{4289885}

\bibitem{Lawi}
Stephan Lawi, \emph{Hermite and {L}aguerre polynomials and matrix-valued
  stochastic processes}, Electron. Commun. Probab. \textbf{13} (2008), 67--84.
  \MR{2386064 (2009a:60090)}

\bibitem{Levy-Schur-Weyl}
Thierry L{\'e}vy, \emph{Schur-{W}eyl duality and the heat kernel measure on the
  unitary group}, Adv. Math. \textbf{218} (2008), no.~2, 537--575. \MR{2407946}

\bibitem{Matsumoto-Novak-primitive}
Sho Matsumoto and Jonathan Novak, \emph{Unitary matrix integrals, primitive
  factorizations, and {J}ucys-{M}urphy elements}, 22nd {I}nternational
  {C}onference on {F}ormal {P}ower {S}eries and {A}lgebraic {C}ombinatorics
  ({FPSAC} 2010), Discrete Math. Theor. Comput. Sci. Proc., AN, Assoc. Discrete
  Math. Theor. Comput. Sci., Nancy, 2010, pp.~403--411. \MR{2673853}

\bibitem{Nica-Speicher-book}
Alexandru Nica and Roland Speicher, \emph{Lectures on the combinatorics of free
  probability}, London Mathematical Society Lecture Note Series, vol. 335,
  Cambridge University Press, Cambridge, 2006. \MR{2266879 (2008k:46198)}

\bibitem{Procesi}
C.~Procesi, \emph{The invariant theory of {$n\times n$} matrices}, Advances in
  Math. \textbf{19} (1976), no.~3, 306--381. \MR{0419491 (54 \#7512)}

\bibitem{Raicu-Generic-tensor}
Claudiu Raicu, \emph{Products of {Y}oung symmetrizers and ideals in the generic
  tensor algebra}, J. Algebraic Combin. \textbf{39} (2014), no.~2, 247--270.
  \MR{3159252}

\bibitem{Rains}
E.~M. Rains, \emph{Combinatorial properties of {B}rownian motion on the compact
  classical groups}, J. Theoret. Probab. \textbf{10} (1997), no.~3, 659--679.
  \MR{1468398}

\bibitem{Razmyslov}
Ju.~P. Razmyslov, \emph{Identities with trace in full matrix algebras over a
  field of characteristic zero}, Izv. Akad. Nauk SSSR Ser. Mat. \textbf{38}
  (1974), 723--756. \MR{0506414 (58 \#22158)}

\bibitem{Stanley-volume-2}
Richard~P. Stanley, \emph{Enumerative combinatorics. {V}ol. 2}, Cambridge
  Studies in Advanced Mathematics, vol.~62, Cambridge University Press,
  Cambridge, 1999, With a foreword by Gian-Carlo Rota and appendix 1 by Sergey
  Fomin. \MR{1676282}

\bibitem{Okounkov-Vershik}
A.~M. Vershik and A.~Yu. Okun\cprime~kov, \emph{A new approach to
  representation theory of symmetric groups. {II}}, Zap. Nauchn. Sem.
  S.-Peterburg. Otdel. Mat. Inst. Steklov. (POMI) \textbf{307} (2004),
  no.~Teor. Predst. Din. Sist. Komb. i Algoritm. Metody. 10, 57--98, 281.
  \MR{2050688}

\end{thebibliography}

\def\cprime{$'$} \def\cprime{$'$}
\providecommand{\bysame}{\leavevmode\hbox to3em{\hrulefill}\thinspace}
\providecommand{\MR}{\relax\ifhmode\unskip\space\fi MR }
\providecommand{\MRhref}[2]{%
  \href{http://www.ams.org/mathscinet-getitem?mr=#1}{#2}
}
\providecommand{\href}[2]{#2}

\end{document}